\documentclass[11pt]{amsart}
\usepackage{amsmath, amssymb, amsthm, verbatim}
\usepackage[scaled]{helvet} % ss
\usepackage[T1]{fontenc}
\usepackage{textcomp}
\normalfont
\usepackage{mathrsfs}

\usepackage[usenames,dvipsnames]{xcolor}

\usepackage[margin=1in]{geometry}
\usepackage[color, final]{showkeys} 		% change "draft" to "final" to remove all these labels along the margin
\definecolor{refkey}{rgb}{0.8,0.8,0.8}
\definecolor{labelkey}{rgb}{0.9,0,0.1}

\usepackage{graphicx}
\usepackage[utf8]{inputenc}
\usepackage[english]{babel}
\usepackage{amssymb}
\usepackage{tikz-cd}
\usepackage{csquotes}

\usepackage{hyperref}
\usepackage[capitalize]{cleveref}
\crefname{ineq}{Ineq.}{inequalities}
\creflabelformat{ineq}{#2{\upshape(#1)}#3} 
\usepackage[normalem]{ulem}
\usepackage{enumerate}
\usepackage{enumitem}
\usepackage{todonotes}
\usepackage{xspace}
\usepackage{tcolorbox}

\hypersetup{
  colorlinks   = true, %Colours links instead of ugly boxes
  urlcolor     = blue, %Colour for external hyperlinks
  linkcolor    = blue, %Colour of internal links
  citecolor   = red %Colour of citations
}

\makeatletter
\newtheorem*{rep@theorem}{\rep@title}
\newcommand{\newreptheorem}[2]{%
	\newenvironment{rep#1}[1]{%
		\def\rep@title{#2 \ref{##1}}%
		\begin{rep@theorem}}%
		{\end{rep@theorem}}}
\makeatother

\newtheorem{theorem}{Theorem}[section]
\newreptheorem{theorem}{Theorem}
\newtheorem{lemma}[theorem]{Lemma}
\newtheorem{proposition}[theorem]{Proposition}
\newtheorem{corollary}[theorem]{Corollary}
\newreptheorem{corollary}{Corollary}

\newtheorem{definition}[theorem]{Definition}
\newtheorem{construction}[theorem]{Construction}
\newtheorem{claim}{Claim}
\newtheorem*{claim*}{Claim}

\theoremstyle{definition}
\newtheorem{example}[theorem]{Example}

\newtheorem{remark}[theorem]{Remark}

\theoremstyle{remark}

\numberwithin{equation}{section}

\newcommand{\mc}{\mathcal}

\newcommand{\mf}{\mathfrak}
\newcommand{\scr}{\mathscr}

\newcommand{\abs}[1]{\left\lvert \vcenter{\hbox{$\displaystyle #1 $}} \right\rvert}
\newcommand{\norm}[1]{\left\lVert \vcenter{\hbox{$\displaystyle #1 $}} \right\rVert}
\newcommand{\inner}[1]{\left\langle #1 \right\rangle}

\newcommand{\wh}{\widehat}

\newcommand{\Isom}{\mathrm{Isom}}

\newcommand{\Z}{\mathbb{Z}}

\newcommand{\Q}{\mathbb{Q}}
\newcommand{\R}{\mathbb{R}}

\newcommand{\SL}{\mathrm{SL}}
\newcommand{\SO}{\mathrm{SO}}

\DeclareMathOperator{\Ad}{Ad}

\DeclareMathOperator{\Hom}{Hom}

\newcommand{\tensor}{\otimes}
\DeclareMathOperator{\im}{Im}
\DeclareMathOperator{\id}{id}
\DeclareMathOperator{\Homeo}{Homeo}
\DeclareMathOperator{\Lipeo}{BiLip}
\DeclareMathOperator{\diam}{diam}
\DeclareMathOperator{\Lip}{Lip}
\DeclareMathOperator{\Nbhd}{Nbhd}
\DeclareMathOperator{\PSL}{PSL}

\DeclareMathOperator{\wcu}{\mathcal{W}^{cu}}
\DeclareMathOperator{\wcs}{\mathcal{W}^{cs}}

\DeclareMathOperator{\Haus}{Haus}

\renewcommand{\bar}{\overline}
\DeclareMathOperator{\pr}{pr}
\newcommand{\at}{\bigg|}
\newcommand{\bdx}{\mathfrak{X}}
\newcommand{\bdy}{\Upsilon}
\newcommand{\flow}{{A'_Q}}

\DeclareMathOperator{\Bc}{\mathcal{B}}
\DeclareMathOperator{\Cc}{\mathcal{C}}
\DeclareMathOperator{\Dc}{\mathcal{D}}
\DeclareMathOperator{\Ec}{\mathcal{E}}
\DeclareMathOperator{\Fc}{\mathcal{F}}

\DeclareMathOperator{\Hc}{\mathcal{H}}

\DeclareMathOperator{\Kc}{\mathcal{K}}

\DeclareMathOperator{\Oc}{\mathcal{O}}

\DeclareMathOperator{\Qc}{\mathcal{Q}}

\DeclareMathOperator{\Sc}{\mathcal{S}}

\DeclareMathOperator{\Uc}{\mathcal{U}}
\DeclareMathOperator{\Wc}{\mathcal{W}}
\DeclareMathOperator{\Vc}{\mathcal{V}}

\DeclareMathOperator{\Hb}{\mathbb{H}}

\DeclareMathOperator{\Nb}{\mathbb{N}}
\DeclareMathOperator{\Pb}{\mathbb{P}}
\DeclareMathOperator{\Rb}{\mathbb{R}}

\DeclareMathOperator{\Zb}{\mathbb{Z}}

\newcommand{\wt}{\widetilde}

\newcommand{\til}{\widetilde}
\newcommand{\grad}{\nabla}

\newcommand{\op}{\operatorname}
\newcommand{\of}{\circ}
\newcommand{\set}[1]{\left\{ #1 \right\}}

\newcommand{\cout}[1]{}

\definecolor{darkcyan}{rgb}{0. 0.65, 0.65}

\newcommand\rest[1]{{\raisebox{-.5ex}{$|$}_{#1}}}

\newcommand{\eps}{\epsilon}

\newtheorem{Structural Stability Theorem}[theorem]{Structural Stability Theorem}

\def\bt{\begin{theorem}}
\def\et{\end{theorem}}
\def\bd{\begin{definition}}
\def\ed{\end{definition}}
\def\bl{\begin{lemma}}
\def\el{\end{lemma}}

\def\be#1\ee{\begin{align}\begin{split} #1 \end{split}\end{align}}
\def\beq#1\eeq{\begin{align*}\begin{split} #1 \end{split}\end{align*}}

\definecolor{grey}{rgb}{0.22, 0.32, 0.51}

\begin{document}

\title{Boundary actions of lattices and $C^0$ local semi-rigidity}

\author{Chris Connell$^{\dagger}$, Mitul Islam$^{\ddagger}$, Thang Nguyen, Ralf Spatzier$^{\ddagger\dagger}$}

\address{Department of Mathematics,
Indiana University, Bloomington, IN 47405}
\email{connell@indiana.edu}

\address{Mathematisches Institut, Heidelberg University, Im Neuenheimer Feld 205, Heidelberg 69120}
\email{mislam@mathi.uni-heidelberg.de}

\address{Department of Mathematics, Florida State University, 
    Tallahassee, FL, 32304}
\email{tqn22@fsu.edu}

\address{Department of Mathematics, University of Michigan, 
    Ann Arbor, MI, 48109.}
\email{spatzier@umich.edu}

\thanks{$^\dagger$ Supported in part by Simons Foundation grant \#965245}
\thanks{$^{\ddagger}$ Supported in part by Emmy Noether Project 427903332 (funded by the DFG)}
\thanks{$^{\ddagger\dagger}$ Supported in part by NSF grant DMS 2003712.}

\subjclass[2010]{Primary 53C24; Secondary 53C20,37D40}

\date{}

\begin{abstract}
  We consider actions of cocompact lattices in semisimple Lie groups of the noncompact type on their boundaries $G/Q$, $Q$ a parabolic group, the so-called standard actions.  We show that perturbations of the standard action in the homeomorphism group continuously factor onto the original standard action by a semi-conjugacy close to the identity. 
This generalizes works by Bowden, Mann, Manning and Weisman in the setting of negative curvature or Gromov hyperbolic groups.
Finally, we also construct perturbations of the action of  lattices on the geodesic boundary which  are not $C^0$ semi-conjugate to the original action. 
\end{abstract}

\maketitle

\tableofcontents
\section{Introduction}

Local rigidity of actions has played an important role in dynamics over the last few decades.  This was most explored for smooth  actions on manifolds with perturbations that are close in the $C^1$ or possibly $C^k$ topology to the original action.

Classically one investigated   structural stable systems.
 These are actions such that a sufficiently small perturbation  in the $C^1$-topology is $C^0$-conjugate or at least $C^0$-orbit-equivalent to the unperturbed action.  There are many such systems in dynamics, in particular diffeomorphisms and flows. Indeed, much attention was paid to characterizing the latter in the 1970s and 1980s. The simplest structurally stable diffeomorphisms and flows come from uniformly hyperbolic dynamics, in particular Anosov systems,  as was proved by Anosov in the 1960s. 
This class includes geodesic flows of closed manifolds $M$ of negative sectional curvature.  There one also has a natural action of the fundamental $\pi _1 M$ on the geodesic compactification $\partial \til{M}$.  For constant curvature $-1$ manifolds $M$, Sullivan proved  local rigidity of these boundary actions using structural stability of the  geodesic flows by a duality argument \cite{Sullivan85}.  This was followed by work of Ghys, Kanai, Katok-Spatzier, and recently  by  Kapovich, Kim and Lee \cite{Ghys93,Kanai, KatokSpatzier97,KapovichKimLee19}.    The latter three in particular consider projective actions of uniform lattices in a higher rank semisimple Lie group,  items of great  classical interest.

These results were an exploit  of structural stability of hyperbolic actions and required $C^1$ dynamics for the perturbation. 
Recently, Bowden and Mann  investigated $C^0$ perturbations of boundary actions of cocompact groups of isometries of a simply connected Riemannian manifold $X$ of negative curvature \cite{BowdenMann19}.  In particular, this includes the action of a uniform lattice in a rank one semisimple Lie group on its Furstenberg boundary.

Remarkably, Bowden and Mann  proved a $C^0$ local rigidity result for these actions, i.e. for  perturbations in the homeomorphism group of the manifold close in the compact open topology.  Note that in this generality, while the boundary $\partial_\infty X$ of $X$ is a compact manifold, in fact topologically a sphere, it does not always carry a differentiable structure for which the boundary action is $C^1$.   Thus $C^1$-local rigidity of such actions does not even make sense.  Later, Mann and Manning generalized their work to Gromov hyperbolic groups acting  on their boundaries assuming the latter are homeomorphic to a sphere \cite{MannManning21}. Later yet, Mann, Manning and Weisman dealt with the general case of Gromov hyperbolic groups \cite{Mann-Manning-Weisman}.  In a work in progress, they are generalizing these results further to   relatively hyperbolic groups for the action on the Bowditch boundary.

The main goal of this paper is to establish similar results for actions of uniform lattices in higher rank semisimple Lie groups on their Furstenberg boundaries (\cref{thm:main_thm_lattice}). These results fit in with the general program of studying actions of higher rank semisimple Lie groups $G$ and their lattices $\Gamma$ on manifolds.  This quest was initiated  by R. Zimmer  in the 1980s and is now called the Zimmer program \cite{Zimmer83,Zimmer84, Zimmer-ICM}.  It followed his remarkable generalization of Margulis' superrigidity theorem (for homomorphisms from  lattices into Lie groups) to cocycles over  $G$ or $\Gamma$ actions over finite measure preserving actions, e.g. volume preserving actions on compact manifolds.   Much  has been achieved since then; see the surveys by Fisher \cite{Fisher20,Fisher-survey2007}.  In particular, Brown, Fisher and Hurtado made major progress towards Zimmer's conjecture that lattices $\Gamma$ cannot act on ``low-dimensional'' manifolds \cite{BrownFisherHurtado22,BrownFisherHurtado20,BrownFisherHurtado21}. This follows earlier special results by Farb-Shalen, Polterovich and Franks-Handel \cite{FarbShalen99,Polterovich02,FranksHandel03}.

At this point we have many examples of actions of such groups  which are not of algebraic nature per se, even though they are constructed from  algebraic actions via a blow-up construction \cite{KatokLewis96, Benveniste00, Benveniste-Fisher,Fisher-survey2007}.  The classification problem however persists, in a softer form, asked by Zimmer,  namely that the actions are homogeneous  on a dense open set. Labourie in \cite{Labourie98} and later    Margulis  in his list of problems for the new century \cite{Margulis-problems} included a precise formulation  . 

Establishing the desired classification under natural dynamical or geometric hypotheses has also been pursued with vigor.  For example, we now have good understanding of ``hyperbolic'' actions of $G$.  For instance, Brown, Rodriguez Hertz and Wang classified actions of higher rank lattices on tori and nilmanifolds \cite{BrownHertzWang17}.
More recently,  Damjanovic, Spatzier, Vinhage and Xu \cite{DamjanovicSpatzierVinhageXu22} proved algebraicity of volume preserving actions of $G$ with many  Anosov elements on any compact manifold.
%Brown, Rodriguez Hertz and Wang classified actions of higher rank lattices on tori and nilmanifolds \cite{BrownHertzWang17}.

Zimmer's program mostly discussed volume preserving actions.  Actions on boundaries of $G$ or $\Gamma$ action on $G/Q$, $Q$ a parabolic, however cannot preserve any probability measure. Still they are important building blocks in understanding general actions of these groups on compact manifolds.  Hence we are keen on understanding their rigidity properties, both local and global ones.  
For $C^k$ ($k \geq 1$) or Lipschitz actions of uniform lattices, local rigidity  was resolved  in \cite{Kanai, KatokSpatzier97,KapovichKimLee19}, as discussed above.  
%Now there are even global rigidity results by Brown, Rodriguez Hertz and Wang .....
Brown, Rodriguez Hertz and Wang even get global rigidity results for certain groups and parabolics \cite{AFZ}. 

However, $C^0$-local rigidity for the standard $\Gamma$ actions  on boundaries has not been studied in the higher rank case.  This is precisely what we achieve in \cref{thm:main_thm_lattice} - our main result in this paper  - for uniform lattices.
%in our  main result,   \cref{thm:main_thm_lattice}.
This opens up a new direction in the Zimmer program, namely studying continuous actions of $G$ and $\Gamma$.  
So far the only known and fabulous results for higher rank lattices are for actions on the circle or the real line  due to  Morris \cite{WitteMorris94} for $SL(n,\Z)$ and other lattices of full $\Q$-rank.  This was proven in full for all higher rank lattices by Deroin and Hurtado \cite{DeroinHurtado20} recently.  These works followed earlier results for $C^1$-actions by Ghys, Burger-Monod and Navas \cite{Ghys99,BurgerMonod99,Navas02}.  

Naturally, one can now  ask questions about $C^0$ local semi-rigidity for other algebraic actions, e.g. ones preserving an affine structure.  For these, local rigidity properties are known under  higher regularity hypotheses.  Indeed,  Fisher and Margulis prove $C^k$ local rigidity for a suitable $k$ depending on the type of  action and the dimension of the manifold \cite{FM1,FM2,FM3}. These follow  earlier work by Margulis and Qian \cite{Margulis-Qian}.  Are these $C^0$-locally semi-rigid?  Is there a general  $C^0$ Zimmer program?  We refer to Weinberger's survey for further constructions, results, questions and conjectures \cite{Weinberger11,FarbShalen00}. 

We note that such actions are usually not $C^0$ locally rigid.  We actually find a general construction of $C^0$ perturbations for  any $C^1$ action of a countable group on a compact manifold with a dense orbit; see   \cref{construction:non_conjugate_perturbation}. Even better, in our  
Theorem \ref{thm:geod_bdy_not_rigid}, we construct new $C^{\infty}$ actions of 
$G$ and $\Gamma$ on spheres, deforming   the boundary action of $G$ on the sphere at infinity of its locally symmetric space.  These are not $C^0$ semi-conjugate to the unperturbed action. Our constructions generalize further, cf. Remark \ref{rem:semi}.   Let us also  note that Uchida and Kuroki had constructed other $C^0$ actions of $SL(3,\R)$ on $S^4$ \cite{Uchida85,Kuroki05}.  These are neither deformations of  nor semi-conjugate to our actions; indeed they only have finitely many orbits. Finally, we mention the recent work by Fisher and Melnick \cite{FM2022} who construct many new examples of real analytic actions of $\SL(n,\R)$ on manifolds of dimension $n$ which form a discrete set of  $C^0$ conjugacy classes.  Their real analytic conjugacy classes however range over a continuum.

 We now informally discuss the main ingredients of our work. Hyperbolicity will again play a fundamental role, just as 
in the local rigidity results alluded to at the beginning of the paper.    While we do not have simple north-south dynamics anymore, individual hyperbolic elements act by Morse-Smale diffeomorphisms, a more complicated expression of hyperbolicity.  The magic glue however is provided by the geodesic boundary of the symmetric space that all the Furstenberg boundaries naturally sit in. This is where the structure of the flats at infinity comes in.  In the end, we will exploit to a large extent that quasiflats in higher rank symmetric spaces are close to a finite union of flats (\cref{lem:quasiflat}).  % 3
This approach mirrors the  work of Bowden and Mann which  centrally uses quasi-geodesics and Morse lemma techniques.

To our knowledge, this is the first instance of the use of quasi-isometric rigidity in higher rank dynamics, a connection long sought after.  In fact, it appears that this is the first application of these techniques outside of coarse geometry.

To start, let us first formulate a precise definition of local semi-rigidity.  
Given a topological space $F$, endow its group of homeomorphisms $\Homeo (F)$ with  the compact open topology.  Suppose $\Gamma$ is a finitely generated group with finite generating set $\Gamma _0$.  Suppose $\rho _0$ is an action of  $\Gamma$  on a topological space $F$ by homeomorphisms.

We call $\rho _0$ \emph{ $C^0$-locally semi-rigid} if there is a neighborhood $N$ of the $\id \in \Homeo (F)$  such that:  if $\rho: \Gamma \to\Homeo (F)$ is an action of $\Gamma$ on $F$ such that for all $\gamma \in \Gamma _0$,  
$\rho (\gamma) \cdot \rho_0 (\gamma ) ^{-1} \in N$, then there exists  a continuous surjective map $f: F \to F$ which intertwines $\rho$ and $\rho_0$ (i.e. $f \circ \rho(\gamma)=\rho_0(\gamma) \circ f$ for all $\gamma \in \Gamma$).  In other words, $\rho _0$ is a continuous factor of $\rho$. One also calls such a map $f$ a {\em semi-conjugacy}, or more precisely, a {\em $(\rho,\rho_0)$-semi-conjugacy}.

Obtaining semi-conjugacies is  often the optimal local rigidity result one can obtain for $C^0$-perturbations, as we will see in  \cref{construction:non_conjugate_perturbation}. There we adapt the well-known Denjoy construction to construct  $C^0$-perturbations  of any $C^1$-action $\rho_0$ of a countable group on a compact manifold with a dense orbit. These perturbations are not $C^0$ conjugate to $\rho_0$. Hence in the context of lattice actions on boundaries, the next result, our main theorem, is the best possible.

\begin{theorem}[Theorem \ref{thm:rigidity_G/Q}]
\label{thm:main_thm_lattice}
Suppose $\Gamma$ is a uniform lattice in a connected linear semisimple Lie group $G$ without compact factors. Let $\rho_0$ be the action of $~\Gamma$ on $G/Q$ by left multiplication. Then any $C^0$-action $~\rho$ of $\Gamma$ that is sufficiently close to $\rho_0$ in the $C^0$ topology is semi-conjugate to $\rho_0$. Moreover, the semi-conjugacy is close to $\id$ in $\Homeo(F)$. 
\end{theorem}

We  note that in the  rank one case, our result is a special case of Bowden and Mann's work \cite{BowdenMann19}.  Moreover, our proof  essentially simplifies to their argument. Furthermore, it follows from Thereom \ref{thm:rigidity_G/Q} that pre-images of points under the semi-conjugacy have small diameter, showing the Denjoy construction is, in essence, the only obstruction to conjugacy. We remark that our proof does not cover the case of non-uniform lattices.  We naturally conjecture that the statements still hold true.

From our Main Theorem, in \cref{sec:Applications}, we  recover prior local rigidity results that yield conjugacies,  under additional assumptions on the regularity of the perturbations.   Kapovich, Kim and Lee \cite{KapovichKimLee19} prove this for Lipschitz perturbations of the standard action, using a structural stability approach \`a la Sullivan.

\begin{corollary}{\cite[Theorem 1.2]{KapovichKimLee19}}\label{cor:LipRigidityIntro}
    Let $\Gamma<G$ be a uniform lattice in a connected linear semisimple Lie group $G$ of higher rank and without compact factors. Let $Q$ be a parabolic subgroup of $G$, and let $\rho_0$ be the action of $\Gamma$ on $G/Q$ by left translation. Then there is a neighborhood $U$ of $\rho_0$ in Lipschitz topology such that every $\rho\in U$ is $C^0$-conjugate to $\rho_0$.
\end{corollary}

For smooth $C^1$ close perturbations, Kanai under more restrictive assumptions  and Katok-Spatzier in general obtained the following, with very different proofs:

\begin{corollary}\cite{Kanai, KatokSpatzier97} \label{cor:ks}
    Let $\Gamma<G$ be a uniform lattice in a connected linear semisimple Lie group $G$ of higher rank and without compact factors. Let $Q$ be a parabolic subgroup of $G$, and let $\rho_0$ be the action of $\Gamma$ on $G/Q$ by left translation. Then there is a neighborhood $U$ of $\rho_0$ of smooth actions that are close to $\rho_0$ in $C^1$-topology such that every $\rho\in U$ is $C^\infty$-conjugate to $\rho_0$.
\end{corollary}

In up-and-coming work \cite{AFZ}, Brown, Rodriguez Hertz and Wang  obtain a $C^k$-version of Corollary \ref{cor:ks}.  Remarkably, their work even applies to  non-uniform lattices.  For minimal Furstenberg boundaries of split simple groups, they even prove a global rigidity result \cite{AFZ,AFZ22}.

Naturally one  asks how rigid the actions of lattices $\Gamma \subset G$ are on other types of boundaries.  The most natural is the geodesic boundary. In the higher rank case, it turns out that we can always deform nontrivially and these deformations can be made as smooth as we wish. For $k \geq 0$, we will say that $\rho_0$ is \emph{not locally semi-rigid among $C^k$-actions} if any sufficiently small neighborhood of $\rho_0$ in the space of $C^k$ actions contains an action that is not $C^0$-semi-conjugate to $\rho_0$.

\begin{theorem}(see \cref{sec:action_on_geod_bdry})
\label{thm:geod_bdy_not_rigid}
    Let $\Gamma$ be a lattice for a  higher rank symmetric space of noncompact type $X$, not necessarily  uniform. Then the standard action of $\Gamma$ on $\partial_\infty X$ is not locally semi-rigid among $C^k$ actions for any $k\in [0,\infty]$.
\end{theorem}

We  note that, in a similar vein, Fisher and Melnick have constructed new actions  of $G$ on circle bundles over projective spaces \cite{FM2022}.  They even find real analytic actions.  We however do not know if the perturbations in Theorem \ref{thm:geod_bdy_not_rigid}  can be made real analytic.

One wonders if analogous results hold more generally in non-positive curvature. More precisely, consider a compact non-positively curved manifold $M$ without local Euclidean de Rham factor together with the action of its fundamental group $\pi_1(M)$ on the geodesic boundary $\partial_{\infty}{\til{M}}$ of the universal cover $\wt{M}$. Is this action $C^0$ locally semi-rigid? We suspect that    local semi-rigidity holds for manifolds with isolated flats while we believe that  local semi-rigidity fails for graph manifolds.  It seems rather difficult to formulate a precise condition when local semi-rigidity holds.

\subsection*{Outline of the arguments:} 
We first give an outline for the proof of \cref{thm:main_thm_lattice}. Consider the standard action of a uniform lattice $\Gamma$ on $G/Q$ for a parabolic subgroup $Q$. We employ the suspension construction to get a Weyl chamber bundle $\Gamma \backslash G/M_Q$ over the  locally symmetric space associated to $\Gamma$. The Weyl chamber flow gives rise to center stable and center unstable manifolds. Each of them, on the Weyl chamber bundle over the universal cover, is isometric to the symmetric space associated to $G$. Given a perturbed action on $G/Q$, we construct a new foliation on the Weyl chamber bundle that is a perturbation of the center stable foliation. In the smooth setting, this construction can be done by pulling back appropriate foliations by a diffeomorphism between suspensions. This diffeomorphism is constructed by using the Implicit Function Theorem to solve certain coboundary equations, see for example \cite[Appendix]{ConnellNguyenSpatzier22}. A key requirement in our argument here is that the perturbed center stable foliation is continuous but with $C^1$ leaves. The argument in \cite{ConnellNguyenSpatzier22}  in the aforementioned smooth setting does not guarantee this. Our approach in this paper is to construct each leaf of the foliation on the Weyl chamber bundle over the universal cover by \emph{developing} from a local leaf. A local leaf is constructed by taking a barycenter of finitely many center stable leaves with weights varying smoothly. 

The next main step is establishing a weak form of structural stability. The intersection of each center unstable leaf with a perturbed center stable leaf is a smooth manifold and it has the same dimension as a center leaf for the Weyl chamber flow. Locally, this intersection is a biLipschitz image of a center leaf. We can improve it to a global biLipschitz image of a center leaf. Then we use quasi-isometric rigidity to prove that the intersection is shadowing an actual center leaf for the Weyl chamber flow. This is done first by showing that a $L$-biLipschitz flat (i.e. the image of a flat under an $L$-biLipschitz map) is uniformly close to a single flat when $L$ is sufficiently close to $1$. The global shadowing conclusion follows from how different flats union to a symmetric totally geodesic subspace that is isometric to a center leaf.

After establishing a weak form of structural stability on the large scale, we can start constructing a semiconjugacy. For this, we use the correspondence - via shadowing - between intersections of center unstable leaves with perturbed center stable leaves and center stable leaves respectively. The proof that the semi-conjugacy is well-defined and continuous is similar to the proof that a quasi-isometric embedding extends to a continuous map on the boundaries.

We finish this section with an outline for the proof of \cref{thm:geod_bdy_not_rigid}. First  note that the standard action $\rho_0$ of $G$ on the geodesic boundary of the associated symmetric space preserves the Tits metric. In other words, Weyl chambers map to Weyl chambers isometrically under $G$. The perturbed action $\rho$ that we construct will still map Weyl chambers to Weyl chambers, exactly as $\rho _0$ does.  However, we will move in the Weyl chamber directions, moving either towards the center of the chamber or the Weyl chamber faces.

\subsection*{Acknowledgements}
We thank Katie Mann,  David Fisher, Zhiren Wang and  Aaron Brown  for helpful discussions and conversations. 
CC and TN thank University of Michigan, CC thanks Heidelberg University, CC and MI thank Karlsruhe Institute of Technology, and CC and TN thank Max Planck Institute for Mathematics for hosting them at various periods during the course of this project.

\section{Preliminaries I: Barycenters}
\label{sec:barycenter}

\subsection{Definition of barycenter}
Let $F$ be a $C^{k}$ Riemannian manifold for some $k\geq 2$. We say a set $K\subset F$ is {\em geodesically convex} or {\em totally convex}  if for every pair of points $x,y\in K$ there is a unique minimizing geodesic segment of $F$ between $x$ and $y$ and this segment is contained in $K$. As is well known, normal neighborhoods in Riemannian manifolds are geodesically convex \cite{doCarmo92}.

Suppose $\mu$ is a probability  measure  with $\op{supp}(\mu) \subset K$ where $K$ a geodesically convex subset of $F$. Consider the function \[x\mapsto \mc{B}_\mu(x)=\frac12 \int_F  d(x,y)^2d\mu(y).\]
The {\em barycenter of $\mu$} in $K$, denoted by $\op{bar}(\mu)$ in $K$, is the unique point in $K$ where the minimum of the function occurs. Note that the existence of a minimum is guaranteed by the compactness of $K$ while the uniqueness depends on strict convexity of $\Bc_{\mu}$ (and hence depends on $d$ and $\op{supp}(\mu)$).

\begin{remark}\label{rem:barybound}
Suppose $\mu$ is a probability measure such that $\op{supp}(\mu) \subset B_{\eps}(x)$  where $B_{\eps}(x)$ is a metric ball of radius $\eps$ (in the metric $d$) which is geodesically convex. Assume that $\op{bar}(\mu)$ exists. Then, 
\[ 
\op{bar}(\mu) \in B_{2\eps}(x).
\]
Indeed, for any point $y \not \in B_{2 \eps}(x)$ and any $z\in B_{\eps}(x)$, we have by the triangle inequality
\[
2\eps \leq d(y,x)\leq d(y,z)+d(z,x)< d(y,z)+\eps
.\]
So $d(y,z)\geq \eps > d(x,z)$ and thus $\mc{B}_\mu(x)< \mc{B}_\mu(y)$. Thus $y$ cannot be the barycenter of $\mu$. 
\end{remark}

We now introduce the following condition which is sufficient for the barycenter map to be well-defined. 
\vspace{.5em}

\emph{Condition $(\star)$: } We will say that a measure $\mu$ on $F$ satisfies \emph{Condition $(\star)$} if 
\begin{itemize}
\label{list:condition_*}
    \item $\op{supp}(\mu)$ is contained in a compact geodesically convex subset $K_{\mu} \subset K$ of $F$,
    \item $\Bc_{\mu}(x)$ has one critical point in $K_\mu$, i.e. there exists $x_0 \in K_\mu$ such that
    \begin{align}\label{eq:G}
    G(x_0,\mu)=\nabla_{x=x_0} \Bc_{\mu}(x)=\int_F d(x_0,y)\grad_{x=x_0} d(x,y)d\mu(y)=0,
    \end{align}
    \item $D_xG(x,\mu)$, the Hessian of $\Bc_{\mu}(x)$, exists and is positive definite on $K_\mu$ (cf. \cref{rem:C2_regularity_dist_function}). 
\end{itemize}
\vspace{.5em}

Note that if $\mu$ satisfies  \emph{Condition $(\star)$}, then $\op{bar}(\mu)$ in $K_{\mu}$ is well-defined. Moreover, measures satisfying this condition are well-behaved under affine diffeomorphisms.

\begin{remark}
\label{rem:C2_regularity_dist_function}
    Since the Riemannian metric on $F$ is $C^2$, then $y \mapsto d(x_0,y)^2$ is $C^2$ on metric balls around $x_0$ of sufficiently small radius (see \cref{lem:reg_of_dist}). Thus the Hessian of $\Bc_{\mu}(x)$ exists on a sufficiently small neighborhood of $x$. Hence the last condition in \emph{Condition $(\star )$} is satisfied by $\mu$ provided $\diam(K_{\mu})$ is sufficiently small. 
\end{remark}

\begin{lemma}
\label{lem:barycenter_under_affine_map}
    Suppose $M,N$ are $C^2$-Riemannian manifolds. Let $f:M \to N$ be an affine diffeomorphism with respect to the respective Riemannian connections. If $\mu$ and $f_*\mu$ are  measures on $M$ and $N$ satisfying \hyperref[list:condition_*]{Condition $(\star)$} with $\diam(K_\mu)$ and $\diam K_{f_*\mu}$ sufficiently small, then
     \[\op{bar}(f_* \mu)=f(\op{bar}(\mu)).\]    
\end{lemma}
\begin{proof}

By \cref{rem:C2_regularity_dist_function}, $\Bc_{\mu}$ and $\Bc_{f_*\mu}$ are $C^2$ functions on $K_\mu$ and $K_{f_*\mu}$.  Note that it suffices to prove that if $G(x,\mu)=0$, then $G(f(x),f_*\mu)=0$. Indeed, since $\mu$ (resp. $f_*\mu$) satisfies Condition $(\star)$, it has a unique point of minimum for the function $x \mapsto \Bc_{\mu}(x)$ (resp. $x \mapsto \Bc_{f_* \mu}(x)$) in $K_{\mu}$ (resp. $K_{f_*\mu}$).  We only need to show that $f$ maps one critical point to another.

In order to verify this, note that $d(z,y)\nabla_zd(z,y)=v_{z,y}$ where $v_{z,y}$ is the tangent at $z\in M$ to the geodesic $\sigma$ with $\sigma(0)=z$ and $\sigma(1)=y$. Then $\op{bar}(\mu)$ is the unique point such that 
\[0=G(\op{bar}(\mu))=\int_{M}v_{\op{bar}(\mu),y}d\mu(y).\]
Then
\[
\int_{N} v_{f(\op{bar}(\mu)),y}d(f_*\mu)(y)=\int_{M} v_{f(\op{bar}(\mu)),f(y)}d\mu(y)=\int_{M}df_{\op{bar}(\mu)}(v_{\op{bar}(\mu),y})d\mu(y),
\]
where the second equality uses the fact that $f$ is affine (and maps geodesic segments to geodesics segments) and the first equality is the change of variables formula. Then
\[\int_{N} v_{f(\op{bar}(\mu)),y}d(f_*\mu)(y)=df \left( \int_{M} v_{\op{bar}(\mu),y} d\mu(y) \right)=df(0)=0.\]
Thus $G(f(\op{bar}(\mu)),f_*\mu)=0$. By uniqueness of the critical point of $G$ in $K_{f_*(\mu)}$ (see Condition $(\star)$), 
 \[f(\op{bar}(\mu))=\op{bar}(f_{*}\mu).\]
\end{proof}

\vspace{.5em}

Finally let us explain the following special case of barycenter maps that we will be interested in. Let $\mu=\mu_{z,w}$ be a finite weighted sum of Dirac measures at points $z=\set{z_1,\dots,z_n}\subset F$ with weights $w=\set{w_1,\dots,w_n}\in \R^n$. Then writing $G(x,w,z):=G(x,\mu_{z,w})=\sum_{i=1}^n w_id(x,z_i)\nabla_x d(x,z_i)$ we have the implicit equation $G(\op{bar}(\mu_{z,w}),z,w)=0$. Denoting the covariant derivative at $x=\op{bar}(\mu_{z,w})$ by $D_x$, we obtain the Hessian of $\mc{B}_{\mu_{z,w}}$,
\begin{align}\label{eq:DG}
D_x G(x,w,z)=\sum_{i=1}^n w_i \left(D_xd(x,z_i)\tensor \nabla_xd(x,z_i)+d(x,z_i)D_{x}\nabla_x d(x,z_i)\right).
\end{align}

In \cref{sec:derivative_barycenter}, we will show  that if $\diam_F\set{z_1,\dots, z_n}$ is sufficiently small, then  $\op{bar}(\mu_{z,w})$ is well-defined and is well-behaved under affine diffeomorphisms. 

\begin{remark} \label{rem:std_action_fiber_bary_equiv}\
\begin{enumerate}
    \item If $\diam\set{z_1,\dots,z_n}$ has sufficiently small diameter, then $\mu_{z,w}$ satisfies \hyperref[list:condition_*]{Condition $(\star)$} and $\op{bar}(\mu_{z,w})$ is well-defined.
    \item Suppose $F'$ is another $C^3$ Riemannian manifold and $f: F\to F'$ is an affine diffeomorophism. If $\diam \set{f(z_1),\dots,f(z_n)}$ is also sufficiently small, then $\op{bar}(f_*\mu_{z,w})=\op{bar}(\mu_{z,w})$.
\end{enumerate}
The smallness of the diameter that we require in the above remark is quite explicit and depends only on the Riemannian metrics on $F$ and $F'$. See \cref{lem:Q_id} and \cref{cor:Q_pos_def} for details.  \\
Part (1) of the remark follows from \cref{lem:Q_id}, because this lemma shows that $D_xG(x,w,z)$ is positive definite in a neighborhood of $x=\op{bar}(\mu_{z,w})$. In order to prove part (2) of the remark, we first apply part (1) to $f_*\mu_{z,w}$ and then apply \cref{lem:barycenter_under_affine_map}. 
\end{remark}

\subsection{$C^{1,2}$ regularity} 

\label{sec:C_12_functions}
Next, we will discuss the regularity of the barycenter maps when the metric varies. For this, the appropriate regularity of the family of metrics will be $C^{1,2}$ that we now define. But we first introduce the notion of $C^{\ell,2}$ foliations and functions.

\begin{definition}
Suppose $N$ is a smooth manifold and $\Fc$ is a continuous foliation of $N$ with smooth leaves. Let $\ell \geq 1$. 
\begin{enumerate} 
\item  The foliation $\Fc$ is $C^{\ell,2}$ if $\Fc$ is a $C^{\ell}$ foliation and the leaf-wise 2-jets of the leaf inclusions vary transversally in a $C^\ell$ way. In particular, $p:N \to L$ is a $C^{\ell,2}$-fiber bundle provided $p:N \to L$ is a $C^0$ fiber bundle and the foliation $\Fc=\set{p^{-1}(x):x \in L}$ of $N$ by fibers is a $C^{\ell,2}$ foliation.

\item Suppose $\Fc$ is a $C^{\ell,2}$-foliation and $f:N \to \Rb$ is a continuous function. Then $f$ is a $C^{\ell,2}$-function provided the leaf-wise derivatives up to order $2$ exist and vary transversally in a $C^\ell$ way (i.e. the tangential derivatives up to order 2 are transversally differentiable up to order $\ell$).
\end{enumerate}
\end{definition}

In particular, let us explain what we will mean by a $C^{1,2}$ family of metrics on a manifold $F$. Let $B(0,\eps_0) \subset \Rb^m$ and let $\set{g_s: s \in B(0,\eps_0)}$ be a family of $C^{2}$ Riemannian metrics on $F$. Consider the product foliation $\Fc=\set{ F \times \set{s}: s \in B(0,\eps_0)}$ of $F \times B(0,\eps_0)$ that is $C^{\ell,2}$ for any $\ell \geq 1$. 

\begin{definition}
We will say that the family $\set{g_s: s \in B(0,\eps_0)}$ is a \emph{$C^{1,2}$ family of metrics} if the function $F \times B(0,\eps_0) \ni (x,s) \mapsto g_s(X_x,X_x)$ is a $C^{1,2}$ function for any smooth vector field $X$ on $F$.
\end{definition}

 A consequence of having a $C^{1,2}$ family of metrics is the following regularity of  distance functions. See the moreover part of \cref{lem:reg_of_dist} for a proof. 

\begin{lemma}[\cref{lem:reg_of_dist}]
\label{lem:C2_reg_of_dist}
Suppose $\set{g_s: s \in B(0,\eps_0)}$ is a $C^{1,2}$ family of metrics on a compact manifold $F$ such that the sectional curvature of each $g_s \in [-a^2,b^2]$ where $a,b >0$ are two fixed constants. Then, for all $z$ and $x$ with $d_s(x,z)$ sufficiently small, $d_s(x,z)^2$ is $C^2$ in $x$ and $z$, $C^1$ in $s$ and $\nabla_{x}^s d_s(x,z)^2$ is $C^1$ in $s$. 
\end{lemma}

\subsection{Regularity of the barycenter maps when the metric varies} In \cref{sec:derivative_barycenter}, we will prove several results on the regularity of the barycenter map as the Riemannian metric on $F$ as well as the measures on $F$ vary. For the convenience of the reader, we now quote some of the key results from \cref{sec:derivative_barycenter} that we will require later in the paper.

Fix constants $\eps_0>0$, $a,b \geq 0$, $m,n \in \Nb$ and $k \geq 3$. Let $F$ be a smooth manifold and $B(0, \eps_0) \subset \Rb^m$.  Suppose that
\begin{enumerate}
    \item $B(0,\eps_0) \ni s \mapsto g_s$ is a $C^{1,2}$ family of Riemannian metrics on $F$ with sectional curvatures $\kappa_s \in [-a^2,b^2]$,
    \item there are $C^1$ functions $w_i:B(0,\eps_0) \to \Rb$ and $z_i: B(0,\eps_0) \to F$ for $1 \leq i \leq n$ such that $\mu(s):=\sum_{1}^n w_i(s) \delta_{z_i(s)}$ are probability measures on $F$.
\end{enumerate} 

Let $r(s)=\diam\set{z_1(s),\dots,z_n(s) }$.

\begin{proposition}
\label{prop:bary_C1}
Suppose the setting is as above. Then there exists a constant $R>0$ (that depends only on $a$, $b$, and injectivity) such that whenever $r(s_0)<R$, $\op{bar}_{s_0}(\mu(s_0))$ exists and the map $s \mapsto \op{bar}_s(\mu(s))$ is $C^1$.
\end{proposition}
\begin{proof}
This is \cref{prop:dbary_limit}.
\end{proof}

The next proposition gives a finer estimate on the derivative of the barycenter map as the metrics $g_s$ and the measures $\mu(s)$ vary. In this next result, we will use $d^{\op{Gr}}_{(s,x)}$ to denote the Grassmannian distance on $\op{Gr}(T_{(s,x)}(B(0,\eps_0) \times F))$ induced by the Riemannian metric $ds^2+d_{g_s}^2$ on $B(0,\epsilon_0)\times F$.

\begin{proposition}
\label{prop:parallel_close}
Suppose the setting is as above. Suppose that there is a continuous foliation $\Hc$ on  $B(0,\epsilon_0)\times F$ such that:
\begin{enumerate}[label=(\alph*)]
    \item $\Hc$ has $C^1$ leaves which transversely vary continuously in the leafwise $C^1$-topology,
    \item for each $\xi \in F$, there is a $C^1$ function $\varphi_\xi:B(0,\epsilon_0)\to F$ satisfying $\varphi_\xi(0)=\xi$ and the foliation $\Hc$ is given by
    \[
    \Hc=\set{(s,\varphi_{\xi}(s)) : s \in B(0,\eps_0)}_{\xi \in F},
    \]
    \item the tangent distribution of $\Hc$ is denoted by $\Dc$.
\end{enumerate}
Fix $\xi_1,\dots,\xi_n \in F$. Consider $z_i(s):=\varphi_{\xi_i}(s)$ and the measures $\mu(s):=\sum_{i=1}^n w_i(s) \delta_{z_i(s)}$. Suppose $\op{b}: B(0,\eps_0) \to F$ is the the barycenter map 
\[ \op{b}(s):=\op{bar}_s(\mu(s))) \]
and let $G=\set{(s,\op{b}(s)):s \in B(0,\eps_0)}$ be the graph of $\op{b}$.  Then:
\begin{enumerate}
\item  $D_{s=s_0}\op{b}(s)$ and $T_{(s_0,\op{b}(s_0))}G$ vary continuously in $s_0,\xi_1,\dots, \xi_n$, 
\item for every $\epsilon>0$ and $c\in [0,1)$, there exists $\delta>0$ with the following property: if $s_0\in B(0,c\epsilon_0)$ and $r(s_0)<\delta$, then 
\[
d^{\op{Gr}}_{(s_0,\op{b}(s_0))} \left( T_{(s_0,\op{b}(s_0))}G,\mc{D}\left( s_0,\op{b}(s_0) \right) \right)<\epsilon.
\]
 \end{enumerate}
\end{proposition}

\begin{proof}
This is \cref{prop:parallel_close_in_appendix}.
\end{proof}

\section{Standard Actions and the Generalization of a Lemma of Bowden-Mann}

Our primary goal in this section is to formulate and prove a generalization of a lemma of Bowden and Mann \cite{BowdenMann19}. We will obtain this generalization for a class of actions that we  will call standard actions.

\subsection{Standard Actions} 
\label{sec:standard-action}
Suppose $M$ is a compact non-positively curved Riemannian manifold. Then $\pi_1(M)$ acts on the geodesic boundary of the universal cover $\wt{M}$ of $M$, which we denote by $\partial_\infty \wt{M}$. This boundary action is a primary example of a \emph{standard action}. Note that $\pi_1(M)$ also acts on $T^1\wt{M}$ by isometries of the Sasaki metric. The relationship between these two aforementioned actions is the motivation behind the definition of a \emph{standard action} (\Cref{def:standard action}). This definition, which may appear strange at first sight, is an abstraction of what happens in this case; see \Cref{lem:non-pos-curv-boundary-action} for details. 

We will now informally explain this relationship. Let $\wh{\rho_0}$ be the action of $\pi_1(M)$ on $\wt{M} \times \partial_{\infty}\wt{M}$, where $\pi_1(M)$ acts on $\wt{M}$ by deck transformations and on $\partial_{\infty} \wt{M}$ by the boundary action. Consider the map $\Phi:T^1\wt{M} \to \wt{M} \times \partial_{\infty}\wt{M}$ defined by the foot-point projection to $\wt{M}$ and the radial projection to infinity in $\partial_{\infty}\wt{M}$. The map $\Phi$ maps each fiber in $T^1\wt{M}$ homeomorphically to $\partial_{\infty} \wt{M}$. The most crucial point is that $\Phi$ is a homeomorphism that intertwines the two actions. The map  $\Phi$ also has some regularity: $\Phi^{-1}(\wt{M} \times \set{\xi})$ is a $C^1$-submanifold of $T^1\wt{M}$ for any $\xi \in \partial_{\infty} \wt{M}$. In fact, these are the center-unstable manifolds for the geodesic flow in $T^1\wt{M}$. In \cref{def:standard action}, the intertwining of the product action (on $\wt{M} \times \partial_\infty \wt{M}$) with an isometric action (on $T^1\wt{M}$) is the content of part (3) while regularity of the map $\Phi$ is the content of part (2). Note that generally $\Phi^{-1}(\wt{M} \times \set{\xi})$ is only continuous in $\xi$; the definition makes no further assumptions about this `transversal' regularity (see \cref{rem:standard_action_transversal_reg}).

We will now precisely define standard actions. Recall the definition of $C^{\ell,2}$ fiber bundles from \cref{sec:C_12_functions}. We will denote by $\pr_{X_i}$ the natural $i$-th coordinate projection map $X_1 \times X_2 \to X_i$ for $i=1,2$. 

\begin{definition}
\label{def:standard action} Suppose $F$ and $X$ are differentiable (hence smooth) manifolds, $\Gamma \leq \Isom(X)$ is a discrete subgroup and $\ell \geq 1$. We will say that the action $\rho_0: \Gamma \to \Homeo(F)$ is a {\em standard action} provided there exists a $C^{\ell,2}$-fiber bundle $p:X_F \to X$ with fiber $F$ such that:

\begin{enumerate}
    \item there exists a family of homeomorphisms $\{\pi_x: F_x:=p^{-1} (x) \to F ~|~ x \in X\}$,
    \item the map $\Phi: X_F \to X \times F$ defined by \[\Phi(z):=(p(z), \pi_{p(z)}(z))\] is a homeomorphism  and \[(\pr_F \circ \Phi)^{-1}(\xi)=\Phi^{-1}(X \times \xi)\] is a $C^\ell$-submanifold of $X_F$ for every $\xi \in F$,
    \item there exists a family of $C^{\ell,2}$-Riemannian metrics $\set{g_x: x \in X}$ on the fibers $F_x$ in $X_F$ such that the map $\gamma_x: (F_x,g_x)\to (F_{\gamma x},g_{\gamma x})$ defined by
    \begin{align*}
         \gamma_x:=\pi_{\gamma x}^{-1} \circ \rho_0(\gamma) \circ \pi_x
    \end{align*}
    is affine (with respect to the respective Riemannian connections) for any $\gamma\in \Gamma$ and $x \in X$.
\end{enumerate}
If $X_F,p, \Phi,\{\pi_x\},\{g_x\},\ell$ are as above, we will call $\rho_0:\Gamma \to \Homeo(F)$ a {\em standard action with associated data} $(X_F,p,\Phi,\{\pi_x\},\{g_x\},\ell)$.
\end{definition}

\begin{remark}
Note that the metric $\set{g_x}$ is determined by the maps $\set{\pi_x}$ in the following sense. Suppose $\set{h_x}$ is another family of metrics such that $\gamma_x':=\pi_{\gamma x}^{-1} \circ \rho_0(\gamma) \pi_x$ is also affine. Then $ \id:=(\gamma'_x)^{-1}\circ \gamma_x: (F_x,g_x) \to (F_x,h_x)$ is an affine map. 

In many cases, e.g. $F=S^n$, this implies that $g_x$ and $h_x$ differ only by a scaling factor. But this is not always the case. For instance, if $A \in SL(2,\Zb)$ and $g$ is the Euclidean metric on a torus $\mathbb{T}^2$, then $\id:(\mathbb{T}^2,g) \to (\mathbb{T}^2, A^*g)$ is an affine map but the metrics aren't multiples of each other.
\end{remark}

\begin{remark}
\label{rem:standard_action_transversal_reg}
Note that there is no assumption on the regularity of the sub-manifolds $(\pr_F \circ \Phi)^{-1}(\xi)$ when $\xi$ varies. 
\end{remark}

We will discuss examples in the next subsection. First, however,  let us   introduce the notion of induced action: given a standard action $\rho_0$, we can induce an action $\wh{\rho}_0$ on the fiber bundle $X_F$. 

\begin{definition}
\label{def:induced action}
Suppose $X, F$ are as above and $\rho_0:\Gamma \to \Homeo(F)$ is a standard action with associated data $(X_F,p,\Phi,\{\pi_x\},\{g_x\},\ell)$. Then the induced action on $X_F$, denoted by $\wh{\rho}_0$, is the pullback by $\Phi:X_F \to X \times F$ of the action of $\Gamma$ on $X \times F$ where the action on $X$ is by isometries and the action on $F$ is by $\rho_0$
(i.e. $\gamma \in \Gamma$ acts by $\gamma \cdot (x,\xi)=(\gamma \cdot x, \rho_0(\gamma) \cdot \xi)$). 
\end{definition}

In this language, we  say that $\rho_0$ is a standard action provided the action $\wh{\rho_0}$   preserves the fibers and the fiberwise action is affine with respect to the family of metrics $\{g_x:x \in X\}$. 

\subsection{Examples of Standard Action}

In this section, we discuss several examples of standard actions.

     \subsubsection{\bf Geodesic boundary actions:} Our first example  of a standard action will be  on a sphere and fundamental to our proof of  Theorem \ref{thm:main_thm_lattice},   %The sphere that we consider will be
     the geodesic boundary of a simply connected non-positively curved Riemannian manifold. 
 
     \begin{lemma}
     \label{lem:non-pos-curv-boundary-action}
     Let  $X$ be a simply connected complete smooth Riemannian manifold of non-positive sectional curvature. Suppose $\Gamma \leq \Isom(X)$ be any subgroup.  Let $\rho_0$ be the standard action of $\Gamma$ on the geodesic boundary $\partial_\infty X$ of $X$, which is homeomorphic to ${\mathbb{S}}^{\dim(X)-1}$. Then $\rho_0$ is a standard action.
     \end{lemma}
     \begin{proof}
     
     Let $X_F=T^1X$ and $F=\partial_\infty X$. We will show that the definition of the standard action is satisfied for $\ell=1$. The standard foot-point projection $p:T^1X \to X$ gives $T^1X$ the structure of a smooth fiber bundle. Thus the fiber bundle is $C^{\ell,2}$ (in fact $C^{\ell,\infty}$) for any $\ell \geq 1$. Consider the family of homeomorphisms $\pi_x:T_x^1X \to \partial_\infty X$ given by the radial projections. Then $\Phi:T^1X \to X \times \partial_\infty X$ defined by $\Phi(v)=(p(v),\pi_{p(v)}(v))$ is clearly a homeomorphism. For any $\xi \in \partial_\infty X$, $$({\rm pr}_F \circ \Phi)^{-1}(\xi)=\Phi^{-1}(X \times \{\xi\})=\set{v \in T^1X : \pi_{p(v)}(v)=\xi}$$ is a center-stable manifold of the geodesic flow on $T^1X$. We claim that $({\rm pr}_F \circ \Phi)^{-1}(\xi)$ is a $C^1$-submanifold of $X_F$ for any $\xi \in F$. Indeed, each center-stable manifold is given by evaluating the gradient of the Busemann function at all points of $X$. By \cite[Proposition 3.1]{HeintzeImHof77}, the gradient of the Busemann function is $C^1$ when $X$ is as in this Lemma. Thus each center-stable manifold is $C^1$.

    Let $h$ be the given Riemannian metric on $X$. For each $x \in X$, let $g_x$ be the Riemannian metric on $T_x^1X$ induced by $h$. The fiber $(T_x^1X,g_x)$ is isometric to the standard unit sphere. In fact, this metric is the restriction of the Sasaki metric on $T^1X$ on the fiber $T_x^1X$. Since the Sasaki metric is smooth whenever $h$ is smooth, $x \mapsto g_x$ varies smoothly (and hence $C^{\ell,\infty}$ for any $\ell \geq 1$). If $\gamma \in \Isom(X)$, then $\gamma$ acts by an isometry of the Sasaki metric. Hence $\gamma_x$ is an isometry. 
    \end{proof}
    
    Note that $({\rm pr}_F \circ \Phi)^{-1}(\xi)$ only varies continuously in $\xi$ in general. But as remarked above, this lack of transversal regularity of the weak-stable manifolds $({\rm pr}_F \circ \Phi)^{-1}(\xi)$ is not a problem since the definition does not require it.
    
    \begin{remark}
    Suppose $X$ and $\Gamma$ is as above. If $({\rm pr}_F \circ \Phi)^{-1}(\xi)$ is $C^{\infty}$ for each $\xi \in F$, then the above proof shows that $\rho_0$ is in fact a standard action with $\ell=\infty.$ In particular we mention three such cases of interest: 
    \begin{enumerate}
        \item when $X/\Gamma$ is a compact manifold with negative sectional curvature. 
        \item when $X$ is a Hilbert geometry with strictly convex boundary, a similar construction works as long as we allow the fibers to have non-Riemannian metrics. 
        \item when $X$ is a Riemannian symmetric space.
    \end{enumerate}
    The latter will be of particular interest to us later in the paper and we discuss it below.

    \end{remark}

    \subsubsection{\bf Boundary actions on $G/Q$:}
    \label{sec:G/Q-sec}
    We will now consider the case of a lattice in a semisimple Lie group acting on the $G/Q$ boundaries, which are subsets of the geodesic boundary of the corresponding Riemannian symmetric space. For notation used in this section, we refer the reader to \Cref{sec:Sym}.
 
\begin{lemma}
\label{lem:G/Q_standard_action}
Suppose $G$ is a connected semisimple Lie group with finite center and without compact factors, $Q < G$ is a parabolic subgroup and $\Gamma < G$ is a discrete subgroup. Then the action of $\Gamma$ by left multiplication on  $F=G/Q$ is a standard action. 
\end{lemma}
\begin{proof}Let $K$ be a maximal compact subgroup of $G$. Then $X=G/K$, equipped with a left $G$-invariant Riemannian metric, is a Riemannian symmetric space with $\Isom(X)^0=G$. Here $F=G/Q$ and $X_F$ will be a space diffeomorphic to $G/M_Q$. In \Cref{sec:Sym}, we will see that $G/M_Q$ is the Weyl chamber $Q$-face bundle.

We will use the notation: for any $v \in T^1X$, let $\gamma_v$ be the geodesic in $X$ determined by $v$. Consider the continuous trivialization of the unit tangent bundle $\Phi_0:T^1X \to X \times S^{\dim(X)-1}$ defined by $\Phi_0(v)=(\pi(v),\gamma_v(\infty))$. Recall that the $\Gamma$ action on $\partial_\infty X$ is a standard action.

Fix $\bdx \in \partial_\infty X$ be such that $Q=G_{\bdx}$, so that $G/Q$ is diffeomorphic to  $G \cdot \bdx \subset \partial_\infty X$ via the orbit map $gQ \mapsto g \cdot \bdx$.  In particular, $G \cdot \bdx$ is a smooth sub-manifold of $\partial X$ as $G$ acts smoothly on $\partial_\infty X$ (\cite[Theorem 1.2 Pg 302]{Bredon72}).

 Let $$X_F:=\set{v \in T^1X :  \gamma_v(\infty) \in G \cdot \bdx}.$$ 
 Define $\Phi:X_F \to X \times F$ by restriction of $\Phi_0$, i.e 
 \begin{align}
 \label{eqn:phi-and-phi0}
 \Phi:=\Phi_0|_{X_F}.
 \end{align} Then $\Phi$ is a continuous trivialization of $X_F$ as a fiber bundle over $X$ with fibers $F=G/Q$. 
 
 We claim that $X_F$ is a smooth submanifold of $T^1X$ that is diffeomorphic to $G/M_Q$. To prove the claim, fix $v_0 \in T^1_{eK}X$ such that $\gamma_{v_0}(\infty)=\bdx$. Consider the orbit map $G \to X_F \subset T^1X$ defined by $g \mapsto g \cdot v_0$. Note that the stabilizer of $v_0$ is $K\cap Q$, i.e. $M_Q$. Further, $G$ acts transitively on $X_F$. Indeed, if $v \in X_F$, then there exists $g \in G$ such that $gv\in T_{eK}X$ and $\gamma_v(\infty)=g \cdot \gamma_{v_0}(\infty)$ (since $G$ acts as transitively on $X$ and and $\gamma_v(\infty) \in G \cdot \bdx$). Thus $X_F$ is diffeomorphic to $G/M_Q$ via the above map. Since $G$ action on $T^1X$ is smooth and $X_F$ is the $G$-orbit of $v_0$, $X_F$ is an immersed submanifold of $T^1X$ (again by \cite[ch.VI, Theorem 1.2]{Bredon72}). But since the $G$ action on $T^1X$ is properly discontinuous, this implies that $X_F$ is locally closed and hence, an embedded submanifold of $T^1X$. 
 
 We now finish the proof that the action on $G/Q$ is a standard action, by restricting all the data from $T^1X$ and $\partial_\infty X$ to the smooth subbundle $X_F\subset T^1X$ and $F=G \cdot \bdx\subset \partial_\infty X$ respectively. Note that with the above claim, it becomes evident that $X_F$ has the structure of a smooth fiber bundle over $X=G/K$ with fibers homeomorphic to $G\cdot \bdx\cong G/Q$. Let $p:T^1X \to X$ and $\{\pi_x:x \in X\}$ be the maps as in \Cref{lem:non-pos-curv-boundary-action} and define $p':=p|_{X_F}$ and $\pi'_x:=\pi_x|_{X_F}$ for all $x \in X$. The $G$-invariant Riemannian metric on $T^1X$ induces a family of metrics ${g'_x:x \in X}$ on the fibers of $X_F$. Then the $\Gamma$ action on $G/Q$ is a standard action with the data $(X_F,p',\Phi, \set{\pi'_{x}}, \set{g'_x},\ell)$ for any $\ell \geq 1$. 
\end{proof}
\begin{remark}
\label{rem:associated_data_for_G/M_Q}
As seen in the proof above, $X_F$ is diffeomorphic with $G/M_Q$. Thus by pre-composing with this diffeomorphism, one can also show that $\Gamma$ has a standard action on $G/M_Q$ with appropriate associated data. In fact, this is the associated data that we will use in \Cref{prop:BMMforG/Q}. Also see \cref{sec:G/M_Q} for further discussion about this action.
\end{remark}

\subsubsection{\bf Factor actions:} This is a degenerate example. Let $X_F=X\times F$ where $F$ is any compact Riemannian manifold  and $\Gamma$ acts by the identity on $X$ and by an isometric action $\rho_0$ on the $F$-factor. More generally, we can take $\rho_0$ to be an affine action on F with respect to the Levi-Civita connection on $F$. Note that $X$ can be any smooth manifold here; in particular, a point.

\subsubsection{\bf Pull-back actions:} Restriction  of standard actions to a closed invariant submanifold $F'\subset F$ is another natural way of getting new examples.  This is essentially a generalization of \cref{lem:G/Q_standard_action} to more general invariant submanifolds. However, for this restriction to be a standard action, one has to carefully ensure that the corresponding submanifolds in $X_F$ are $C^\ell$ for an appropriate $\ell$.

\subsection{Generalized Bowden-Mann Lemma} In this section we generalize \cite[Lemma 3.1]{BowdenMann19} of Bowden and Mann (later adapted by Mann and Manning \cite{MannManning21} for Gromov hyperbolic groups).  We formulate our generalization in the language of standard actions.  Recall the definition of standard action $\rho_0$ (\cref{def:standard action}) and the induced action $\wh{\rho_0}$ (\cref{def:induced action}). 

\begin{proposition}(cf.  \cite[Lemma 3.1]{BowdenMann19})\label{prop:BM} Suppose $X$ is a Riemannian manifold, $\Gamma < \Isom(X)$ is a discrete subgroup acting freely such that $\Gamma\backslash X$ is compact, and $F$ is a smooth manifold. Suppose $\rho_0: \Gamma \to \Homeo(F)$ is  a standard action with associated data $(X_F,p,\Phi,\{\pi_x\},\{g_x\},\ell)$ where $\ell \geq 1$. For any $\rho: \Gamma \to \Homeo(F)$ that is sufficiently  $C^0$-close to $\rho_0$, there is a continuous map $\til{f}:X \times F \to X_F$ such that:

\begin{enumerate}
\item \label{prop:BM_part0}$p(\til{f}(x,\xi))=x$ and $\til{f}(\gamma \cdot x,\rho(\gamma)\xi) =\wh\rho_0(\gamma)\til{f}(x,\xi)$,

\item   \label{prop:BM_part1} the map $\til{f}(\cdot,\xi):X \to \til{f}(X \times \{\xi\})$ is a $C^{1}$ diffeomorphism for every $\xi \in F$. Moreover, for any $(x_0,\xi_0)\in X \times F$, $(\rho,y,\xi) \mapsto D_y\til f(y,\xi)$ is continuous in all three variables at $(\rho_0,x_0,\xi_0)$, 

 \item \label{prop:BM_part4} for every $\epsilon>0$ and for every $x\in X$, there exists a neighborhood $\Uc = \Uc(x,\eps)$ of $\rho_0$ such that 
\[\sup_{\rho \in ~\Uc} ~\sup_{\xi \in F} ~d_{F}(\pi_{x}(\til f(x,\xi)),\xi)<\eps, \] 

\item  \label{prop:BM_part2}
suppose that for any compact set $\Kc\subset X$, there exist constants $c,\alpha$ depending only on $\diam(\Kc)$  such that for all $x,z \in \Kc$, $\pi_z^{-1}\circ \pi_x: F_x\to F_z$ is H\"older with H\"older constants $c$ and $\alpha$. Then for any $\epsilon, R >0$ and $(x,\xi) \in X \times F$, there is a neighborhood $\Uc = \Uc(\eps,R,x,\xi)$ of $\rho_0$ such that: if $\rho \in \Uc$ and $\eta:=\pi_x(\til f(x,\xi))\in F$, then 
\[
\sup_{y \in B_R(x)} d_{F_y}\left( \til{f}(y,\xi), (\pr_F \circ \Phi)^{-1}(\eta) \cap F_y \right) <\epsilon,
\]
where $d_{F_y}$ is the distance on the fiber $F_y$ (induced by the Riemannian metric $g_y$),

\item \label{prop:BM_part3} 
Suppose $h$ is a Riemannian metric on $X_F$ such that $\wh{\rho_0}(\Gamma) < \Isom(X_F,h)$ and $h|_{F_x}=g_x$ for all $x \in X$. For any $\epsilon > 0$, there exists a neighborhood $\Uc = \Uc(\eps)$ of $\rho_0$  such that 
\[
\sup_{\rho \in ~\Uc}  \left( \sup_{(x,\xi) \in X \times F} d^{\op{Gr}}_{\til f(x,\xi)}\left( T_{\til f(x,\xi)} \til f(X \times \xi), T_{\til f(x,\xi)} (\pr_F \circ \Phi)^{-1}(\eta_x) \right) \right) < \eps,
\]
where $\eta_x:=\pi_x(\wt{f}(x,\xi))$.
\end{enumerate}
\end{proposition}
\begin{remark}
    Part (4) of this proposition also holds under the following alternate assumption. Suppose $X_F$ carries a Riemannian metric $h'$ such that $h'|_{F_x}=g_x$ and  each fiber $(F_x,g_x)$ is a totally geodesic submanifold of $(X_F,h')$ (or even more generally, there exists $0<L<1$ such that $d_{X_F}(u,v) \geq Ld_{F_x}(u,v)$ for any $u,v, \in F_x$).

    Note that the two assumptions do not imply each other. For instance, consider the case where $X=\wt{M}$ is a simply connected Riemannian manifold and $X_F=T^1\wt{M}$. If $X$ has a negatively curved metric, then the H\''older assumption (in part (4)) as well the totally geodesic assumption (mentioned in the above paragraph) is satisfied. But if $X$ is non-positively curved then the alternate assumption stated in this remark is satisfied while the H\"older assumption is not.
\end{remark}
\begin{proof}
Without loss of generality, we can assume that $\ell=1$. Since $\rho_0:\Gamma \to \Homeo(F)$ is a standard action with associated data $(X_F,p,\Phi,\{\pi_x\},\{g_x\},1)$, we have:
\begin{enumerate}[label=(\alph*)]
    \item a $C^{1,2}$ fiber bundle $p:X_F \to X$ with fiber $F$, 
    \item  a family of homeomorphisms $\{\pi_x:x \in X\}$ where $\pi_x: F_x \to F$,
    \item a homeomorphism $\Phi: X_F \to X \times F$ given by $\Phi(z)=(p(z), \pi_{p(z)}(z))$,
    \item $(\pr_F \circ \Phi)^{-1}(\xi)$ is a $C^{1}$-submanifold of $X_F$ for every $\xi \in F$,
    \item a $C^{1,2}$-smoothly varying family of Riemannian metrics $\set{g_x: x \in X}$ on the fibers $F_x$ such that
    $\gamma_x:=\pi_{\gamma x}^{-1} \circ \rho_0(\gamma) \circ \pi_x$
    is affine for any $\gamma\in \Gamma$ and $x \in X$.
\end{enumerate}
Moreover $\wh{\rho}_0$ denotes the action on $X_F$ induced by $\rho_0$.

{\bf Construction of $\wt{f}$: } Let $D\subset X$ be a pre-compact connected fundamental domain for the action of $\Gamma$ on $X$. Let $p_{\Gamma}: X \to \Gamma \backslash X$ be the covering map and let $\set{\mathcal O_i}$ be a finite cover of $\Gamma \backslash X$ by connected open sets. Since $\Gamma$ acts freely and properly discontinuously, we may assume that there exists an open cover $\set{\til{\mathcal O}_j}$ of $X$ such that: for each $j$, there is an $i$ for which the restriction of the covering map $p_\Gamma|_{\wt{\Oc_j}}:\til{\mathcal O}_j\to \Gamma \backslash X$ is an isometry onto $\Oc_i$. Indeed, we can obtain such a cover by  shrinking the diameter of $\Oc_i$ sufficiently and taking their pre-images under the covering map.

Let $\Wc=\{\Wc_j: 1 \leq j \leq m' \}$ be the finite sub-family of $\set{\til{\mathcal O}_i}$, consisting of open sets intersecting $\bar D$ non-trivially. Let $\set{\sigma_i,\mc{O}_i}$ be a partition of unity on $\Gamma\backslash X$ subordinate to the open cover $\set{\mathcal O_i}$ and and let $\til{\sigma_i}: X \to \mathbb{R}$ be the lift of $\sigma_i$.

The $\Gamma$-action on $\set{\til{\mc{O}}_i}$ introduces an equivalence relation $\sim_{\Gamma}$ on $\mathcal W$, namely $\Wc_i \sim_{\Gamma} \Wc_j$ if and only if $\Wc_i=\gamma \Wc_j$ for some $\gamma \in \Gamma$. Now we choose one representative from each equivalence class in $\Wc$ and label them as $\Uc_1,\dots,\Uc_{m''}$. Note that we specially mark these $\Uc_i$'s because we will define our map on these $\Uc_i$'s and extend it equivariantly. 

First, for each of the open sets $\Uc_1,\dots,\Uc_{m''}$, we define a map $\psi_{\Uc_i}:\Uc_i\times F\to X_F$ by $$\psi_{\mathcal U_i}(x,\xi)=\pi_x^{-1}(\xi).$$
That is, for each $\xi\in F$, the map $\psi_{\mathcal U_i}(\cdot,\xi)$ is a local section $\mc{U}_i\to X_F$.

As $\mc{W}$ is a finite set, there is a finite subset $S_{\Wc} \subset \Gamma$ such that: if $1 \leq j \leq m'$, then $\mathcal W_j=\gamma \mathcal U_i$ for some $\gamma \in S_{\Wc}$ and $1 \leq i \leq  m''$.

Then for $1 \leq j \leq m'$, if $\mathcal W_j=\gamma \mathcal U_i$ we define $\Psi_{{\Wc}_j}:{\mathcal W}_j\times F\to X_F$ to be 
\begin{align*}
\Psi_{\mathcal W_j}(\gamma x, \rho(\gamma)\xi)=\wh\rho_0(\gamma)\psi_{\mathcal U_i}(x,\xi),
\end{align*}
where $(x,\xi)\in \mathcal U_i\times F$.
We note that if $\mathcal W_i=\gamma' \mathcal W_j$ then we have $\Psi_{\mathcal W_i}(\gamma' x, \rho(\gamma')\xi)=\wh\rho_0(\gamma') \Psi_{\mathcal W_j}(x,\xi)$ for any $x\in \mathcal W_j$ and any $\xi \in F$.

We consider the neighborhood $\Omega$ of $\bar{D}$ defined by: $x \in \Omega$ provided $\{\wt{\mc{O}_i}:x \in \til{\mc{O}_i}\} \subset \Wc$, i.e. all the lifts $\wt{\mc O}_i$ of the open cover $\{\mathcal O_i\}$ containing $x$ belong to $\Wc$. Fix any $ x \in \Omega$. We will now define the map $\til{f}(x,\cdot)$ (see \eqref{eqn:define_f_tilde}). Suppose that, possibly after re-indexing the $\mathcal W_i$'s, $\mathcal W_1,\dots, \mathcal W_m$ are all the $\mc{W}_i$'s that contain $x$. By the remark above, each $\Wc_j$ corresponds to $\gamma_j \in \Sc_{\Wc}$ such that 
\begin{align}
\label{eqn:define_wj}
  \Wc_j = \gamma_j \Uc_i  
\end{align} for some $\Uc_i$. For any $\xi \in F$, set
\begin{align}
    \label{eqn:define_xi_j}
    \xi_j:= \rho_0(\gamma_j)\rho(\gamma_j^{-1}) (\xi). 
\end{align}
Then, 
\[\Psi_{\mathcal W_j}(x,\xi)=\wh\rho_0(\gamma_j)\Psi_{\mathcal U_i}(\gamma_j^{-1}x,\rho(\gamma_j)^{-1}\xi)= \pi_x^{-1}(\rho_0(\gamma_j)\rho(\gamma_j)^{-1}(\xi))=\pi_x^{-1}(\xi_j).\]

For every $\xi\in F$ define, 
\begin{align}
\label{eqn:define_f_tilde}
\til{f}(x,\xi):= \op{bar}_x\left( \sum_{j=1}^m \wt{\sigma_j}(x) \delta_{\pi_x^{-1}(\xi_j)}\right)= \op{bar}_x\bigg(\sum_{j=1}^m \til{\sigma}_{j}(x) \delta_{\Psi_{\mathcal W_j}(x,\xi)}\bigg),
\end{align}
where $\delta_z$ is the Dirac measure at $z$, and $\op{bar}_x$ is the barycenter operation on the fiber $F_x$.
Since $x$ belongs to a compact set, it follows that $\pi_x^{-1}(\xi_j)$ is close to $\pi_x^{-1}(\xi)$ when $\rho$ is sufficiently close to $\rho_0$, and hence the barycenter is well-defined. 

 In order to extend to a $(\wh\rho,\wh\rho_0)$-equivariant map $\til f:X\times F\to X_F$, we need to check that the map $\til{f}$ defined on $\Omega\times F$ is equivariant. That is, if both $x$ and $\gamma x$ are in $\Omega$ for some $\gamma\in \Gamma$ then $\til{f}(\gamma x,\rho(\gamma)\xi)=\wh\rho_0(\gamma)\til f(x,\xi)$, for every $\xi\in F$.
Indeed, if $\mathcal W_1,\dots, \mathcal W_m$ are the open sets containing $x$ then $\gamma\mathcal W_1,\dots, \gamma \mathcal W_m$ are the open sets containing $\gamma x$. By construction, 

\begin{align*}\til{f}(\gamma x, \rho(\gamma)\xi)
&=\op{bar}_{\gamma x}\bigg(\sum_{i=1}^m \til{\sigma}_{i}(x) \delta_{\Psi_{\gamma\mathcal W_i}(\gamma x, \rho(\gamma)\xi)}\bigg) \\
&=\op{bar}_{\gamma x}\left(\sum_{i=1}^m \til{\sigma}_{i}(x) \delta_{\wh\rho_0(\gamma)\Psi_{\mathcal W_i}(x,\xi)}\right)=\wh\rho_0(\gamma) \left(\op{bar}_x\bigg(\sum_{i=1}^m \til{\sigma}_{i}(x) \delta_{\Psi_{\mathcal W_i}(x,\xi)}\bigg)\right) \\
&=\wh\rho_0(\gamma)\til{f}(x,\xi),
\end{align*}
where the second to last equality follows from the fact that the fiber-wise barycenter construction is $\wh\rho_0$-equivariant as $\rho_0$ is a standard action (cf. \cref{rem:std_action_fiber_bary_equiv}).
Therefore the map $\til{f}$ is equivariant on $\Omega\times F$. Since $\Omega$ is an open neighborhood of $\bar D$, $\til f$ extends to a map $\til f: X\times F\to X_F$ that is $(\wh\rho,\wh\rho_0)$-equivariant. This finishes the construction of $\til f$. Now we proceed to prove all parts of the proposition.

\smallskip

\hyperref[prop:BM_part0]{Proof of Part (1)}: This follows from the construction of $\til f.$

\smallskip

\hyperref[prop:BM_part1]{Proof of Part (2)}: Fix any $\xi \in F$. Note that by definition, $p\of \til f=\id$ and hence $\til f|_{X \times \{\xi\}}$ is injective. Then for proving (2), it suffices to show that $\til f$ is $C^{1}$ along the horizontal leaf $X\times \{\xi\}$. Indeed, as $\til f$ covers the identity, this implies that $\til f|_{X \times \{\xi\}}$ is a diffeomorphism onto its image $\til f (X \times \{\xi\})$. Because $\til f$ is equivariant, it is enough to verify the regularity of $\til f(\cdot, \xi)$ on the neighborhood $\Omega$ of $\bar D$ for every $\xi\in F$.

As $\wh{\rho_0}(\Gamma)\backslash X_F$ is a $C^{1,2}$-fiber bundle, we can choose an atlas of $C^{1,2}$-smooth local trivializations. Refining the open cover $\{\mathcal O_i\}$ if necessary, we can assume that the local trivializations of $\wh{\rho_0}(\Gamma)\backslash X_F$ are of the form $\{\mathcal O_i \times F\}$. Let $\{\mathcal W_i\times F\}$ be the finite subset of the corresponding local trivializations, as defined before in \cref{eqn:define_wj}. For $\xi'\in F$, consider the map $\Wc_i \ni y \mapsto \phi_{\xi'}(y):=\pi_y^{-1}(\xi')$. By definition of standard action, each $\phi_{\xi'}$ is a $C^1$ map.

Each local trivialization has a continuous foliation $\Hc$ whose leaves are graphs of $\phi_{\xi'}$ for every $\xi'\in F$, i.e. $\set{(y,\pi_{y}^{-1}(\xi')):y \in \Wc_i}$ and $\xi' \in F$. The weights $\til \sigma_i(y)$ also vary $C^{1}$ in $y$. On each local trivialization $\mathcal W_i\times F$, by pulling back the metrics on $F_y$ to $F$, we obtain a $C^{1,2}$-family of metrics $\set{g_y:y \in \Wc_i}$ on $F$. For $ 1 \leq i \leq q$, let $\xi_i:=\rho_0(\gamma_i)\rho(\gamma_i)^{-1} \xi$, as in \cref{eqn:define_xi_j}. Then \cref{eqn:define_f_tilde} implies that on each local trivialization $\Wc_i \times F$,
\begin{align}
\label{eqn:define_f_in_charts}
\til f(y,\xi)=\op{bar}_y\left( \sum_{i=1}^q\til\sigma_i(y)\delta_{\varphi_{\xi_i}(y)}\right),
\end{align}
where $\op{bar}_y$ is the fiber-wise barycenter with respect to the metric $g_y$.

We now claim that $\til{f}(X \times \xi)$ is a $C^1$ submanifold of $X_F$. The proof of this claim is an application of  \cref{prop:bary_C1}. Each $\mathcal W_i$, via a coordinate chart, can be thought as an open ball in some Euclidean space. In this coordinate chart, $\til f(y,\xi)$ is then a fiber-wise barycenter map on $F_y$ with respect to the metrics $g_y$. Then \cref{prop:bary_C1} implies that for a fixed $\xi$, $\til{f}(y,\xi)$ varies in a $C^{1}$ way with respect to the base point $y\in \mathcal W_i$. The same holds after pushing back to $X_F$ under the trivialization. Now observe that for any $y \in X$, there exists $\gamma \in \Gamma$ and $y_0 \in \Wc_i$ such that $\wh{\rho_0}(\gamma)\til f(y_0,\xi)=\til f(y,\xi)$. Then the claim again follows because $\wh{\rho_0}(\gamma)$ acts by smooth diffeomorphisms.

We now prove the second claim: if $(x,\xi) \in X \times F$, then there exists a neighborhood of $(\rho_0,x,\xi)$ where the map $(\rho,y,\xi') \mapsto D_{y}\til f(y,\xi')$ is continuous in all variables.  The proof is an application of \cref{prop:parallel_close}. Since $\til f$ is $\wh{\rho_0}(\Gamma)$ equivariant, it suffices to prove it on a local trivialization chart $\set{\Wc_i \times F}$. By \cref{prop:parallel_close}, $D_y\til f(y,\xi)$ varies continuously with respect to $\xi_i$. Thus $D_y\til f(y,\xi)$ varies continuously as $\rho \to \rho_0$ or as $\xi$ varies continuously. Furthermore $D_y\til f(y,\xi)$ is continuous in $y$ by part (2).  Thus the moreover part holds on $\set{\Wc_i \times F}$. 

\smallskip

\hyperref[prop:BM_part4]{Proof of Part (3)}:  
 By part (2), $(\rho,x,\xi) \mapsto \til f(x,\xi)$ is continuous in all three variables. As $\pi_x$ is also continuous, the conclusion follows.

\smallskip

\hyperref[prop:BM_part2]{Proof of Part (4)}: Since $\til f$ is $(\wh{\rho},\wh{\rho_0})$-equivariant, note that it suffices to prove the claim for $x\in \bar D$. 
Fix $\eps,R>0$. Let $x \in \bar{D}$ and $\xi \in F$. 

Let $\eps'>0$. We first claim that there exists $\Uc = \Uc(\eps',R,x,\xi)$ such that
\begin{align}
\label{eqn:xi_in_ball}
 \sup_{\rho \in ~\Uc} \sup_{y\in B_R(x)}d_{F_y}\left(\pi_y^{-1}(\xi),\til f(y,\xi)\right)<\epsilon',   
\end{align}
provided $\eps'$ is sufficiently small. We now prove this claim. From the proof of part (3), we know that $(\rho,y,\xi) \mapsto \til f(y,\xi)$ is a continuous map. Since the metrics $\set{g_y}$ are a $C^{1,2}$ family of metrics on the fibers $F_y$, $d_{F_y}(u,v)$ varies continuously in $y$ provided $d_{F_y}(u,v)$ is sufficiently small; see \cref{lem:C2_reg_of_dist}. Then $(\rho,y,\xi) \mapsto \til d_{F_y}(\til f(y,\xi),\pi_y^{-1}(\xi))$ is a continuous function in all variables near $(\rho_0,x_0,\xi_0)$ provided $d_{F_{x_0}}(\til f(x_0,\xi),\pi_{x_0}^{-1}(\xi_0))$ is sufficiently small.

Suppose $\eps'>0$ is sufficiently small. Let $y \in B_R(x)$. Then part (3) and the discussion in the previous paragraph imply that there exists a neighborhood $U_y \subset X$ of $y$ and a neighborhood $\Uc'= \Uc'(\eps',y,\xi)$ of $\rho_0$  such that: 
\[
 \sup_{\rho \in \Uc'} \sup_{z \in U_y} \left(\pi_y^{-1}(\xi),\til f(y,\xi)\right)<\epsilon'.
\]
Now since $B_R(x)$ is pre-compact, we can find $\Uc = \Uc(\eps',R,x,\xi)$ such that \cref{eqn:xi_in_ball} is satisfied and this finishes the proof of the claim.

Let $\eta:=\pi_x(\til f(x,\xi))$. Then by \cref{eqn:xi_in_ball}, $d_{F_x}(\pi_x^{-1}(\xi),\pi_x^{-1}(\eta))<\epsilon'$. Moreover for any $y \in B_R(x)$, $(\pr_F \circ \Phi)^{-1}(\eta) \cap F_y = \pi_y^{-1}(\eta)$. We now claim that after refining $\Uc$, 
\begin{align}
\label{eqn:main_claim_in_proof_of_part3}
\sup_{\rho \in \Uc} \sup_{y \in B_R(x)} d_{F_y}\left( \til{f}(y,\xi), \pi_y^{-1}(\eta)\right) < \eps.
\end{align}

To prove this claim, let $\set{s_1,\dots,s_{d}}$ be the fixed finite generating set of $\Gamma$ corresponding to the compact fundamental domain $\bar{D}$. For any $y \in B_R(x)$, there exists $\gamma\in \Gamma$ such that $y\in \gamma\bar{D}$.  Let $K'= K'(y)$ be the word length of $\gamma$ with respect to this generating set, i.e. $\gamma=s_{i_1}\cdots s_{i_{K'}}$. Note that $K_0:=\max_{y \in B_R(x)} K'(y)$ exists and depends only on $R$ and the choice of generating set (i.e. choice of $\bar{D}$).

By the assumption in part (4), there exist $c> 1$ and $0<\alpha\le 1$ such that for any
$x,z\in \gamma \cdot \bar{D}$ and $u,v\in F_x$, the map $\pi_z^{-1}\circ \pi_x: F_x\to F_z$ satisfies
\[d_{F_z}(\pi_z^{-1}\circ \pi_x(u),\pi_z^{-1}\circ \pi_x(v))<c\:d_{F_x}(u,v)^\alpha.\]
Note that by assumption, $c,\alpha$ depend only on $\diam(\bar{D})$ (and hence are independent on which $\Gamma$-translate of $\bar{D}$ the points $x$ and $z$ lie in).
Then
\begin{align*}
d_{F_y}\left( \til{f}(y,\xi), \pi_y^{-1}(\eta)\right) &\le d_{F_y}\left( \til{f}(y,\xi), \pi_y^{-1}(\xi)\right) +d_{F_y}\left( \pi_y^{-1}(\xi), \pi_y^{-1}(\eta)\right) \\
&\le d_{F_y}\left( \til{f}(y,\xi), \pi_y^{-1}(\xi)\right) + c^{1+\alpha+\dots+\alpha^{K'-1}}d_{F_x}(\pi_x^{-1}(\xi),\pi_x^{-1}(\eta))^{\alpha^{K'}} \\
&\le \epsilon' + c^{1+\alpha+\dots+\alpha^{K'-1}}(\epsilon')^{\alpha^{K'}} \le \epsilon' + c^{K'}(\epsilon')^{\alpha^{K'}} \\
&\le 2c^{K'}(\epsilon')^{\alpha^{K'}}.
\end{align*}

Recall that $K_0=\max_{y \in B_R(x)} K'(y)$ and $K_0$ is independent of $c,\alpha$ and $\eps'$. For fixed constants $c, \alpha$ and  $K_0$, the function $M(\eps'):=\left( \max_{1 \leq K' \leq K_0} c^{K'}(\epsilon')^{\alpha^{K'}} \right)$ is continuous in $\eps'$ and it goes to $0$ when $\eps' \to 0$. So  we can choose  $\eps'=: \eps _0$ sufficiently small such that  $M(\eps_0) < \frac{\eps}{2}$.  
We then choose the neighborhood $\Uc$ of $\rho_0$ for $\eps'=\eps_0$; see \cref{eqn:xi_in_ball}. Then for every $\rho\in \Uc$, we have \cref{eqn:xi_in_ball} and thus the above discussion implies that \cref{eqn:main_claim_in_proof_of_part3} holds. This establishes the claim and finishes the proof of \hyperref[prop:BM_part2]{part (4)}.

\smallskip

\hyperref[prop:BM_part3]{Proof of Part (5)} To prove (5), we start with the same setup as in the proof of (2).  We choose a finite family of $C^{1,2}$-local trivializations $\{\mathcal O_i\times F\}$ of $X_F/\wh\rho_0(\Gamma)$ and let $\set{\Wc_i \times F}$ be the local trivialization charts as earlier. As in the proof of (2), we get a family $\set{g_y:y \in \Wc_i}$ of $C^{1,2}$ Riemannian metrics on $F$ parametrized by $\mathcal W_i$, which is homeomorphic to a ball in some $\Rb^{\dim \Wc_i}$. The set $\mathcal W_i\times F$ has a continuous foliation $\Hc=\set{\Hc_{\xi'}:\xi' \in F}$ whose leaves $\Hc_{\xi'}:=\{(y,\pi_y^{-1}\xi'):y\in \mathcal W_i\}$ are $C^1$. 
As in \cref{eqn:define_f_in_charts}, for every $\xi\in F$ and $y\in \Wc_i$, we have 
$$\til f(y,\xi)=\op{bar}_y\left(\sum_{i=1}^q\til\sigma_i(y)\delta_{\varphi_{\xi_i}}(y)\right),$$
where $\varphi_{\xi_i}(y)=\pi_y^{-1}(\xi_i)$ and $\xi_i=\rho_0(\gamma_j)\rho(\gamma_j^{-1}) (\xi).$ When $\rho \to \rho_0$, $\xi_i$ converges to $\xi$ and hence, $\varphi_{\xi_i}(y)$ converges to $\varphi_{\xi}(y)$ locally uniformly on $\Wc_i$. Then the diameters $\diam \{\varphi_{\xi_i}(y):i=1,\dots,q\}$ tend to 0 for every $y\in \mathcal W_i$. We are now in the setup of applying \cref{prop:parallel_close}.

 Fix $\eps>0$. Let $\delta$ be as in \cref{prop:parallel_close}. There exists a neighborhood $\Uc$ of $\rho_0$ such that if $\rho \in \Uc$, then $\left(\sup_{y\in \Wc_i}\diam \set{\varphi_{\xi_i}(y): 1 \leq i \leq q} \right) < \delta$. Then, by \cref{prop:parallel_close} for any $\rho \in \Uc$ and $y \in \Wc_i$,  
\begin{align}
\label{eqn:grassmannian_dist_small_in_proof}
d^{\op{Gr}}_{\til f(y,\xi)}\left( T_{\til f(y,\xi)} \til f(X \times \xi), T_{\til f(y,\xi)} (\pr_F \circ \Phi)^{-1}(\eta_y) \right) < \eps.
\end{align}
where $\eta_y=\pi_y(\til f(y,\xi))$.

Now we explain how to get this for any point $(y,\xi)$. As $\wh{\rho_0}(\Gamma)$ acts by isometries on $(X_F,h)$, the above Grassmanian distance does not change under $\rho_0(\Gamma)$ action. As $X_F/\wh{\rho_0}(\Gamma)$ is compact, we can choose finitely many local trivializations to cover it; as discussed in the beginning of the proof of part (2). As above, we now apply \cref{prop:parallel_close} for each of these finitely many local trivializations and conclude that there exists a neighborhood $\Uc$ of $\rho_0$ for which \cref{eqn:grassmannian_dist_small_in_proof} holds. 
\end{proof}

\section{Preliminaries II: General Results about Topological Actions}

The main purpose of this section is proving the local semi-rigidity holds for general uniform lattices if it holds for torsion-free lattices. Readers may skip this section if you want to assume the lattices being torsion-free.

In this section, we will work in the very general framework of a group acting by homeomorphisms (on one occasion, by bi-Lipschitz homeomorphisms) on a compact metric space. We will apply these general results later on in the paper to derive two important applications: first, to discuss uniqueness properties of semi-conjugacies between nearby actions, and second, to promote semi-conjugacies from subgroups of finite index to the whole group. In particular, this will allow us to prove local uniqueness of the semi-conjugacies in our main theorem as well as bypass the torsion elements in lattices. 

Before we get into the results, let us consider some elementary examples that illustrate some of the subtleties of working with semi-conjugacies. First, the semi-conjugacies (even conjugacies) between actions need not be unique. Moreover, one may think that if $\rho \to \rho_0$, then any family of semi-conjugacies $\set{\phi_\rho}$ must automatically converge to $\id$. However this also need not be the case.

\begin{example}
Consider the action of $SO(n+1)$ on $S^n$. The action is conjugated to itself by both $\id$ and $-\id$. If we want a discrete group action, we can consider an embedding of a lattice into $SO(n+1)$ and consider the action of this lattice on $S^n$.
\end{example}

More generally we have,

\begin{example}
Consider an action $\rho$ of a group $G$ on a manifold $M$ and suppose $\Gamma<G$ is a subgroup with a nontrivial centralizer $Z(\Gamma)<G$. Then if $\rho_0$ is the restriction of the action to $\Gamma$, then $\rho_0$ is self-conjugate both by $\rho(\sigma)$ for any $\sigma\in Z(\Gamma)$. However for $\sigma\neq 1$, the homeomorphism $\rho(\sigma)$ need not be close to the identity. 
\end{example}

However, we will prove some results in  \cref{sec:unique_semiconjug} and \cref{sec:extension_semiconjug} to control the behavior of semi-conjugacies under some additional assumptions.

\subsection{Some Definitions}

For this section, we assume that $(F,d_F)$ is a compact metric space. We equip $C^0(F,F)$, the space of continuous functions from $F$ to itself, with the distance $d_0(f,g)=\sup_{x\in F}d_F(f(x),g(x))$.  We will denote by $\Bc_{r}(f)$ the open ball in $(C^0(F,F),d_0)$ of radius $r$ around the point $f$. Our focus will be on the group $\Homeo(F)$ of homeomorphisms of $F$, equipped with the compact open topology, i.e. the topology generated by the basic open sets $\set{f \in \Homeo(F): f(K) \subset U}$ where $K \subset F$ is a compact set and $U \subset F$ is an open set.  

Here $\Gamma$ will always denote a finitely generated group. We assume that $\Gamma$ acts on $F$ by homeomorphisms. A faithful $C^0$-action of $\Gamma$ on $F$ is an injective homomorphism $\rho: \Gamma \to \Homeo(F)$. 
The space $\Hom(\Gamma, \Homeo(F))$ is the space of faithful actions of $\Gamma$ on $F$ by homeomorphisms.   We endow it with  the compact open topology.

\begin{definition}
Suppose $\rho, \rho_0:\Gamma \to \Homeo(F)$ are two actions of $\Gamma$. A  continuous surjective map $\phi:F \to F$ is a $(\rho|_H, \rho_0|_H)$-semi-conjugacy for the group $H< \Gamma$ if 
\[
\rho_0(h) \circ \phi =\phi \circ \rho(h) \text{ for all } h \in H.
\]
\end{definition}
If $H=\Gamma$, we will simplify the notation to $(\rho,\rho_0)$ semi-conjugacy.
\begin{definition}
Suppose $\lambda>1$.

\begin{enumerate}
    \item We say that $f \in\Homeo(F)$ is $\lambda$-expanding on an open subset $U \subset F$ if: for all $x',x'' \in U$,   
     \[
    d_F(f \cdot x', f \cdot  x'') \geq \lambda \cdot d_F(x',x'').
    \]
    \item If $\Uc$ is an open cover of $F$, we say that $\rho: \Gamma \to\Homeo(F)$ is a $(\lambda, \Uc)$ uniformly expanding action on $F$ if: 
 for each $x \in F$, there exists $U_x \in \Uc$ and $g_x \in \Gamma$ such that $\rho(g_x)$ is $\lambda$-expanding on $U_x$.
\end{enumerate}
\end{definition}

\subsection{Uniqueness of Semi-Conjugacies}
\label{sec:unique_semiconjug}

 As noted in the examples above, semi-conjugacies and even conjugacies between actions are not always unique. However, under additional assumptions they become unique provided the actions are sufficiently close and the (semi-)conjugacy is sufficiently close to identity. We will provide such conditions in this subsection.

\begin{lemma}\label{lem:semi_unique}
Suppose $\rho_0: \Gamma \to \Homeo(F)$ is a $(\lambda,\Uc)$-uniformly expanding action where $\lambda>1$ and $\Uc$ is an open cover of $F$. Then there exists $\eps>0$ such that: if $\rho:\Gamma\to \op{Homeo}(F)$ is another action and $\phi,\phi':F\to F$ are $(\rho,\rho_0)$ semi-conjugacies satisfying $d_0(\phi,\phi')<\eps$, then $\phi=\phi'$.
\end{lemma}

\begin{proof}
By compactness of $F$, we can extract a finite sub-cover of $\Uc$ such that $\rho_0$ is $(\lambda,\set{U_1,\dots,U_k})$-uniformly expanding. Let $\eps$ be the Lebesgue number of the sub-cover $\set{U_1,\dots,U_k}$ of $F$, i.e. if $d_F(x,y)<\eps$, then $x,y \in U_i$ for some $1 \leq i \leq k$. Suppose $\phi, \phi'$ are $(\rho,\rho_0)$-semi-conjugacies such that $d_0(\phi',\phi)<\eps$. It suffices to show that $\phi^{-1}(p)=(\phi')^{-1}(p)$ for all $p\in F$. 

To prove this fix $p\in F$. Let $x \in \phi^{-1}(p)$ and $y \in (\phi')^{-1}(p)$, i.e. $\phi(x)=\phi'(y)=p$.  We will show that $y \in \phi^{-1}(p)$, i.e. $\phi(y)=\phi(x)$. Note that by symmetry it will imply that $x \in (\phi')^{-1}(p)$ and finish the proof.

Suppose, on the contrary, that $\phi(x)\neq \phi(y)$. But $d_F(\phi(x),\phi(y)) = d_F(\phi'(y),\phi(y)) < \eps$. Then there exists $U_i$ that contains $\phi(x)$ and $\phi(y)$. By the $(\lambda,\set{U_1,\dots, U_k})$-uniform expansion assumption, there exists $\gamma \in\Gamma$, such that
\begin{align}
\label{eqn:use_expansion_1}
   d_F(\rho_0(\gamma)\phi(x),\rho_0(\gamma)\phi(y))> \lambda d_F(\phi(x),\phi(y)).
\end{align}

We let $\gamma_1, \dots, \gamma_k$ be elements of $\Gamma$ such that $\gamma_i$ is $(\lambda, U_i)$-expansive.

Then there exists $i_1\in \{1,\dots, k\}$ such that $\phi(x)$ and $\phi(y)$ belong to $U_{i_1}$. By the $(\lambda,U_{i_1})$- expansion assumption,
\begin{align}
\label{eqn:use_expansion_2}
   d_F(\phi'(\rho(\gamma_{i_1}y)),\phi(\rho(\gamma_{i_1}x)))=d_F(\rho_0(\gamma_{i_1})\phi(x),\rho_0(\gamma_{i_1})\phi(y))> \lambda d_F(\phi(x),\phi(y)).
\end{align}
If $d_F(\phi(\rho(\gamma_{i_1}x),\phi(\rho(\gamma_{i_1}y)))< \eps$, then there exists $i_2\in \{1,\dots, k\}$ such that $\phi(\rho(\gamma_{i_1}x)), \phi(\rho(\gamma_{i_1}y)) \in U_{i_2}$. Applying the action by $\rho_0(\gamma_{i_2})$, we get
\begin{align}
   d_F(\rho_0(\gamma_{i_2})\phi(\rho(\gamma_{i_1}x),\rho_0(\gamma_{i_2})\phi(\rho(\gamma_{i_1})y))> \lambda d_F(\phi(\rho(\gamma_{i_1})x),\phi(\rho(\gamma_{i_1})y))>\lambda^2 d_F(\phi(x),\phi(y)).
\end{align}
If $d_F(\phi(\rho(\gamma_{i_2}\gamma_{i_1})x),\phi(\rho(\gamma_{i_2}\gamma_{i_1})y))< \eps$, we continue above process by induction. Since $\lambda>1$, there is $m\in \mathbb N$ such that \[d_F(\phi(\rho(\gamma_{i_m}\cdots\gamma_{i_1})x),\phi(\rho(\gamma_{i_m}\cdots\gamma_{i_1})y))< \eps,\] but \[d_F(\phi(\rho(\gamma_{i_{m+1}}\cdots\gamma_{i_1})x),\phi(\rho(\gamma_{i_{m+1}}\cdots\gamma_{i_1})y))\ge \eps.\]
On the other hand,
\[d_F(\phi'(\rho(\gamma_{i_{m+1}}\cdots\gamma_{i_1})y),\phi(\rho(\gamma_{i_{n+1}}\cdots\gamma_{i_1})y))=d_F(\phi(\rho(\gamma_{i_{m+1}}\cdots\gamma_{i_1})x),\phi(\rho(\gamma_{i_{m+1}}\cdots\gamma_{i_1})y))\ge \eps.\]
This contradicts with the assumption that $d_0(\phi,\phi') < \eps$.

Thus, we must have $\phi(x)=\phi(y)$, which finishes the proof. 
\end{proof}

\subsection{Reduction to Normal Subgroups}
\label{sec:extension_semiconjug}

 We will now apply the criteria from the previous subsection to obtain semi-conjugacies between actions of a group assuming existence of semi-conjugacies for suitable normal subgroups.

\begin{lemma}
\label{lem:torsion1}
Suppose $\Gamma_0< \Gamma$ is a normal subgroup and $\Gamma/\Gamma_0$ is finitely generated. Suppose that for any $\eps >0$ there is a neighborhood $U=U(\eps)$ of $\rho_0$ such that if $\rho \in U$, then either $\rho$ is not semi-conjugate to $\rho_0$ or there is exactly one $(\rho|_{\Gamma_0},\rho_0|_{\Gamma_0})$-semi-conjugacy in $\Bc_{\eps}(\id)$ that we denote by $\phi_\rho$.
Then there exists a small enough neighborhood $U_0$ of $\rho_0$ such that for any $\rho \in U_0$ such that $\rho|_{\Gamma_0}$ is semi-conjugate to $\rho_0|_{\Gamma_0}$, then $\phi_\rho$ is a $(\rho,\rho_0)$-semi-conjugacy (i.e. for the action of the whole group $\Gamma$).
\end{lemma}
\begin{proof}
Fix a finite generating set $\set{a_1 \dots, a_n}$ of $\Gamma/\Gamma_0$.  Fix an $\eps>0$ and let $U$ be a neighborhood satisfying the assumption. For $1 \leq i \leq n$, define $\phi_i:=\rho_0(a_i^{-1}) \circ \phi_{\rho} \circ \rho(a_i)$. It sufficies to show that we can find a smaller open subset $U_0\subset U$ such that for $\rho\in U_0$  for which $\rho|_{\Gamma_0}$ is semiconjugate to $\rho_0|_{\Gamma_0}$ by $\phi_\rho$, we have $\phi_i=\phi_\rho$ for $1\leq i \leq n$. 

Since $\Gamma_0$ is normal in $\Gamma$, it is easy to verify that $\phi_i$ is a $(\rho|_{\Gamma_0},\rho_0|_{\Gamma_0})$-semi-conjugacy. We claim that whenever $\rho$ is sufficiently close to $\rho_0$, then for all  $i=1,\dots,n$ we have,
\begin{align*}
d(\phi_i,\id)&\leq d(\rho_0(a_i^{-1}) \phi_{\rho}\rho(a_i),\rho(a_i)^{-1} \phi_{\rho} \rho(a_i))+d(\rho(a_i)^{-1}\phi_\rho \rho(a_i),\id) \\
 &\leq d(\rho(a_i^{-1}),\rho_0(a_i^{-1})) + d(\rho(a_i)^{-1}\phi_\rho \rho(a_i),\id)\\
 & <  \frac{\eps}{2}+\frac{\eps}{2}=\eps.
\end{align*}
Indeed, for the first term we can choose $U_0$ sufficiently small so that $\max_{1\leq i\leq n}d(\rho_0(a_i^{-1}),\rho(a_i^{-1}))< \eps/2$. The second $\frac{\eps}{2}$ estimate can be achieved by further shrinking $U_0$ as follows. We note that the conjugation action of $\Homeo(F)$ on itself is continuous. Hence there exists $\delta \ge 0$ such that 
\[d(\rho(a_i)^{-1}\varphi \rho(a_i),\id)<\frac{\eps}{2},\]
for every $\varphi\in B_\delta(\id)$. Now we let $U_0=U(\delta)$ be the neighborhood of $\rho_0$ obtained by applying the assumption to the constant $\delta$. It follows that $d(\rho(a_i)^{-1}\phi_\rho \rho(a_i),\id)<\frac{\eps}{2}$.

Thus $\phi_i \in \Bc_\eps(\id)$ is a $(\rho|_{\Gamma_0},\rho_0|_{\Gamma_0})$-semi-conjugacy. Then by the hypothesis on $U_0$, $\phi_i=\phi_{\rho}$. 
\end{proof}

\begin{corollary}\label{cor:torsionfree}
Suppose $\Gamma_0<\Gamma$ is a normal subgroup, $\Gamma/\Gamma_0$ is finitely generated and $\rho_0: \Gamma \to \Homeo(F)$ is such that $\rho_0|_{\Gamma_0}$ is a $(\lambda,\Uc)$-uniformly expanding action on $F$ for some open cover $\Uc$. Suppose that for any $\eps >0$ there is a neighborhood $U=U(\eps)$ of $\rho_0$ such that if $\rho \in U$ and $\rho|_{\Gamma_0}$ is  semi-conjugate to $\rho_0|_{\Gamma_0}$, then there is a $(\rho|_{\Gamma_0},\rho_0|_{\Gamma_0})$-semi-conjugacy in $\Bc_{\eps}(\id)$.
Then there exists $\eps_0>0$ and a neighborhood $U_0=U(\eps_0)$ of $\rho_0$ such that: if $\rho \in U_0$ and $\rho|_{\Gamma_0}$ is  semi-conjugate to $\rho_0|_{\Gamma_0}$, then $\rho$ is semi-conjugate to $\rho_0$. Moreover, there is a $(\rho,\rho_0)$-semi-conjugacy $\phi_{\rho}$ such that $d_0(\phi_\rho,\id)<\eps_0$.
\end{corollary}
\begin{proof}
Let $\eps$ be the constant obtained by applying \cref{lem:semi_unique} to the action $\rho_0|_{\Gamma_0}$ and the open cover $\Uc$. Set $\eps_0:=\eps/2$ and $U_0:= U(\eps_0)$. Suppose $\rho \in U_0$ is such that $\rho|_{\Gamma_0}$ is  semi-conjugate to $\rho_0|_{\Gamma_0}$. By assumption, there is a semi-conjugacy $\phi_\rho$ such that $d_0(\phi_\rho,\id)<\eps_0$. If $\phi'$ is another such semi-conjugacy in $\Bc_{\eps_0}(\id)$, then \cref{lem:semi_unique} implies that $\phi'=\phi_\rho$. So $\Bc_{\eps_0}(\id)$ contains exactly one semi-conjugacy. Then \cref{lem:torsion1} implies that $\phi_\rho$ is a $(\rho,\rho_0)$-semi-conjugacy. 
\end{proof}

\subsection{Upgrading semi-conjugacy to conjugacy}
\label{sec:lipschitz}

Let $\Lipeo(F)$ be the space of bi-Lipschitz homeomorphisms of $F$. The Lipschitz distance between $f \in \Lipeo(F)$ and $\id$ is defined by $d_{\Lip}(f,\id):=\sup_{x\in F} \dfrac{d_F(f(x),f(y))}{d_F(x,y)}$. We consider the distance function $d_L$ defined for $f,g \in \Lipeo(F)$ by 
\[
d_L(f,g)=d_{\Lip}(f\circ g^{-1},\id)+d_{\Lip}(g\circ f^{-1},\id).
\] 
If  $f$ is $\lambda$-Lipschitz, then $d_{\Lip}(g \circ f^{-1},\id) < \eps$ implies that $g$ is $\lambda \eps$-Lipschitz. Thus, if $f \in \Lipeo(F)$ is $\lambda$-biLipschitz, then every point in the $\eps$-ball around $f$ (in $d_L$ metric) is $(\lambda \eps)$-biLipschitz.

Suppose $\Gamma$ is a finitely generated group and fix a generating set $S$. We will say that two actions $\rho,\rho_0:\Gamma \to \Lipeo(F)$ are $\eps$-close if $\max_{s \in S} d_L(\rho(s),\rho_0(s)) <\eps$. The action $\rho_0$ is Lipschitz locally rigid if there exists $\eps>0$ such that any action $\rho: \Gamma \to \Lipeo(F)$ that is $\eps$-close to $\rho_0$, is $C^0$-conjugate to $\rho_0$.

We now prove a very general result about upgrading an equivariant  $C^0$ semi-conjugacy to a conjugacy when the actions have a uniform expansion property. 

\begin{proposition}\label{prop:removesemi}
Suppose $\rho_0: \Gamma \to \Lipeo(F) \subset\Homeo(F)$ is a $(\lambda, \Uc)$ uniformly expanding action. Then there exists $\eps_0>0$ and a neighborhood $V_0$ of $\rho_0$ in $\Lipeo(F)$ such that: if $\rho \in V_0$ and there is a $(\rho,\rho_0)$ semi-conjugacy $\phi_\rho$ in $\Bc_{\eps_0}(\id)$, then $\phi_\rho$ is a homeomorphism (i.e. a conjugacy).
\end{proposition}
\begin{proof}

 By compactness of $F$, we may extract a finite subcover $\set{U_i}_{i=1}^k$ of $\Uc$ with corresponding expanding group elements $\gamma_1,\dots,\gamma_k$. Let $\eps$ be the Lebesgue number of this finite cover (i.e. every ball of radius $\eps$ lies entirely in at least one $U_i$). Set $\eps_0:=\frac{\eps}{2}$. 
 
 Fix some $\lambda'$ such that $\lambda>\lambda' > 1$. Consider the neighborhood of $\rho_0$ in given by
 \begin{align*}
     V_0:=\set{ \rho \in \Hom(\Gamma, \Lipeo(F)) : \max_{1 \leq i \leq k} d_L(\rho(\gamma_i), \rho_0(\gamma_i)) < \frac{\lambda'}{\lambda}}
 \end{align*}
 Then for any $\rho \in V_0$, $\rho(\gamma_i)$ is $\lambda'$-expanding on $U_i$ for $1 \leq i \leq k$.

 We will show that the conclusion of this proposition is satisfied by this $V_0$ and $\eps_0$. Suppose $\rho \in V_0$ and $\phi_{\rho}$ is a $(\rho,\rho_0)$ semiconjugacy in $\Bc_{\eps_0}(\id)$. In order to prove this result, it suffices to show that $\phi_{\rho}$ is bijective.  
 
 We first show that $\phi_{\rho}$ is injective. Note that if $d_F(x,y) \geq \eps$, then $\phi_{\rho}(x) \neq \phi_{\rho}(y)$. Indeed, if $\phi_{\rho}(x)=\phi_{\rho}(y)$, then 
 \begin{align*}
 d_F(x,y) \leq d_F(x,\phi_{\rho}(x)) + d_F(\phi_{\rho}(y),y) < 2 \eps_0 = \eps,
 \end{align*}
 where the last inequality comes from $d_0(\phi_{\rho},\id) <\eps_0$.

Suppose there exist $\xi,\zeta\in F$ such that $\phi_{\rho}(\xi)=\phi_{\rho}(\zeta)$. Then $d_F(\xi,\zeta) < \eps$ and both points lie in some $U_i$.  Suppose \[
\frac{\eps}{\lambda'} \leq d_F(\xi,\zeta)<\eps
\]
As $\phi_{\rho}(\xi)=\phi_{\rho}(\zeta)$, 
\[
\phi_{\rho}(\rho(\gamma_i)(\xi))=\rho_0(\gamma_i)\phi_{\rho}(\xi)=\rho_0(\gamma_i)\phi_{\rho}(\zeta)=\phi_{\rho}(\rho(\gamma_i)(\zeta)). 
\]
However this is impossible as 
\[
d_F(\rho(\gamma_i) \xi,\rho(\gamma_i) \zeta) \geq \lambda' d_F(\xi,\zeta) \geq \lambda' \cdot \frac{\eps}{\lambda'}= \eps.
\]
 
 Continuing inductively as in the proof of \cref{lem:semi_unique}, we conclude that for any $n \geq 1$, $\phi_{\rho}(\xi) \neq \phi_{\rho}(\zeta)$ whenever  $\frac{\eps}{(\lambda')^n}<d_F(\xi,\zeta)<\frac{\eps}{(\lambda')^{n-1}}$. As $\lambda'>1$,  this implies that $\phi_{\rho}$ is injective. 
\end{proof}

\begin{corollary}
Suppose $\rho_0: \Gamma \to \Lipeo(F)$ is a $(\lambda, \Uc)$ uniformly expanding action. Further assume that for every $\eps>0$ sufficiently small, there exists a neighborhood $V = V(\eps)$ of $\rho_0$ in $\Lipeo(F)$ such that: if $\rho \in V$, then there exists a $(\rho,\rho_0)$ semi-conjugacy in $B_{\eps}(\id)$. Then $\rho_0$ is Lipschitz locally rigid.
\end{corollary}
\begin{proof}
Suppose $\eps_0$ and $V_0$ are as in the previous proposition. Let $0< \eps < \eps_0$ and $V= V(\eps)$ be as in the hypothesis. Then $V \cap V_0$ is a neighborhood of $\rho_0$ such that: whenever $\rho \in V \cap V_0$, $\rho$ is semi-conjugate to $\rho_0$. Indeed, the hypothesis implies the existence of a $(\rho,\rho_0)$ semi-conjugacy in $\Bc_{\eps}(\id) \subset \Bc_{\eps_0}(\id)$. The previous proposition then implies that $\phi_{\rho}$ is a conjugacy. This finishes the proof. 
\end{proof}

\subsection{A Generalized Denjoy Construction}
\label{sec:denjoy}

We remark that there are many actions which are semiconjugate but not conjugate. This had been understood by Bowden and Mann in their work on boundary actions on spheres - cf. Proposition 4.1 of \cite{BowdenMann19} which generalizes  
 the Denjoy construction 
(\cite{Denjoy}).  Using their main argument, we provide a simple generalization that works for any countable group acting on a manifold with a dense orbit.

We will first need the following \cref{prop:Carpet}. It is a variant of \cite[Proposition 4.1 ]{BowdenMann19} and is a consequence of \cite[Proposition 2.2]{DeSouza}. We will first recall \cite[Proposition]{DeSouza}. Let $M^n_{n-1}$ denote the Sierpinski space of dimension $n-1$, i.e. $M^n_{n-1}:=S^n - \cup_{i \in \Nb} U_i$ where $\set{U_i: i \in \Nb}$ is a family of open balls dense in $S^n$ and $\set{i \in \Nb: \diam(U_i)<\eps}$ is finite for any $\eps>0$. If $A \subset S^n$ is finite, let $M^n_{n-1}/\sim_A$ be obtained by collapsing each $\partial U_i$ to a single point for which $U_i$ does not intersect $A$. Proposition 2.2 of \cite{DeSouza} states that $M^n_{n-1}/\sim_A$ is homeomorphic to $S^n$ minus a set of $~\#A$ disjoint open tame balls. Here an open ball $U \subset S^n$ is tame if both $\overline{U}$ and $S^n-U$ are homeomorphic to closed balls.

Observe that Proposition 2.2 of \cite{DeSouza} applies equally well in all dimensions $n$, including dimension $n=4$. This is because the proof of \cite[Proposition 2.2]{DeSouza} follows from the Decomposition Theorem (see \cite[Proposition 1.11]{DeSouza}) which in turn follows from Quinn's annulus theorem and both of these results work in any dimension. Hence our \cref{prop:Carpet} also works in all dimensions including $n=4$.

\begin{proposition}\label{prop:Carpet}
Given a countable dense subset $Z$ of an open $n$-disk $B^n$ there exists a surjective, continuous map $h:\bar{B^n}\to \bar{B^n}$ that is identity on the boundary $\partial (B^n)$, injective on the complement of $h^{-1}(Z)$, and such that the pre-image of each point in $Z$ is a closed ball.
\end{proposition}

\begin{proof}
Let $B^n \cup \set{z_0}$ be the one-point compactification of $B^n$ and thus is homeomorphic to $S^n$. We also enumerate $Z=\set{z_1,z_2,\dots}$. 

We first observe that the proof of Proposition 4.1 of \cite{BowdenMann19} applies to any countable set, and not just an orbit.  Then we apply this to the countable set $Z'=Z\cup\set{z_0} \subset S^n$ to obtain a continuous surjective map $f:S^n\to S^n$ which is injective on the complement of $f^{-1}(Z')$ and such that $\set{C_i:=f^{-1}(z_i): z_i \in Z'}$ is a family of disjoint closed balls. As shown in the proof of \cite[Proposition 4.1]{BowdenMann19}, the space $S^n -\bigcup_{z' \in Z'}{\rm int} \left({f}^{-1}(z') \right)$ is a Sierpinski (or Menger) space of dimension $n-1$, denoted by $M^{n}_{n-1}$.

Let $A=\set{z_0}$ and apply \cite[Proposition 2.2]{DeSouza}; see the remarks before the statement of this proposition. Consider the quotient map $q_A: M^n_{n-1} \to M^n_{n-1}/\sim_A$ where $M^{n}_{n-1}/\sim_A$ is obtained by collapsing each $\partial C_i$ to a single point for $i \geq 1$. By the proposition, $M^n_{n-1}/\sim_A$ is homeomorphic to the complement of a tame open ball in  $S^n$, and hence is a closed ball $\bar{B^n}$ and $q_A$ is injective on $M^n_{n-1}-(\cup_{ i \geq 1}\partial C_i)$. Consider the map $h': M^{n}_{n-1} \cup (\cup_{i\geq 1}  C_i) \to M^n_{n-1}/\sim_A$ such that $h'|_{M^n_{n-1}}=q_A$ and $h'|_{C_i}$ collapses $C_i$ to the point $q_A(\partial C_i)$.  Now observe that both $M^{n}_{n-1} \cup (\cup_{i\geq 1}  C_i)$ and $M^n_{n-1}/\sim_A$ are homemorphic to $\bar{B^n}$. Then, by abuse of notation, we can think of $h'$ as a map from $\bar{B^n}$ to itself that is continuous, surjective, is injective on the complement of $h'^{-1}(Z)$ and $(h')^{-1}(z)$ is a closed ball for any $z \in Z$.

Consider $h'|_{\partial \bar{B^n}}$ which is a homeomorphism as $Z \cap \partial \bar{B^n} = \varnothing$. Extend $(h'|_{\partial \bar{B^n}})^{-1}$   to a homeomorphism  $h''$ of $\bar{B^n}$, e.g. by coning it off. 
Then let $h:=h' \circ h''$. 
\end{proof}

This allows us to give a very general construction of non-conjugate perturbations.  

\begin{construction}
\label{construction:non_conjugate_perturbation}
Consider any $C^1$-action $\rho_0$ of a countable group $\Gamma$ on a compact smooth manifold $M$. If $\rho_0(\Gamma)$ has a dense orbit, then there exists a $C^0$ action $\rho$ of $\Gamma$ on $M$ that is not $C^0$ conjugate to $\rho _0$. Moreover, $\rho$ can be chosen to be arbitrarily $C^0$ close to $\rho_0$.
\end{construction}
\begin{proof} 
 Suppose that $\Gamma \cdot x$ is a dense orbit. We will then form a new $C^0$-action $\rho$ by first expanding a portion of the orbit $\Gamma\cdot x$ into a small disjoint union of closed balls $\bigcup_{\gamma\in\Gamma} \bar{B}(\gamma x)$.  As we will explain below, the radii of the closed balls $\bar{B}(\gamma x)$ can be chosen to decay in such a way that the resulting manifold $M'$ is still homeomorphic to $M$.

We start by fixing a Riemannian metric on $M$ and choosing one closed ball $\bar{B}(x_0,r_0)\subset M$ for some $0<r_0<{\rm injrad}(M)$ and with $x_0\in \Gamma \cdot x$ whose boundary misses the orbit $\Gamma \cdot x$. We can do this since there are an uncountable number of choices for the radius $r_0$, but only a countable number of orbit points. Proceeding inductively, having chosen disjoint balls $\bar{B}(x_0,r_0), \bar{B}(x_1,r_1),\dots,\bar{B}(x_n,r_n)$ we can choose a closed ball $\bar{B}(x_{n+1},r_{n+1})$ with $0<r_{n+1}<{\rm injrad}(M)$ contained in the open set $M\setminus \bigcup_{i=0}^{n}\bar{B}(x_i,r_i)$ centered at some $x_{n+1}\in \Gamma \cdot x$ whose boundary misses $\Gamma \cdot x$. By induction we can find a countable number of balls such that $M\setminus \bigcup_{i\geq 0} \bar{B}(x_i,r_i)$ has no interior and contains no orbit point. Now for each closed ball $\bar{B}(x_i,r_i)$ we apply Proposition \ref{prop:Carpet} to obtain a surjective continuous map $h_i:\bar{B}_i\to \bar{B}(x_i,r_i)$ from a smooth closed ball $\bar{B}_i$ such that the pre-image of each orbit point $\gamma x\in \bar{B}(x_i,r_i)$ is a closed ball $\bar{B}(\gamma x)$ and is injective on the complement of these disks. As the $\bar{B}(x_i,r_i)$ are all disjoint and $h_i$ restricts to a homeomorphism on $\bar{B}_i$, these maps can then be combined together to an overall surjective continuous map $h:M'\to M$ where $M'$ is formed by removing each $\bar{B}(x_i,r_i)$ from $M$ and gluing in $\bar{B}_i=h_i^{-1}(\bar{B}(x_i\rest{\partial \bar{B}_i},r_i))$ using $h_i\rest{\partial \bar{B}_i}$. The map $h$ is injective on the complement of $h^{-1}(\Gamma \cdot x)$. Clearly $M'$ is homeomorphic to $M$ because the map $f':M' \to M$ defined by $f'|_{B_i}=h_i$ and $f'|_{M'-\cup_{i}B_i}=\id$ is a homeomorphism.

Since the action of $\rho_0$ is $C^1$, it is locally linear around each orbit point. We can extend each $C^1$ diffeomorphism $\rho_0(\gamma')$ to $M'-\bigcup_{\gamma\in\Gamma} \bar{B}(\gamma x)\cong (M-\Gamma\cdot x)\bigcup_{\gamma\in\Gamma}\partial \bar{B}(\gamma x)$ by taking for each $\gamma$ and $v\in \partial \bar{B}(\gamma x)\cong S^{n-1}$ the limit of $\rho_0(\gamma')z^v_i$ along a sequence of points $z^v_i$ which approach $\gamma x$ radially (with respect to some $C^1$ chart) along the direction $v$. (Here we are assuming that the $\eps_\gamma$ are chosen sufficiently small.) This sequence will necessarily tangentially approach $\gamma'\gamma x$ along a radial direction $w$ which we identify with a unique point in $\partial \bar{B}(\gamma'\gamma x)$. This will produce a homeomorphism $\partial \bar{B}(\gamma x)\to \partial \bar{B}(\gamma'\gamma x)$. Lastly, we extend this homeomorphism on $\partial \bar{B}(\gamma x)$ to all of $\bar{B}(\gamma x,\eps_\gamma)$ by coning to $\gamma x$ to produce $\rho(\gamma'):\bar{B}(\gamma x)\to \bar{B}(\gamma'\gamma x)$. The compatibility of these homeomorphisms follows from the fact that $\rho_0$ is a $C^1$ action on $M$, and thus we obtain an action $\rho$ on $M\cong M'$. (See the proof of Proposition 4.1 of \cite{BowdenMann19} for details:  While they consider only  actions on $S^n$,  their arguments apply in our generality.)
The desired semiconjugacy collapses each $\bar{B}(\gamma x,\eps_\gamma)$ to $\gamma x$. Note that $\rho$  cannot have a dense orbit by construction, and hence cannot be $C^0$-conjugate to $\rho_0$. 
\end{proof}

\section{Preliminaries III: Symmetric Spaces and Furstenberg Boundaries}
\label{sec:Sym}

Here we review some well-known facts about symmetric spaces and their various boundaries. These can be found in \cite[Section 2.7]{Eberlein96}.

Let $X$ be a symmetric space of non-compact type without Euclidean factors and $\Isom(X)^0=G$ a semisimple Lie group. For any point $p \in X$, $K:={\rm Stab}_G(p)$ is a maximal compact subgroup of $G$. The orbit of the point $p$ induces a  natural diffeomorphism of $X$ with $G/K$. Moreover, $X$ carries a left $G$-invariant metric of non-positive curvature. We denote by $\partial X$ the (smooth) geodesic boundary of $X$ with respect to any such metric. The diffeomorphism classes of $X$ and $\partial X$ are independent of the choice of such metric.

\begin{definition}
A subgroup $Q<G$ is called a parabolic subgroup if there exists $ \bdx \in \partial X$ such that $Q=G_{\bdx}:=\{g\in G: g\cdot \bdx= \bdx\}$.
\end{definition}
We note that $G_\bdx$ acts transitively on the symmetric space $X$.

Fix $o\in X$ and $\bdx \in \partial X$. Let $\mathfrak{g}=\mathfrak{k} + \mathfrak{p}$ be the Cartan decomposition of the Lie algebra $\mf{g}$ determined by $o$. We define the {\em (real) rank of $G$}, $\op{rk}(G)$, to be the dimension of any maximal abelian subalgebra of $\mf{p}$. We fix a maximal abelian subalgebra of $\mf{p}$ and label it $\mf{a}$.

Note that $\mathfrak{p}$ can be identified with the tangent space to $X$ at $o$.  Then  let $v\in \mathfrak p$ be given by  $v:=\gamma_{o \bdx}'(0)$ where $\gamma_{o\bdx}$ is the geodesic ray from $o$ to $\bdx$.  Consider  the homomorphism $T_{\bdx}:G_{\bdx}\to G$ defined by 
\[
T_{\bdx}(g):=\lim_{t\to +\infty}\exp(-tv)g \exp(tv),
\]
where this limit exists by \cite[Proposition 2.17.3]{Eberlein96}. 

Let $Z(v):=\{w\in \mathfrak g:[v,w]=0\}$ be the centralizer of $v$ in the Lie algebra $\mathfrak{g}$. We then define the following subsets:
\begin{align}
\label{eqn:define_subgroups}
Z_{\bdx}&:=\exp(Z(v))=\{g\in G:\exp(-tv)g \exp(tv)=g \text{ for all } t\in \mathbb R\} \nonumber \\ 
M_{\bdx}&:=K\cap Z_{\bdx}=K\cap G_{\bdx} \nonumber \\
N_{\bdx}&:=\ker(T_{\bdx})\\ 
A_{\bdx}&:=\exp\{Z(v)\cap \mathfrak p\}, \nonumber
\end{align}
We note that $Z_{\bdx}$, $M_{\bdx}$, and $N_{\bdx}$ are subgroups of $G$, while $A_{\bdx}$ is generally only a subset (unless $Z(v) \subset \mathfrak{p}$, i.e. $v$ is a regular vector).  Setting $Q=G_\bdx$, we will also use the notation $Z_Q, M_Q, N_Q, A_Q$ to denote $Z_{\bdx}, M_{\bdx}, N_{\bdx}, A_{\bdx}$ respectively.

In this section, we list some facts about the geometry associated to a parabolic subgroup $Q=G_{\bdx}$.

\subsection{Generalized Iwasawa Decomposition} 
\label{sec:iwasawa}
From \cite[Proposition 2.17.5(5)]{Eberlein96}, we have the Generalized Iwasawa decomposition
\[
G=K \cdot A_Q \cdot N_Q.
\]
The indicated decomposition of elements of $G$ is unique. 

By \cite[Proposition 2.17.5(4)]{Eberlein96}, there is an analogous decomposition of $Q$:
\begin{align}
\label{eqn:Q_decomposition}
Q=Z_Q \cdot N_Q = M_Q \cdot A_Q \cdot N_Q = N_Q \cdot A_Q \cdot M_Q=N_Q \cdot Z_Q.
\end{align}
Again, the indicated decomposition of elements of $Q$ is unique.

\subsection{Parallel Sets}
\label{sec:parallel_sets} 

The reference for this subsection is  \cite[Pg 138]{Eberlein96}. By slight abuse of notation, let $\gamma_{o\bdx}$ also denote the bi-infinite geodesic extension of the geodesic ray from $o$ to $\bdx$. 

The {\em parallel set of $\gamma_{o\bdx}$}, denoted by $F(\gamma_{o\bdx})$, is the union of all bi-infinite geodesics parallel to $\gamma_{o\bdx}$. We will also use the notation $F(v)$ for the union of all geodesics parallel to the geodesic with initial vector $v$. In particular, $F(\gamma_{o\bdx})=F(\gamma_{o\bdx}'(0))$. The set $F(\gamma_{o\bdx})$ is a totally geodesic submanifold of $X$ and 
\[F(\gamma_{o\bdx})=A_{\bdx} \cdot o.\]

In fact, we also have $F(\gamma_{o \bdx})= Z_{\bdx} \cdot o.$ %\todo{verify if both these equalities are correct.} 

 Let $E(\gamma_{o\bdx})$ be the intersection of all maximal flats of $X$ that contain $\gamma_{o\bdx}$.  Define
     \[
     E_{\bdx}:=\bigcap \left\{ \mathfrak{u} \subset \mathfrak{p} : \mathfrak{u} \text{ is an abelian subspace containing } v \right\}.
     \] 
    Then \[E(\gamma_{o\bdx})=\exp(E_{\bdx})\cdot o \] and $E(\gamma_{o\bdx})$ a totally geodesic submanifold that is isometric to $\mathbb R^d$, where $1 \leq d = \dim(E_\bdx) \leq {\rm rk}(G)$.
     
    The set $F(\gamma_{o \bdx})$ isometrically splits as $\mathbb R^d\times CS(\gamma_{o\bdx})$ where $d=\dim E_{\bdx}$  and $CS(\gamma_{o\bdx})$ is a symmetric space of rank $({\rm rk}(G)-d)$ \cite[Lemma 2.20.9]{Eberlein96}.

   The parallel sets produce a totally geodesic foliation of the symmetric space with isometric leaves. More precisely, $X$ is foliated by totally geodesic subspaces $\set{F(\gamma_{q\bdx}): q \in X}$. The leaf space of this foliation can be naturally identified with $N_{\bdx}$ (since $F(\gamma_{o \bdx})=Z_{\bdx}\cdot o$ and $G_{\bdx}=N_{\bdx} Z_{\bdx}$). Given $o$ and $q$ in $X$, there is a unique isometry $H_{o,q}:F(\gamma_{o\bdx})\to F(\gamma_{q\bdx})$ such that for any sequence $(o_n)$ of points in $F(\gamma_{o\bdx})$ that tends to $\bdx$, the distance $d(o_n,H_{o,q}(o_n))$ tends to 0.
   
  Using the above foliation, we can define a retraction map $retr: X\to F(\gamma_{o\bdx})$ such that 
  \[retr|_{F(\gamma_{q\bdx})}=H_{q,o}\] for every $q\in X$. This retraction map is distance non-expanding. Furthermore, $retr^{-1}(o)=N_{\bdx}\cdot o$. In particular, the symmetric space $X$ is  diffeomorphic to the product $F(\gamma_{o\bdx})N_{\bdx}$.

\subsection{Weyl chamber $Q$-faces at infinity} 
\label{sec:Weyl_face_at_infinity}

The Weyl chamber face at infinity associated to $\bdx$ is defined as $\mathcal C(\bdx)=\{\bdy\in \partial_\infty X: G_\bdy^0=G_{\bdx}^0\}$.  
    %(We remark that the Weyl chambers associated to $\bdx$ are sometimes referred to in other sources as "the Weyl chamber {\em faces} associated to $\bdx$" when $G_\bdx$ is not a minimal parabolic.)
   
    We will say that two Weyl chamber $Q$-faces at infinity, $\mc{C}(\bdx)$ and $\mc{C}(\bdy)$, are of {\emph same type} if $\mathcal C(\bdy)= g \mathcal C(\bdx)$ for some $g \in G$. Then $G$ acts transitively on the set of Weyl chamber $Q$-faces at infinity of the same type with the stabilizer of $\mc{C}(\bdx)$ being equal to $G_{\bdx}=Q$. Thus the set of all Weyl chambers $Q$-faces at infinity of the same type as $\Cc(\bdx)$ can be identified with $G/Q$ and inherits the structure of a smooth manifold. Moreover, $G/Q$ is diffeomorphic with $K/M_Q$ by the generalized Iwasawa decomposition. %The space $G/Q$ may also be identified with the $G$ orbit $G\cdot \bdx\subset \partial X$, where $G \cdot \bdx$ inherits the structure of a smooth submanifold of $\partial X$ (\cite[Theorem 1.2 Pg 302]{Bredon}). 
    %since $K$ acts transitively on $G\cdot \bdx$ with point stabilizer $M_Q<K$.
    Using this identification, we may write the Weyl chamber $Q$-face at infinity $\Cc(\bdx)$ as simply $\xi$ whenever $Q_\bdx$ is the stabilizer of $\xi$ for some $\xi\in G/Q$.
    
\subsection{Weyl chamber $Q$-faces based at a point}\label{sec:Weyl_face}
The Weyl chamber face associated to $\bdx$ based at $o$ is defined as $\mc C_o(\bdx)=\exp\{tY: t\geq 0, Y(o)=\gamma_{o\bdy}'(0), \bdy\in \mathcal C(\bdx)\}\cdot o$. Observe that $\mc C_o(\bdx)$ consists of the union of geodesic rays starting at $o$ which point to the corresponding Weyl chamber $Q$-face at infinity $\mc C(\bdx)$. The stabilizer of $\mc \Cc_o(\bdx)$ in $G$ is $M_{\bdx}=M_Q$. 
%The set $\Cc_p(\bdx)$ has a transitive $G$-action and the stabilizers

Given a Weyl chamber $Q$-face at infinity $\xi\in G/Q$, we let $\Cc_o(\xi)$ be the unique Weyl chamber $Q$-face base at $o$ that belongs to the equivalent class $\xi$. Note that if $\bdx \in \partial_\infty X$ is such that $Q=G_{\bdx}$, then $\Cc_o(\xi)=\Cc_o(\bdx)$.

\subsection{Opposite Weyl chamber $Q$-faces}
\label{sec:opp_weyl_faces}

Given a Weyl chamber $Q$-face based at $o$, say $\Cc_o(\xi)$,  we let $\Cc_o(\xi^{*})$ denote the opposite Weyl chamber $Q$-face based at $o$. Explicitly this face is given by $\Cc_o(\xi^{*})=\exp\{tY: t\geq 0, Y=-\gamma_{o\bdy}'(0), \bdy\in \mathcal C(\xi)\}\cdot o$. 

Similarly, we will say that a Weyl chamber $Q$-face at infinity $\xi\in G/Q$ is an opposite Weyl chamber $Q$-face to $\eta\in G/Q$ if there exists a point $o$ for which $\Cc_o(\xi)$ is opposite to $\Cc_o(\eta)$.

\subsection{The Weyl chamber $Q$-face bundle $G/M_Q$} 
\label{sec:G/M_Q}
Suppose $Q$ is a parabolic subgroup with compact subgroup $M_Q$ as in \cref{eqn:define_subgroups}. The Weyl chamber $Q$-face bundle on $X$ is the is the space of Weyl chamber faces associated to $Q$   based  at the (any) point $o\in X$. Observe that $G$ acts transitively on this space. In other words, it can be identified as $G\cdot \mathcal C_o(\bdx)\cong G/M_Q$.

We spend the rest of this section discussing the geometry of the space 
\[X_{G/Q}=G/M_Q.\]
The natural map $p_Q: G/M_Q \to G/K$ gives $G/M_Q$ the structure of a smooth fiber bundle over $G/K$ with fiber $K/M_Q$. The fiber over the point $gK \in G/K$ is given by $F_{gK}=g K/M_Q$, and is identified with the space of all Weyl chamber $Q$-faces based at $gK$. We now explain the map $\pi_{gK}:F_{gK} \to K/M_Q$ that identifies each fiber with $K/M_Q$. For every $g\in G$, there is a unique decomposition $g=n_Qa_Qk$ coming from the generalized Iwasawa decomposition $G\cong N_Q A_Q K$; see \cref{sec:iwasawa}. Then define the map $\pi_{gK}:gK/M_Q\to K/M_Q$ by \[\pi_{gK}(gk'M_Q)=kk'M_Q.\]

Finally, the fiber bundle is globally trivial since $G/K$ is contractible. But the generalized Iwasawa decomposition (\cref{sec:iwasawa}) gives an explicit description of a global trivialization. Notice that there is a natural smooth bundle isomorphism (and diffeomorphism) $\Phi: G/M_Q \to G/K \times K/M_Q \cong G/K \times G/Q$ arising from the smoothness of the (reverse) Iwasawa decomposition $G\cong N_QA_QK$ and the resulting diffeomorphism $G/K\cong N_QA_Q$. In particular, if $g=n_Qa_Qk=k'a'_Qn'_Q$ by the generalized Iwasawa decomposition, then 
%\[\Phi(gM_Q)=(n_Qa_QK,k'M_Q).\] 
\[\Phi(gM_Q)=(n_Qa_QK,k'M_Q).\]
Note that the map $g \mapsto k'M_Q$ coming from the generalized Iwasawa decomposition is well-defined. This is because $M_Q A_QN_Q =A_QN_QM_Q$, see \cref{eqn:Q_decomposition}. Hence the map $\Phi$ is well-defined. Further $\Phi$ is clearly a smooth function. Also observe that this definition of $\Phi$ agrees with the earlier definition \cref{eqn:phi-and-phi0} in \cref{lem:G/Q_standard_action}.

\subsection{Riemannian metric on $G/M_Q$}\label{subsec:metric-on-G/MQ} Fixing a point $o=e K \in  X$, let us consider the corresponding  Cartan decomposition 
\begin{align}
\label{eqn:k+p}
    \mathfrak{g}=\mathfrak{k}\oplus \mathfrak{p},
\end{align}
where $\mathfrak{k}=T_eK$. Let us choose a maximal abelian subalgebra $\mf{a}$ in $\mf{p}$. Fix some choice of positive roots $\Lambda^+$ for $\mathfrak{a}$ and let $\Lambda^{++}$ denote the positive simple roots. Then fix a \emph{positive Weyl chamber}  $\mathfrak{a}^+ \subset \mathfrak{p}$ (with respect to $\Lambda^{++}$) based at $o$.

Suppose $\alpha$ is any left $G$-invariant and $K$-bi-invariant metric on $G$. This induces a left $G$-invariant metric on $G/M_Q$. We will always equip $G/M_Q$ with such a metric.

We remark that for any $\alpha$ as above, $\mathfrak{p}=(T_eK)^{\perp}$.  Indeed, $\alpha$ is an $\Ad(K)$ invariant metric on $\mathfrak{g}$ and $\Ad(K)$ preserves the Cartan decomposition. Moreover, there exists $k_0 \in K$ such that $\Ad(k_0)(\mathfrak{a}^+)=-\mathfrak{a}^+$ (i.e. the opposition involution). Let $ v\in \mathfrak{a}^+$ be the barycenter of the unit vectors in $\mathfrak{a}^+$. Then $\Ad(k_0)v=-v$. Thus, for any $v_1 \in \mathfrak{k}$, 
\begin{align*}
\alpha(v_1,v)=\alpha(\Ad(k_0)v_1,\Ad(k_0)v)=\alpha(v_1,-v), 
\end{align*}
which implies that $v \in \mathfrak{k}^{\perp}$. Then, $\Ad(K)$-invariance of $\alpha$ implies that $\Ad(K)v \subset \mathfrak{k}^{\perp}$. Finally $\Ad(K)v$ spans $\mathfrak{p}$. The comparison of dimensions of $\mathfrak{p}$ and $\mathfrak{k}$ proves that $\mathfrak{p}=\mathfrak{k}^{\perp}$. 

 Furthermore, we also observe that such a Riemannian metric $\alpha$ also makes all simple factors of $T_eG$ orthogonal to each other. In fact, a Riemannian metric $\alpha$ as above is unique up to scaling on each simple component of the Lie algebra (see Section 2.3 of \cite{Eberlein96}).

\begin{lemma} 
Let $Q$ be any parabolic subgroup, and suppose $\alpha$ is a Riemannian metric on $G/M_Q$ as above. Setting $\Ec_g=gA_QN_Q \cdot M_Q \subset G/M_Q$, then:
\begin{enumerate}
    \item each fiber $F_{g'}$ is orthogonal to $\Ec_g$, and 
    \item $\Ec_g$ is isometric to the Riemannian globally symmetric space $(G/K, (p_{Q})_* \alpha)$.
    \item each fiber $F_g$ is totally geodesic.
\end{enumerate}

\end{lemma}
\begin{proof}
By left $G$-invariance and right $K$-invariance of $\alpha$, $dp_{Q} \alpha$ is a left $G$-invariant Riemannian metric on $G/K$, hence it endows $G/K$ with the structure of a Riemannian globally symmetric space. (Indeed, by Ad(K) invariance on $T_e(G/K)$, the metric $dp_Q \alpha$  differs from a fixed Riemannian symmetric metric only by re-scaling corresponding to the simple factors of $G$). 
Recall that $\Ec_g$ is the image of a section of the bundle $p_{Q}:G/M_Q \to G/K$. Then left multiplication by $g^{-1}$ is an isometry between $\Ec_g$ and $\Ec_e:=A_QN_Q \cdot M_Q$. By left $G$-invariance of $\alpha$, it is enough to show that:
\begin{itemize}
    \item $T_{eM_Q}(K/M_Q)$ is orthogonal to $T_{e}\Ec_e$,
    \item $dp_{Q}: T_{eM_Q}(G/M_Q) \to T_{eK}(G/K)$ is a Riemannian submersion,
    \item the restriction $d p_{Q}|_{T_e\Ec_e}: T_e\Ec_e \to T_e(G/K)$ is a Riemannian isometry.
\end{itemize}
Note that $T_{e}\Ec_e$ is canonically identified with $\mathfrak{p}=T_e(G/K)$, where $\mathfrak{p}$ is as in \cref{eqn:k+p}. By definition of the metric $\alpha$, $(T_{eM_Q}(K/M_Q))^{\perp}$ is canonically identified with $T_e(G/K)$. This proves the first part. Next note that $d p_{Q}$ is clearly a submersion. Moreover,  
\[\ker (d p_{Q})^{\perp}=(T_{eM_Q}(K/M_Q))^{\perp}\] which is equal to $T_e \Ec_e=T_e(G/K)$ and $(p_{Q})_*\alpha$ is an isometry between $\ker(d p_{Q})^{\perp}$ and $T_e(G/K)$. This proves both the first and the second parts.

For proving the third part, it is enough to show that $K$ is totally geodesic in $G$. But this is clear since the subgroup $K$ is invariant under the Cartan involution on $G$ induced by the Riemannian geodesic symmetry of $G/K$ at $eK$.
\end{proof}
\vspace{1em}

\subsection{Weyl Chamber flow on $G/M_Q$}   Let $A$ be the connected subgroup of $G$ whose Lie algebra is $\mathfrak{a}$. Consider the group $\flow<A$, the subgroup of $A$ consisting of elements commuting with $M_Q$. There is a natural action of $\flow$ on $G/M_Q$ defined by right multiplication, i.e. $g M_Q \mapsto gaM_Q$ for every $a\in \flow$ and $g\in G$. It is well-defined because $\flow$ commutes with $M_Q$. This action is called the {\em Weyl chamber flow} associated to the parabolic subgroup $Q$.

Recall that $\mf{a}^+$ is the positive Weyl chamber with respect to a choice of positive simple roots, see \cref{subsec:metric-on-G/MQ}. We can intersect the Lie algebra of $A_Q'$ with $\mf{a}^+$ to obtain the positive Weyl chamber face $\mf{a}^+_Q$. We define the \emph{positive Weyl chamber flow} on $G/M_Q$ to be the restricted action of ${A'_Q}^{+}:=\exp(\mf{a}^+_Q)\subset A_Q'$.
\vspace{1em}

\subsection{Foliations on $G/M_Q$} 
\label{sec:foliations_weyl_chamber_bundle} 
Recall the notion of Weyl chamber $Q$-faces (based at a point and at infinity) as well as the notion of opposite Weyl chambers from Sections \ref{sec:Weyl_face_at_infinity}, \ref{sec:Weyl_face} and \ref{sec:opp_weyl_faces}.

\subsubsection*{Center stable and center unstable foliations} We will call two based Weyl chamber $Q$-faces, $\Cc_p(\xi)$ and $\Cc_q(\eta)$, asymptotic if their Hausdorff distance is bounded, or equivalently, if $\xi=\eta$. Then, an equivalence class of asymptotic based Weyl chamber $Q$-faces is a Weyl chamber $Q$-face at infinity.

For $\xi \in G/Q$, we define the \emph{center stable manifold} in $G/M_Q$ as
    \begin{align*}
    \wcs(\xi)&:=\{\Cc_p(\bdx):p\in X\}.
    \end{align*}
    % \begin{align*}
    % \wcs(\xi)&:=\{\Phi^{-1} (p,\xi): p\in X\}=\{C_p(\bdx):p\in X\}, 
    % \end{align*}
It is the set of all based Weyl chamber $Q$-faces asymptotic to $\Cc_o(\xi)$. In particular, is a smooth manifold as it is the orbit of a closed subgroup of $G$, by definition. This coincides with the dynamical definition that the center stable manifold of a point $gM_Q$ in $G/M_Q$ consists of all points that maintain a bounded distance under the positive Weyl chamber flow  to the positive Weyl chamber given by the $A_Q^{'+}$ orbit of $gM_Q$.

The \emph{center unstable manifold} is defined as:
  \begin{align*}
    \wcu(\xi)&:=  \wcs(\xi^*),
  \end{align*}
where $\xi^*$ is the Weyl chamber $Q$-face at infinity that is opposite to $\xi$; see \cref{sec:opp_weyl_faces}. 

Then $\set{\wcs(\xi): \xi \in G/Q}$ and $\set{\wcu(\xi): \xi \in G/Q}$ are two different continuous foliations of $G/M_Q$ with smooth leaves. The leaves of these foliations intersect along orbits of the Weyl chamber flow. Indeed, $\wcs(\xi) \cap \wcu(\eta)=\set{\Cc_p(\xi): p \in X, \Cc_p(\xi^*) \in \wcs(\eta^*)}$ is the set of Weyl chamber $Q$-faces that are forward asymptotic to $\xi$ and backward asymptotic to $\eta$. Under the map $p_Q: G/M_Q \to G/K$, this set is diffeomorphic to $F(\gamma)$, where $\gamma$ is a (any) geodesic with end-points $\gamma(+\infty) \in \xi$ and $\gamma(-\infty) \in \eta$. Thus the intersection is either empty or a smooth submanifold homeomorphic to  $G_Q/K_Q \cong A_Q N_Q$; see \cref{sec:parallel_sets}.

Let $\Phi:G/M_Q \to G/K \times G/Q$ be the map defined as above in \Cref{sec:G/M_Q}. Then $\Phi^{-1}(p,\xi)=\Cc_p(\xi)$ and $\Phi^{-1}(X \times \xi)=\wcs(\xi)$ for any $\xi \in G/Q$. 

\subsubsection*{Finer foliations} Each leaf $\wcs(\xi)$ admits finer foliations. Indeed, there is a natural homeomorphism 
\[
\wcs(\xi) \to F(\gamma_{p\bdx})N_\bdx
\] 
where $F(\gamma_{p\bdx})$ is identified with a center leaf in $\wcs(\xi)$, and the unipotent subgroup $N_\bdx$ is identified with the strong unstable leaf of $(x,\xi)$. The foliation of $\wcs(\xi)$ by $\{F(\gamma_{q\bdx}):q\in X\}$ is perpendicular to the foliation by orbits of $N_\bdx$ with respect to the metric defined earlier on $G/M_Q$. (This follows from the fact that $K,A_\bdx$ and $N_\bdx$ are all perpendicular with respect to $\alpha$.) A similar foliation can be defined for each leaf of $\wcu(\xi)$.

\subsection{Example}
Consider the symmetric space $X=\Hb^2 \times \Hb^2$ and let $G=\PSL_2(\Rb) \times \PSL_2(\Rb)$ be the connected component of the isometry group of $X$.

We consider the following standard subgroups of $\PSL_2(\Rb)$: $A_*=\set{\begin{bmatrix} e^t & 0 \\ 0 & e^{-t} \end{bmatrix} : t \in \Rb}$, $N_*=\set{\begin{bmatrix} 1 & s \\ 0 & 1 \end{bmatrix}: s \in \Rb}$, $M_*=\set{e}$, and $P_*=M_* A_* N_*=\set{\begin{bmatrix} e^{t} & s \\ 0 & e^{-t} \end{bmatrix}:  s,t \in \Rb}$.

 Let $K\cong\SO(2) \times \SO(2)$ be the maximal compact subgroup of $G$ so that $X \cong G/K$. Consider the geodesic $\gamma(t)=(p, a_t \cdot p)$ where $p =e K_0 \in \Hb^2$ is a fixed base-point. Then $\gamma|_{[0,\infty)}$ is a singular geodesic ray and the parabolic subgroup $Q=\PSL_2(\Rb) \times P_2$ is the stabilizer of $\bdx=\gamma(\infty)$. Then 
\begin{align*}
    Z_Q &=\PSL_2(\Rb)\times A_*, N_Q =N_* \times P_*, M_Q =\SO(2) \times \set{e}\\
    A_Q &= \exp \left( \set{\begin{pmatrix} x & y \\ y & -x \end{pmatrix}: x,y \in \Rb }\right) \times A_*, \text{ and } \flow = \set{e} \times A_*
\end{align*}
% \todo[inline]{$A_Q$ doesn't seem to be correct. Seems like it should be: }
% \color{red}
% $A_Q = \exp \left( \set{\begin{pmatrix} x & y \\ y & -x \end{pmatrix}: x,y \in \Rb }\right) \times A_*$
% \color{black}

Moreover $G/Q \cong \set{[e]}\times \PSL_2(\Rb)/P_* \cong \set{\bdx}\times\partial_\infty \Hb^2\subset\partial_\infty \Hb^2\times \partial_\infty \Hb^2\subset \partial_\infty X$. 
Then the Weyl chamber bundle $G/M_Q \cong \Hb^2 \times \PSL_2(\Rb)$, which is a smooth $S^1$-bundle over the symmetric space $X=\Hb^2 \times \Hb^2$. The $\flow$ action on $G/M_Q$ is trivial on the first factor $\Hb^2$ and acts by right multiplication on the second factor $\PSL_2(\Rb)$. Finally, note that for the above geodesic ray, $\gamma$, we have $F(\gamma) \cong \Hb^2 \times \Rb$.

\subsection{Example}

Consider the symmetric space $X=\SL_3(\Rb)/\SO(3)$ and let $G=\SL_3(\Rb)$. Consider the subgroups  
\begin{align*}
    M=\set{\begin{pmatrix}s_1 & 0 & 0 \\ 0 & s_2 & 0 \\ 0 & 0 & s_1s_2\end{pmatrix}: s_i \in \set{\pm 1}}, 
    N=\begin{pmatrix} 1 & * & * \\ 0 & 1 & * \\ 0 & 0 & 1\end{pmatrix}, 
    A=\set{\begin{pmatrix} e^{t_1} & 0 & 0 \\ 0 & e^{t_2} & 0 \\ 0 & 0 & e^{-t_1-t_2} \end{pmatrix}: t_1,t _2 \in \Rb},
\end{align*}
and P=MAN. Then $P$ is up to conjugacy the unique minimal parabolic, and there are two distinct conjugacy classes of proper non-minimal parabolics, given by $Q_-=\begin{pmatrix} * & * & * \\ * & * & * \\ 0 & 0 & *\end{pmatrix}$  and $Q_+:= \begin{pmatrix} * & * & * \\ 0 & * & * \\ 0 & * & *\end{pmatrix}$ respectively. Note that the corresponding Furstenberg boundaries correspond to partial flag varieties, namely $G/Q_- \cong {\rm Gr}_2(\Rb^3)$ and $G/Q_+ \cong \Pb(\Rb^3)$. Observe that these two boundaries are diffeomorphic via the map from ${\rm Gr}_2(\Rb^3)$ to $\Pb(\Rb^{3})$ which takes a 2-plane $V$ to the equivalence class $[V^\perp]$.  However, these two boundaries cannot be equivariantly diffeomorphic.  

Consider geodesic $\gamma(t)=\begin{pmatrix} e^{t} & 0 & 0 \\ 0 & e^{t} & 0 \\ 0 & 0 & e^{-2t} \end{pmatrix}\cdot K$. This is a unit speed geodesic for a suitable choice of symmetric metric on $X$ (all of which differ only by scale). 
Then $\bdx=\gamma(\infty) \in \partial_\infty X$ is stabilized by $Q$. 

The subgroups are:
\begin{align*}
    Z_{Q_-}&:=\begin{pmatrix} * & * & 0 \\ * & * & 0 \\ 0 & 0 & *\end{pmatrix}, 
    N_{Q_-}:=\begin{pmatrix} 1 & 0 & * \\ 0 & 1 & * \\ 0 & 0 & 1\end{pmatrix},
    M_{Q_-}:=\set{\begin{pmatrix} A & 0 \\ 0 & \det(A) \end{pmatrix} : A \in O(2)},\\
     A_{Q_-}&:=\set{\exp\begin{pmatrix} t_1 & s & 0 \\ s & t_2 & 0 \\ 0 & 0 & -t_1-t_2\end{pmatrix}:t_1, t_2, s \in \Rb}, A_{Q_-}'=\set{\begin{pmatrix} e^{t} & 0 & 0 \\ 0 & e^{t} & 0 \\ 0 & 0 & e^{-2t}\end{pmatrix}:t \in \Rb}.
\end{align*}
Note that $A_{Q_-}$ consists of all matrices in $Z_{Q_-}$ that are symmetric and positive definite. 

Then $G/Q_- \cong K/M_{Q_-}\cong SO(3)/O(2)$ is diffeomorphic to the quotient of $S^2$ by a degree two map, i.e. a copy of the two dimensional real projective space. This matches with our previous remark that $G/Q_-$, which is canonically isomorphic to ${\rm Gr}_2(\Rb^3)$ can be identified with $\Pb(\Rb^{3*})$. 

Moreover, $G/M_{Q_-}=PK/M_{Q_-} \cong P \times {\rm Gr_2}(\Rb^3)$ has the structure of a smooth ${\rm Gr}_2(\Rb^3)$ bundle over $P \cong X$. 

The above subgroups and spaces can be analogously determined for $Q_+$ and they are related by the duality between $\Rb^3$ and $\Rb^{3*}$. 

More generally, consider the symmetric space $X=\SL_d(\Rb)/\SO(d)$. The set of positive simple roots is given by $\Delta=\set{e_i^*-e_{i+1}^*: 1 \leq i \leq d-1}$. For any $\Theta \subset \Delta$, we consider the Weyl chamber face $\mf{a}_\Theta^+$ in the canonical flat $A$ given by the positive cone on the root vectors corresponding to the roots in $\Theta$. (This will be a Euclidean $\abs{\Theta}$-simplex.) We can express the exponential action of $\mf{a}^+_\Theta$ on $p\in X$ as $\mc{C}_p(\bdx)$ for some $\bdx\in \partial_\infty X$.  Define $Q=P_{\Theta}$ as the stabilizer in $G$ of $\bdx$. (Any parabolic subgroup is conjugate to some such $P_{\Theta}$.) Then $G/P_{\Theta}$ is the partial flag variety consisting of the $G$-orbit of $\exp(\mf{a}_\Theta^+)$. In terms of matrices, the groups $P_\Theta$ will be conjugate to the set of matrices that vanish below the diagonal in the coordinates corresponding to $\Theta$. Then  $M_Q$ will be the group isomorphic to $SO(q)$ for $q=d-\abs{\Theta}$ which acts by identity in the coordinates corresponding to $\Theta$. The group $A_Q$ will consist of the matrices in $P_\Theta$ that are symmetric and $A_Q'$ will be the diagonal matrices that are $1$ off of the $\Theta$ coordinates.

\section{Coarse Geometry of Symmetric Spaces}

 Let us fix a Riemannian symmetric space $X$ and let $G=\Isom(X)^0$.

We will call any isometric totally geodesic embedding of $\Rb^{\op{rk}(G)}$ into $X$ a (maximal) flat. The intersection of two flats is called a (singular) \emph{subflat}.  A \emph{half-flat} $H$ is a subset of a flat such that $\partial H$ is a subflat of dimension $\op{rk}(G)-1$. A subset of a flat is called a \emph{polytope} if it is the intersection of finitely many half-flats. 
 
We also define the parallel set of a singular subflat (compare with \cref{sec:parallel_sets}). Given a singular subflat $S$, the \emph{parallel set of the singular subflat $S$} is 
\begin{align}
\label{eqn:parallel_set_of_subflat}
    F(S):=\bigcup \set{ F : F \text{ is a flat such that } F \supset S}.
\end{align}

 We call any $(L,C)$-quasi-isometric embedding of $\R^{\op{rk}(G)}$ into $X$ an {\em $(L,C)$-quasiflat}. In particular, quasiflats are always assumed to be of dimension $\op{rk}(G)$ unless stated otherwise.

\begin{lemma}\cite[Theorem 1.2.5]{KleinerLeeb97}\cite[Theorem 1.1]{Eskin-Farb}\label{lem:quasiflat}
%Let $X$ be a symmetric space. 
For every $L\ge 1$ and $C\ge 0$ there exist $k\in \mathbb N$ and $R\ge 0$ such that if $F$ an $(L,C)$-quasiflat then $F$ is contained in an $R$-neighborhood of a union of $k$ flats in $X$. Moreover, if $L$ is sufficiently close to 1 (the `sufficiently close' depending on type of the root system of $G$), then $k=1$.
\end{lemma}

\begin{proof}
We refer to \cite[Theorem 1.2.5]{KleinerLeeb97} and the paragraph after it.
\end{proof}

\begin{definition}
Let $A, B \subset X$. We say $A$ and $B$ coarsely intersect in a set $D$ if there is a constant $R = R(A,B)$ such that $d^{\Haus}(\Nbhd_R(A)\cap \Nbhd_R(B), D)<+\infty$. In this case, we say that $D$ is a  coarse intersection of $A$ and $B$. 
\end{definition} 
Note that the notion of coarse intersection is well-defined only up to finite Hausdorff distance. Hence we use the term `a coarse intersection'.

{\bf Fact:} The coarse intersection of two flats in a symmetric space is a convex polytope, possibly  unbounded; see for example \cite[Appendix B, Lemma B.1]{Eskin98}.

 Let $\gamma_{pz}$ be a geodesic ray in $X$. Recall that $F(\gamma_{pz})$ is the parallel set of $\gamma_{pz}$, see \cref{sec:parallel_sets}. Let 
 \[
F*:=\bigcap \set{F : F \text{ is a flat in }F(\gamma_{pz})}.
 \]
Note that $F*$ is in fact the unique minimal singular subflat that contains the geodesic extension of the ray $\gamma_{pz}$. Here, minimal means there is no singular subflat of strictly lower dimension that contains $\gamma_{pz}$. The symmetric subspace $F(\gamma_{pz})$ is then $F(F*)$, the parallel set of $F*$.

\begin{lemma}\label{lem:coarse_intersection}
Let $q^*:F(\gamma_{pz})\to X$ be a $L$-biLipschitz embedding and let $F*$ be the intersection of all flats in $F(\gamma_{pz})$. Suppose there is a constant $C$ such that if $F$ is a flat in $F(\gamma_{pz})$, then there exists a flat $F'$ in $X$ such that 
\[
d_{\Haus}(q^*(F),F') \leq C. 
\]
\begin{enumerate}
\item Let $F_1$ and $F_2$ be two flats such that $F_1\cap F_2 = F*$. Then the coarse intersection of $F_1'$ and $F_2'$, which we denote by $\Sc(F_1',F_2')$, is a singular subflat of dimension $\dim(F*)$. Furthermore, if the angle between $F_1$ and $F_2$ is bounded away from 0 in the following sense:
\[\Nbhd_R(F_1)\cap \Nbhd_R(F_2)\subset \Nbhd_{3R}(F*),\]
for any $R\geq 0$. Then there is a choice of $\Sc(F_1',F_2')$ such that
\[
d_{\Haus}(\Sc(F_1',F_2'),q^*(F*)) < C_0,
\]
where the constant $C_0$ is uniform and depends only on the symmetric space $X$ and the constants $C$ and $L$. 
\item Let $F_1,F_2, F_3$ be flats such that $F_i\cap F_j=F*$, for $1\le i\neq j\le 3$. Then the coarse intersections $\Sc(F_1',F_2')$ and $\Sc(F_1',F_3')$ are parallel. 
\end{enumerate}
\end{lemma}
\begin{proof}
 (1) Let $D$ be the coarse intersection of $F_1'$ and $F_2'$. Let $\omega$ be a non-principal ultrafilter and let $(r_n)$ be any sequence converging to infinity. Choose the sequence of base points as $(o_n)=(o)$, where $o\in F*$. The map $q^*$ induces a biLipschitz embedding, denoted by $q^{**}$, of the asymptotic cone $Cone(F(\gamma_{pz}),\omega,r_n,o_n))$ into the asymptotic cone $Cone(X,\omega, r_n,q*(o_n))$. The asymptotic cone of the union $F_1\cup F_2$  is diffeomorphic to $F_1\cup F_2$. The map $q^{**}$ maps this set to the asymptotic cone of $F_1'\cup F_2'$. On the other hand the asymptotic cone of $F_1'\cup F_2'$ is a union of two flats along the asymptotic cone of $D$. The biLipschitz map $q^{**}$ maps $F*$ (more precisely, the asymptotic cone of $F*$) to the asymptotic cone of $D$. Hence the asymptotic cone of $D$ is isometric to $F*$. It follows that there is a choice of $\Sc(F_1',F-2')$, the coarse intersection of $F_1'$ and $F_2'$, such that $\Sc(F_1',F_2')$ is a singular subflat of dimension $\dim(F*)$.

For the second claim, let $D_1\subset F_1'$ be the coarse intersection of $F_1'$ and $F_2'$. Since $q^*(F*)$ is in the $C$-neighborhood of both $F_1'$ and $F_2'$, we can choose the singular subflat $D_1$ to be in a $2C$-neighborhood of $F_2'$. Since $q^*$ is $L$-bilipschitz, $q^*(\Nbhd_{4LC}(F_i))\supset \Nbhd_{2C}(F_i')$. By assumption on the angle between flats $F_1$ and $F_2$, we get that
\[q^*(\Nbhd_{4LC}(F*))\supset \Nbhd_{2C}(F_1')\cap \Nbhd_{2C}(F_2')\supset D_1.\]
It follows that the singular subflat $D_1$ is in a $4L^2C$-neighborhood of $q^*(F*)$.

\smallskip

 \noindent (2) Let $D_1\subset F_1'$ and $D_2\subset F_1'$ be the coarse intersection of $F_1'$ and $F_2'$ and the coarse intersection of $F_1'$ and $F_3'$ respectively. By definition of coarse intersection, there exists $R>0$ such that 
\[d^{\Haus}(D_1, \Nbhd_R(F_1')\cap \Nbhd_R(F_2'))<+\infty,\]
and
\[d^{\Haus}(D_2, \Nbhd_R(F_1')\cap \Nbhd_R(F_3'))<+\infty.\]
There exists $R'$ such that 
\[d^{\Haus}(\Nbhd_{R'}(q^*(F_1))\cap \Nbhd_{R'}(q^*(F_2)), \Nbhd_R(F_1')\cap \Nbhd_R(F_2'))<+\infty,\]
and
\[d^{\Haus}(\Nbhd_{R'}(q^*(F_1))\cap \Nbhd_{R'}(q^*(F_3)), \Nbhd_R(F_1')\cap \Nbhd_R(F_3'))<+\infty.\]
On the other hand, $q^*(F*)$ has finite Hausdorff distances to $\Nbhd_{R'}(q^*(F_1))\cap \Nbhd_{R'}(q^*(F_2))$ and $\Nbhd_{R'}(q^*(F_1))\cap \Nbhd_{R'}(q^*(F_3))$. Hence, $D_1$ and $D_2$ have finite Hausdorff distance. Since $D_1$ and $D_2$ are singular subflats of the same dimension, having finite Hausdorff distance implies that they are parallel. 
\end{proof}

In the next result, we will use the notion of parallel set for a singular subflat from \ref{eqn:parallel_set_of_subflat}. 

\begin{proposition}\label{prop:subflat_existence}
Let $q^*:F(\gamma_{pz})\to X$ be a $L$-biLipschitz embedding and let $F*$ be the intersection of all flats in $F(\gamma_{pz})$. We assume that there is a constant $C$ such that the image of any flat in $F(\gamma_{pz})$ has  Hausdorff distance at most $C$ to a flat in $X$. 
Then there is a singular subflat $D$, of the same dimension as $\dim(F*)$,  such that $q^*(F(\gamma_{pz}))$ is contained in a bounded neighborhood of $F(D)$, the parallel set of $D$. 
\end{proposition}

\begin{proof}
We first let $Y$  be a totally geodesic symmetric subspace such that $F(\gamma_{pz})$ isometrically splits as $Y\times F*$. Note that in this splitting, we identify $F*$ with $\{o\}\times F*$, where $o$ is a fixed point in $Y$.

We pick finitely many flats $\{l_1, \dots, l_k\}$ in $Y$ passing through $o$ such that for any point $y\in Y$ and any flat $l$ containing $o$ and $y$, there is a flat $l_i$ in $Y$ such that $l\times F*$ and $l_i\times F*$ intersect in $F*$ at the an angle bounded away from 0 in the sense stated in \cref{lem:coarse_intersection}. By \cref{lem:coarse_intersection}, the images of these two flats are contained in bounded neighborhood of two flats that are parallel to a singular subflat $D$. Furthermore, $D$ has a bounded Hausdorff distance to $q^*(F*)$

Therefore, $q^*(F(\gamma_{pz}))$ is contained in a bounded neighborhood of $F(D)$.
\end{proof}

\section{$C^0$ Local Semi-rigidity of Boundary Actions of Lattices}

In this section, we study actions of higher rank lattices in semisimple Lie groups of the non-compact type $G$ on projective boundaries of $G$ and prove the main theorem.  These are much better behaved than the actions of fundamental groups of general non-positively curved  manifolds on boundary spheres, and rather resemble negatively curved manifolds. We attribute this to the shared good properties for quasi-geodesics and quasi-flats and  the uniform transversality properties of weak stable and unstable manifolds.

We obtain the following local stability result which is a slightly more detailed version of \cref{thm:main_thm_lattice}:   

\begin{theorem}\label{thm:rigidity_G/Q}
Let $G$ be a connected linear semisimple Lie group without compact factors, $\Gamma$ be a uniform lattice in $G$, $Q$ be a parabolic subgroup of $G$ and $\rho_0: \Gamma \to \Homeo(G/Q)$ is the boundary action (i.e. action by left multiplication). Then there is a neighborhood $\Uc_0$ of $\rho_0$ and a neighborhood $V_0$ of $\id$ (in the space of continuous maps on $G/Q$) such that: for every $\rho \in \Uc_0$, there exists a unique $(\rho,\rho_0)$-semiconjugacy $\phi_\rho \in V_0$. Moreover, $\phi_\rho$ converges to $\id$ uniformly as $\rho$ converges to $\rho_0$. 
\end{theorem}

\begin{remark}
We do not necessarily require that $G$ is linear, but rather that $G$ is semisimple without compact factors and $\Gamma<G$ has a torsion-free normal subgroup $\Gamma_0\vartriangleleft  \Gamma$ such that $\Gamma/\Gamma_0$ is finitely generated and $\Gamma_0$ is isomorphic to a lattice in $\op{Isom}(X)^0$ for the symmetric space $X$ of noncompact type associated to $G$. However, we state it in the form above for a simpler statement.
\end{remark}

In \cref{thm:rigidity_G/Q}, the uniqueness of $\phi_\rho$ in $V_0$ is an application of \cref{lem:semi_unique}. Our goal in this section is to prove the existence of the semiconjugacies $\phi_\rho$ for all perturbations close to $\rho_0$, i.e. the $C^0$ local semi-rigidity of $\rho_0$. The construction of semi-conjugacies $\varphi_{\rho}$ will be done  in Section \ref{sec:construction}.  Then we will show in Section \ref{sec:local} that $\varphi_{\rho}$ is continuous and that it converges to the identity 
 uniformly as $\rho$ converges to $\rho _0$. 
\vspace{1em}

\noindent {\bf Initial reductions:} 
We will now explain how in \cref{thm:rigidity_G/Q}, we can always reduce to the case that $G$ is a connected center-free semisimple Lie group without compact factors and $\Gamma$ is a torsion-free uniform lattice in $G$.

Since $G$ is linear,  Selberg's lemma implies there is a normal torsion-free finite index subgroup  $\Gamma_0 \vartriangleleft \Gamma$. By \cref{cor:torsionfree}, it suffices to prove that the action $\rho|_{\Gamma_0}$ of $\Gamma _0$ is  $C^0$-locally semi-rigid and that a  semi-conjugacy is sufficiently close to $\id$ is unique.  

Let $Z(G)$ be the center of $G$ and $h: G \to G/Z(G)$. Then $\Gamma_0$ is isomorphic to $h(\Gamma_0)$ where $h(\Gamma_0)$ is a torsion-free uniform lattice in $h(G)$. Moreover $h(G)$ is a connected semisimple group with trivial center and without compact factors. Finally, $h(Q)$ is a parabolic subgroup of $h(G)$ and $G/Q$ is homeomorphic to $h(G)/h(Q)$.

\bigskip

\noindent {\bf Setup:} Henceforth in this section, we will work with  the following assumptions and notations. Let $G$ be a connected semisimple Lie group with trivial center and without compact factors. Let $X$ be the Riemannian symmetric space attached to $G$, i.e. $G=\Isom(X)^0$. Let $Q<G$ be a parabolic subgroup. We fix a torsion-free uniform lattice $\Gamma$ in $G$. 

\begin{itemize}
    \item  Let $\bdx \in \partial_{\infty}X$ be a point such that $Q=G_{\bdx}$.
    
    \item Fix a basepoint $p \in X$ and let $\gamma_{p \bdx}$ be the geodesic in $X$ that is asymptotic to $\bdx$. We will denote by $F(\gamma_{p\bdx})$ its parallel set (see \Cref{sec:parallel_sets}).

    \item The space $X_{G/Q}:=G/M_Q$ is the Weyl chamber $Q$-face bundle equipped with a $G$-invariant Riemannian metric, see \Cref{subsec:metric-on-G/MQ}. Recall that $p:X_{G/Q} \to X$ is a fiber bundle with base $X$ and fibers homeomorphic to $K/M_Q\cong G/Q$. The map $\Phi:X_{G/Q} \to X \times G/Q$ is the diffeomorphism introduced in \Cref{sec:G/M_Q}.

    \item The action $\rho_0:\Gamma\to \Homeo(G/Q)$ will always denote the algebraic action of $\Gamma$ on $G/Q$ by left multiplication. Then $\rho_0$ is a standard action with associated data $(X_{G/Q},p,\Phi,\{\pi_{x}\},\{g_x\},\ell)$ where $\ell \geq 1$. See \Cref{sec:standard-action} for the definition of standard action, \Cref{rem:associated_data_for_G/M_Q} for the explanation of associated data in this case, and \Cref{sec:G/Q-sec} for a proof of this fact.

\item
Let $\wcs, \wcu$ be the center-stable and center-unstable manifolds in $G/M_Q$ for the Weyl chamber flow, as in \Cref{sec:foliations_weyl_chamber_bundle}. 
\item If $\rho:\Gamma \to \Homeo(G/P)$, then $\wh{\rho}:\Gamma\to \Homeo(X\times G/Q)$ is the induced action (as in \Cref{def:induced action}).
\end{itemize}

\subsection{Construction of equivariant maps} \label{sec:construction}

 The construction proceeds in several steps. As the whole argument is rather complicated, let us give an outline here.  For simplicity, suppose the parabolic $Q$ is minimal. Then, eventually, we want to produce a quasi-flat that comes about as the intersection of unperturbed  center-unstable with the ``perturbed center stable''.  This gives us a perturbation of the center foliation of the unperturbed action which is given by flats.  The hope then is that these actually are quasi-flats.  Then we can apply the work of Kleiner Leeb and Eskin Farb on quasi-flats which gives us shadowing of the ``perturbed center leaves'' in Proposition \ref{prop:image-unique} and allows to define a $\rho-\rho _0$ equivariant map on the boundary in Proposition \ref{prop:semiconj}.  To get there we first get local bi-Lipschitz control of the ``perturbed center foliation'' in Proposition \ref{prop:BMMforG/Q} (the Bowden-Mann Lemma).  From this we then derive global bi-Lipschitz in Proposition \ref{prop:geometric-orbit}.  From a dynamical point of view, this is the analogue of orbit equivalence of the perturbation of an Anosov flow, i.e. structural stability.

We now begin the construction. We first apply the generalized lemma about standard actions to our setup of perturbing the standard action of a uniform lattice on the Furstenberg boundary.

\begin{proposition}
\label{prop:BMMforG/Q}
If $\rho: \Gamma \to \Homeo(F)$ is sufficiently close to $\rho_0$, then there is a continuous map $\til{f}:X \times G/Q \to X_{G/Q}$ satisfying:
\begin{enumerate}
\item $p(\til{f}(x,\xi))=x$ and $\til{f}(\gamma \cdot x,\rho(\gamma)\xi) =\wh\rho_0(\gamma)\til{f}(x,\xi)$
\item  the map $\til{f}(\cdot,\xi):X \to \til{f}(X \times \{\xi\})$ is a $C^{1}$ diffeomorphism for every $\xi \in G/Q$. Moreover, for any $(x_0,\xi_0)\in X \times F$, $(\rho,y,\xi) \mapsto D_y\til f(y,\xi)$ is continuous in all three variables at $(\rho_0,x_0,\xi_0)$.

\item for every $\epsilon>0$ and for every $x\in X$, there exists a neighborhood $\Uc = \Uc(x,\eps)$ of $\rho_0$ such that 
\[\sup_{\rho \in ~\Uc} \left( \sup_{\xi \in G/Q} d_{G/Q}(\pi_{x}(\til f(x,\xi)),\xi) \right)<\eps. \] 

\item  for every $\epsilon, R >0$, there is a neighborhood $\Uc = \Uc(\eps,R)$ of $\rho_0$ such that: if $\rho\in \Uc$, $(x,\xi)\in X\times G/Q$ and $\eta:=\pi_x(\til f(x,\xi))\in G/Q$, then 
$$\sup_{y \in B_R(x)}d_{F_y} \left( \til{f}(y,\xi), (\pr_F \circ \Phi)^{-1}(\eta) \cap F_y \right)<\epsilon,$$ 
where $d_{F_y}$ is the distance on the fiber $F_y$ (induced by the Riemannian metric $g_y$). 

\item  We consider the Riemannian metric on $X_{G/Q}$ constructed in \cref{subsec:metric-on-G/MQ}. For any $\epsilon > 0$, there exists a neighborhood $\Uc =\Uc(\eps)$ of $\rho_0$ such that
\[
\sup_{\rho \in ~\Uc}  \left( \sup_{(x,\xi) \in X \times F}d^{\op{Gr}}_{\til f(x,\xi)}\left( T_{\til f(x,\xi)} \til f(X \times \xi), T_{\til f(x,\xi)} (\pr_F \circ \Phi)^{-1}(\eta_x) \right) \right) < \eps,
\]
where $\eta_x:=\pi_x(\wt{f}(x,\xi))$.
\end{enumerate}
\end{proposition}
\begin{proof}
This is essentially an application \cref{prop:BM} to the standard action  $\rho_0$ of $\Gamma$ on $G/Q$. But deriving this result from \cref{prop:BM} requires a few observations  we will now remark on. Note that the maps $\pi_x:F_x \to F$ are in fact smooth and each fiber $F_x$ is totally geodesic in $G/M_Q$. Let $g_x$ denote the induced Riemannian metric on $F_x$.   If $x \in X$ and $\gamma \in \Gamma$, then the map $\gamma_x:(F_x,g_x) \to (F_{\gamma x},g_{\gamma x})$ is an isometry (so more than just affine). 

The smoothness of $\pi_x$ implies that the H\"older assumption in \cref{prop:BM} part (4) is satisfied. Since the action $\wh{\rho_0}$ is by isometries on $G/M_Q$, note that we can make the following assumption when applying part (4): we can assume that $x$ lies in a fundamental domain of $X/\Gamma$. Moreover, $F$ is also compact. Then the neighborhood $\Uc$ of $\rho_0$ in the conclusion of part (4) can be chosen in such a way that it does not depend on $x$ or $\xi$. 

Finally, part (5) of \cref{prop:BM} applies because $\wh{\rho_0}$ acts by isometries on $G/M_Q$; see \cref{subsec:metric-on-G/MQ}. 

We also remark that in the last claim of \cref{prop:BMMforG/Q}, since $\rho_0(\Gamma)\backslash X_{G/Q}$ is a closed manifold, the claim still holds for any lifted metric on $X_{G/Q}$ from a Riemannian metric on $\rho_0(\Gamma)\backslash X_{G/Q}$.
\end{proof}

For the rest of this section, $\til f:X\times G/Q \to X_{G/Q}$ will always denote the map obtained in \cref{prop:BMMforG/Q}.  In the next proposition, we will control the images of $\til{X}\times\set{\xi}$ under $\til{f}$. But first we make a definition.

\begin{definition}
For any $\xi, \eta \in G/Q$, we define 
\[ 
\Qc_{\xi,\eta}:=\wt{f}(X \times \set{\xi}) \cap \wcu(\eta).
\]
\end{definition}
Note that $\Qc_{\xi,\eta}$ may be an empty set for some choices of $\xi$ and $\eta$. 
\begin{remark}
Note that $\til{f}$ depends on $\rho: \Gamma \to \Homeo(G/P)$. Thus $\Qc_{\xi,\eta} \equiv \Qc_{\xi,\eta}(\rho)$. But we suppress  $\rho$ for the sake of brevity in this subsection. Further, if $\rho=\rho_0$, then \[\Qc_{\xi,\eta}(\rho_0)=\til f_{\rho_0}(X \times \xi) \cap \wcu(\eta)=\wcs(\xi) \cap \wcu(\eta).\] Thus, informally speaking, one should think of $\Qc_{\xi,\eta}\equiv \Qc_{\xi,\eta}(\rho)$ as `perturbed center' manifolds, where the perturbation comes from perturbing the action $\rho_0$ to $\rho$.
\end{remark}

%\subsubsection{`Perturbed center' manifolds are contained in the neighborhood of center manifolds}
\subsubsection{\bf $\Qc_{\xi,\eta}$ is contained in the neighborhood of a center manifold $\wcu(\eta) \cap \wcs(\zeta)$}
The main result of this subsection is \cref{prop:image-unique} which shows that each non-empty `perturbed center' manifold $\Qc_{\xi,\eta}$ is contained in a bounded neighborhood of a center manifold $\wcu(\eta) \cap \wcs(\zeta)$. Along the way, we will establish several properties of the `perturbed center manifolds' $\Qc_{\xi,\eta}$. 

In the following lemma, $d^{\op{Gr}}_{v}$ is the Grassmanian distance between linear subspaces of $T_v(G/M_Q)$, induced by the Riemannian metric on $G/M_Q$.
\begin{lemma}
\label{lem:Q_xi_eta}
\label{prop:quasi_flat_angle_control} 
For any $\eps >0$, there exists a neighborhood $\Uc = \Uc(\eps)$ of $\rho_0$ such that:  for any $\rho \in \Uc$,  if  $\xi,\eta \in G/Q$ and $x \in X$ are such that $\til f(x,\xi)\in \mathcal Q_{\xi,\eta}$, then 
\[ d^{\op{Gr}}_{\til f(x,\xi)}\left( T_{\til f(x,\xi)} \Qc_{\xi,\eta}, T_{\til f(x,\xi)} \left(\wcs(\zeta_0)\cap \wcu(\eta) \right) \right) < \eps\]
 where $\zeta_0=\pi_x(\til f(x,\xi))$. 
\end{lemma}
\begin{proof}
Fix any $\eps'>0$ and let $\Uc=\Uc(\eps')$ be the neighborhood of $\rho_0$ given by Proposition \ref{prop:BMMforG/Q} part (5).  Suppose $\Qc_{\xi,\eta} \neq \emptyset$, $x \in X$ and $\zeta_0=\pi_x(\til f(x,\xi))$. By Proposition \ref{prop:BMMforG/Q} part (5),  the tangent spaces $T_{\til f(x,\xi)}\til f(X \times\{\xi\})$ and  $T_{\til f(x,\xi)}\wcs(\zeta_0)$ are $\epsilon'$-close in the Grassmannian distance $d^{\op{Gr}}_{\til f(x,\xi)}$. On the other hand, for every $\til f(x,\xi)\in \mathcal Q_{\xi,\eta}$,
\[ T_{\til f(x,\xi)}\mathcal Q_{\xi,\eta}=T_{\til f(x,\xi)}\wt{f}(X \times\{\xi\})\cap T_{\til f(x,\xi)}\wcu(\eta).\]

Since $\til f(x,\xi)\in \Qc_{\xi,\eta}$, $\til f(x,\xi) \in \wcu(\eta)$. Thus $\til f(x,\xi) \in \wcs(\pi_x(\til f(x,\xi)) \cap \wcu(\eta)=\wcs(\zeta_0) \cap \wcu(\eta)$. Since $\wcs(\zeta_0)$ and $\til f(X\times \{\xi\})$ form angles at least $\frac \pi 2-\epsilon'$ with $\wcu(\eta)$, there exists $\epsilon=\epsilon(\epsilon')$ such that
\begin{align*}
&d^{\op{Gr}}_{\til f(x,\xi)}\left( T_{\til f(x,\xi)} \Qc_{\xi,\eta}, T_{\til f(x,\xi)} \left(\wcs(\zeta_0)\cap \wcu(\eta) \right) \right)\\ &= d^{\op{Gr}}_{\til f(x,\xi)} \left( T_{\til f(x,\xi)} \til f(X \times \xi)\cap T_{\til f(x,\xi)} \wcu(\eta), T_{\til f(x,\xi)} \wcs(\zeta_0)\cap T_{\til f(x,\xi)} \wcu(\eta)\right) <\eps.
\end{align*}
Furthermore $\epsilon$ tends to 0 when $\epsilon'$ tends to 0. We then obtain the lemma by choosing $\eps'$ sufficiently small.
\end{proof}

Next, we will show that each connected component of $\Qc_{\xi,\eta}$ is biLipschitz diffeomorphic to the parallel set $F(\gamma_{p \bdx})$.

\begin{proposition}\label{prop:geometric-orbit}
Let $L>1$.  Then there is a neighborhood $\mathcal{U}$ of $\rho _0$ such that: if $\rho \in \mathcal{U}$ and 
  $\xi, \eta \in G/Q$ are such that $\mathcal Q_{\xi,\eta} \neq \varnothing$,  then  for each connected component $\Qc$ of $\Qc_{\xi,\eta}$  there is a map $\alpha_{\xi,\eta}: F(\gamma_{p\bdx}) \to \wcu(\eta)$ which is an $L$-biLipschitz diffeomorphism onto $\Qc$, where the distance on $\mathcal Q$ is the restriction of the distance on $\wcu(\eta)$. %Moreover we may choose $L$ so that $L \to 1$ as $\rho \to \rho_0$.
%\end{enumerate}
\end{proposition}

\begin{proof}
Let us fix the connected component $\Qc$ of $\Qc_{\xi,\eta}$.

Let $\xi =gQ\in G/Q$ and let $\gamma_{p\bdx'}$ be the geodesic passing asymptotic to $\bdx'=g\bdx$. We can identify $\wcu(\xi)$ with $F(\gamma_{p\bdx'})(gN_Q g^{-1})$ by first identifying $\wcu(\xi)$ with $X$, and then identifying $X$ with $F(\gamma_{p\bdx'})(gN_Q g^{-1})$. The overall identification can be thought of as the global product structure on a center unstable manifold in terms of center manifolds and strong unstable manifolds.

The center-unstable leaf $\wcu(\eta)$ is foliated by the family of symmetric subspaces $\{F(\gamma_{p{\bdx}}):p\in X\}$. Each  subspace in the family is identified with a center manifold of a point $\Cc_q(\eta)\in \wcu(\eta)\subset X_{G/Q}$. The projection along strong unstable leaves maps every center leaf isometrically to a fixed center leaf. 

Consider the map $retr:\wcu(\eta) \to F(\gamma_{p{\bdx}'})$ which is obtained by projection along strong unstable leaves. Then set $q_{\xi,\eta}:=retr|_{\Qc}$, the restriction of $retr$ to $\Qc \subset \wcu(\eta)$.

By \Cref{lem:Q_xi_eta}, the tangent spaces of $\mathcal Q$ are  $\epsilon$-close to $\wcs(\zeta) \cap \wcu(\eta)$ where $\zeta=\pi_x(\til{f}(x,\xi))$. Hence, the map $q_{\xi,\eta}$ is locally $L$-biLipschitz if $\eps$ is small enough. That is, for every $\til f(x,\xi)\in \Qc$, there is a neighborhood of $\til f(x,\xi)$ such that $q_{\xi,\eta}$ is $L$-biLipschitz. Then note that $q_{\xi,\eta}$ is globally $L$-Lipschitz. %Moreover, the constant $L\equiv L(\epsilon)$ and $L\to 1$ as $\epsilon\to 0$. As $\rho \to \rho_0$, $L \to 1$ because $\epsilon \to 0$ as $\rho \to \rho_0$ (see \Cref{lem:Q_xi_eta}), $L \to 1$. 

We now claim that $q_{\xi, \eta}$ is a diffeomorphism. This is an application of Lemma \ref{lem:pi_surj} from the Appendix. For applying the Lemma, we first need to fix a suitable foliation chart. Note that the center unstable manifold $\wcu(\eta)$ has a global product structure in terms of center manifolds and strong unstable manifolds. This provides two mutually orthogonal foliations on $\wcu(\eta)$ and $\wcu(\eta)$ is homeomorphic to $\Rb^{\dim(X)}$. Finally $\Qc$ is a complete connected  $C^1$ sub-manifold of $\wcu(\eta)$ with tangent spaces which are uniformly close to the center distribution (see \cref{lem:Q_xi_eta}). Thus, we can apply Lemma \ref{lem:pi_surj} and prove that $q_{\xi,\eta}$ is a diffeomorphism.

Then we define $\alpha_{\xi,\eta}: F(\gamma_{p\bdx}) \to \wcu(\eta)$ by setting $\alpha_{\xi,\eta}:=q_{\xi,\eta}^{-1}$. Thus $\alpha_{\xi,\eta}$ is a diffeomeorphism with image $\Qc$. As $\alpha_{\xi,\eta}=q_{\xi,\eta}^{-1}$ is locally $L$-Lipschitz, this implies that $\alpha_{\xi,\eta}$ is globally $L$-Lipschitz. Since $q_{\xi,\eta}=\alpha_{\xi,\eta}^{-1}$ is also globally $L$-Lipschitz, this implies that $\alpha_{\xi,\eta}$ is globally $L$-biLipschitz.
\end{proof}

Next, we show that each connected component $\Qc$ of $\Qc_{\xi,\eta}$ is contained in a uniformly bounded neighborhood of some `center' leaf $\wcu(\eta) \cap \wcs(\zeta)$. 

\begin{lemma}\label{lem:image-unique}
 Let $\rho$ be sufficiently close to $\rho_0$ and suppose $\xi,\eta \in G/Q$ are such that $\mathcal Q_{\xi,\eta}$ is non-empty. Then for each connected component $\Qc$ of $\mathcal Q_{\xi,\eta}$ there is a unique $\zeta\in G/Q$ such that $\Qc$ is contained in a uniform bounded  neighborhood of $\wcu(\eta)\cap \wcs(\zeta)$.
\end{lemma}
\begin{proof}
By \Cref{prop:geometric-orbit}, each connected component $\Qc$ of $\mathcal Q_{\xi,\eta}$ is the bi-Lipschitz image of $F(\gamma_{p\bdx})$ in $\wcu(\eta)$, which is isometric to the symmetric space $X$. By \cref{prop:subflat_existence}, there is a singular subflat $D$ such that $\mathcal Q$ is in a bounded neighborhood of the parallel set of $D$. As a subset of $\wcu(\eta)$, the parallel set of $D$ is the center manifold of some $\zeta\in G/Q$. 

For the uniqueness of $\zeta$, suppose that $\zeta'\in G/Q$ such that $\mathcal Q$ is contained in a bounded neighborhood of $\wcs(\zeta')\cap \wcu(\eta)$. Let $F$ be a maximal flat in $F(\gamma_{pz})$ that maps to $q_{\xi,\eta}^{-1}(F)$ and has finite Hausdorff distance to a flat $F'$ in $\wcu(\eta)$. On the other hand, if there is a flat contained in a bounded neighborhood of a center manifold then the flat is itself contained in the center manifold. Thus, as the flat $F'$ is contained in a bounded neighborhood of two centers $\wcs(\zeta)\cap \wcu(\eta)$ and $\wcs(\zeta')\cap \wcu(\eta)$, the two centers has non-empty intersection, hence coincide. Therefore, $\zeta=\zeta'$.
\end{proof}

Note that we do not know whether $\Qc_{\xi,\eta}$ is connected. So a priori, the previous lemma does not imply that $\Qc_{\xi,\eta}$ is contained in a uniformly bounded neighborhood of the `center' leaf $\wcu(\eta) \cap \wcs(\zeta)$.  We prove this in \cref{prop:image-unique}. But first, we need the following lemma which is a consequence of \cref{prop:BMMforG/Q}. 

\begin{lemma}\label{lem:locallyclose}
For every $R, D>0$, there exists a neighborhood $\mathcal{U}'$ of $\rho_0$ such that: suppose $\rho \in \mathcal{U}'$
and  $\xi,\eta \in G/Q$ are such that $\mathcal Q_{\xi,\eta}$ is non-empty. Then  if $\til f (x,\xi)\in \mathcal Q_{\xi,\eta}$ and $\zeta:=\pi_x(\til f(x,\xi))$, we get  
\begin{align*}
    \mathcal Q_{\xi,\eta}\cap B_R(\til f(x,\xi)) \subset \Nbhd_{D}\left( \wcs(\zeta)\cap \wcu(\eta)\cap B_R(\til f(x,\xi)) \right).
\end{align*}
\end{lemma}

\begin{proof}
By \cref{prop:BMMforG/Q} part (4), for every $\epsilon, R>0$, there exists a neighborhood $\Uc$ of $\rho_0$ such that for every $\rho\in \Uc$, 
\[ \sup_{y\in B_R(x)} d_{F_y}(\til f(y,\xi), \wcs(\zeta)\cap F_y)<\epsilon.\]
Note that $B_R(\til f(x,\xi)) \cap \til f(X \times \xi) \subset \til f(B_R(x),\xi)$. It follows that 
\[  \mathcal Q_{\xi,\eta}\cap B_R(\til f(x,\xi)) \subset \Qc_{\xi,\eta} \cap \til f(B_R(x),\xi) \subset \Nbhd_\epsilon(\wcu(\eta) \cap \wcs(\zeta)).\]
Thus,
\[
\mathcal Q_{\xi,\eta}\cap B_R(\til f(x,\xi))  \subset \Nbhd_\epsilon(\wcu(\eta) \cap \wcs(\zeta)\cap B_R(\til f(x,\xi)).
\]
Now given any $R,D>0$, the lemma follows immediately by choosing $\eps=D$. 
\end{proof}

In the next proposition, the key idea is that any two components of $\mc{Q}_{\xi,\eta}$ ``fellow-travel'' and each lies in a uniform neighborhood of some center leaf which must therefore be the same.

\begin{proposition}\label{prop:image-unique}
        Let $\rho$ be sufficiently close to $\rho_0$ and suppose $\xi,\eta \in G/Q$ be such that $\mathcal Q_{\xi,\eta}$ is non-empty. Then there is a unique $\zeta\in G/Q$ such that $\mathcal Q_{\xi,\eta}$ is contained in a uniformly bounded neighborhood of $\wcu(\eta)\cap \wcs(\zeta)$.
\end{proposition}

\begin{proof}
    Let $\mathcal Q_1$ and $\mathcal Q_2$ be two connected components of $\mathcal Q_{\xi,\eta}$. By \cref{lem:locallyclose}, there exists $D_0$ such that each of $\mathcal Q_i$ is contained in a $D_0$-neighborhood of a center leaf. We note that $D_0$ depends only on the neighborhood of $\rho_0$ and the symmetric space $X$. We would like to prove the two center leaves coincide.

    We let $\zeta_1,\zeta_2\in G/Q$ such that $\mathcal Q_i$ is contained in $D_0$-neighborhood of $\wcs(\zeta_i)\cap \wcu(\eta)$, for $i=1,2$. Suppose for that $\zeta_1\neq \zeta_2$. We let $y\in \wcu(\eta)$ such that $y$ has distance at most $\frac{D_0}{2}$ to both $\wcs(\zeta_1)\cap \wcu(\eta)$ and $\wcs(\zeta_2)\cap \wcu(\eta)$ but $\Nbhd_{10D_0}(B_R(y)\cap \wcs(\zeta_{i})\cap \wcu(\eta))$ does not contain $B_R(y)\cap \wcs(\zeta_{3-i})\cap \wcu(\eta)$, for $i=1,2$, and where the constant $R$ is from \cref{lem:locallyclose}. Furthermore, any intersection of a $L$-biLipschitz flat in $\mathcal Q_i$ with $B_R(y)$ is not contained in a $8D_0$-neighborhood of $B_R(y)\cap \wcs(\zeta_{3-i})\cap \wcu(\eta)$.
    
    We assume that $\rho$ is sufficiently close to $\rho_0$ so that $D<D_0$ and $R>20D_0$. By \cref{lem:locallyclose}, $\mathcal Q_{\xi,\eta}\cap B_{R}(y)$ is contained in a $D$-neighborhood of a center leaf. But this contradicts the last paragraph.
\end{proof}

\subsubsection{\bf The $\zeta$ in \cref{prop:image-unique} is independent of $\eta$}

By \cref{prop:image-unique}, each $\Qc_{\xi,\eta}$ is contained in the bounded neighborhood of a center manifold $\wcu(\eta) \cap \wcs(\zeta)$. In this subsection, we will show that given $\xi$, this $\zeta$ is independent of $\eta$ (provided $\Qc_{\xi,\eta}$ is non-empty). This is the content of \cref{lem:well-defined}, the main result of this subsection. Then, an application of this lemma is \cref{prop:semiconj} where we construct a $(\rho,\rho_0)$ equivariant map on $G/Q$.

But first, we need some preparatory lemmas. Recall the notion of opposite points in $G/Q$ from \cref{sec:opp_weyl_faces}.

\begin{lemma}\label{lem:opposites}
If $\alpha, \beta$ in $G/Q$, then the set of elements opposite to both $\alpha$ and $\beta$ is open dense in $G/Q$.
\end{lemma}

\begin{proof}
The set of points opposite to $\alpha$ corresponds to an open dense Bruhat cell. The same holds for $\beta$.  Hence  the desired set is an intersection of open dense sets.
\end{proof}

\begin{lemma}\label{lem:source_sink}
Suppose $\zeta_1 \neq \zeta_2 \in G/Q$ and $\eta_1, \eta_2 \in G/Q$ are both opposite to $\zeta_1$. Then there exists a sequence $(\gamma_n)$ in $\Gamma$ such that 
\[\lim_{n\to \infty}\rho_0(\gamma_n) \zeta_1 \neq \lim_{n\to \infty}\rho_0(\gamma_n) \zeta_2\]
while 
\[\lim_{n\to \infty} \rho_0(\gamma_n) \eta_1= \lim_{n\to \infty}\rho_0(\gamma_n) \eta_2.\]
\end{lemma}
\begin{proof}

We will first prove this statement for $G$ in place of $\Gamma$.

We recall that every element $\eta$ of $G/Q$ corresponds to an equivalence class of asymptotic based Weyl chambers $Q$-faces. It follows that $\eta$ can be identified with the geodesic boundary of a based Weyl chamber $Q$-face in the equivalence class, and thus can be treated as a closed subset in the geodesic boundary at infinity of $X$.

 We let $F*$ be the minimal singular subflat with $\eta_1$ and $\zeta_1$ contained in its geodesic boundary at infinity(i.e. the singular subflat  of minimal dimension containing $\eta_1$ and $\zeta_1$ in its boundary).
 
 Let $g\in G$ be an element whose  axis is the parallel set of $F*$ (here, by axis of $g$, we mean the $\set{x \in X: d_X(x,g \cdot x)=\inf_{y \in X}d_X(y,g \cdot y)}$). By switching to $g^{-1}$ if necessary, we may assume that $\zeta_1$ is the source and $\eta_1$ is the sink of $g$. Under iteration, $g$ contracts to $\eta_1$ the (open dense) set of elements of $G/Q$ which are opposite to $\zeta_1$.
 
 In particular, $\eta_2$ converges to $\eta_1$ under iterations. Moreover, under iterations, the only element of $G/Q$ converging  to $\zeta_1$ is $\zeta_1$ itself. Hence, $\zeta_2$ converges to a point distinct from $\zeta_1$ under iterations of $g$. 

To obtain the claim for $\Gamma$, we approximate iterations of $g$ by elements of the lattice. Indeed, writing $g^n=c_n\gamma_n$, where $c_n$ lies in a compact fundamental domain, and passing to a subsequence, we can assume that the elements $c_n$ converge to a fixed element $c$. Hence the result follows using $\gamma_n$ instead of $g^n$.
\end{proof}

By pulling back the `perturbed center stable' leaves $\Qc_{\xi,\eta}$, we now introduce a foliation $\Fc(\xi,\eta)$ of $X \times G/Q$. This is the content of the next two lemmas. First however let us  define $\Fc(\xi,\eta)$. Let 
\[
\mc{V}_0:=\left\{ (\xi,\eta) \in G/Q \times G/Q : \mathcal Q_{\xi,\eta}\neq \varnothing \right\}.
\]
\begin{definition}
    For $(\xi,\eta)\in \mathcal{V}_0$, we let
\[\mathcal F(\xi,\eta):=\til{f}^{-1}(\wcu(\eta))\cap (X\times \{\xi\})=\til f^{-1}(\mathcal Q_{\xi,\eta})\cap (X\times \{\xi\}).\]
\end{definition}

We note that each $\mathcal F(\xi,\eta)$ is a $C^1$ manifold. Furthermore, the disjoint union of all such $\mathcal F(\xi,\eta)$ is $X\times G/Q$. The set of leaves is identified with an open subset $\mc{V}_0\subset G/Q \times G/Q$. We note that the intersection of the unstable foliation of $X_{G/Q}$ with $\til f(X\times \{\xi\})$ gives rise to a foliation of the $C^1$ manifold $\til f(X\times \{\xi\})$. By pulling this back by $\til f$ we obtain a foliation of $X\times \{\xi\}$ whose leaves are
\[\set{\mc{F}(\xi,\eta):\eta\in G/Q, (\xi,\eta)\in \mc{V}_0}.\]
Then, for any $\gamma \in \Gamma$, $\set{\wh{\rho}(\gamma)\mc{F}(\xi,\eta) : \eta\in G/Q, (\xi,\eta)\in \mc{V}_0}$ is a foliation of $X \times \set{\rho(\gamma) \xi}$ as $\til f$ is $(\wh\rho, \wh\rho_0)$-equivariant. 

We will now prove the continuity of $\Qc_{\xi,\eta}$ and $\Fc(\xi,\eta)$ using \cref{lem:continuity-of-transverse-intersection}.

\begin{lemma}
\label{lem:continuityQ}
    The map $\Vc_0 \ni (\xi,\eta) \mapsto \Qc_{\xi,\eta}$ is continuous with respect to the pointed Hausdorff topology. 
\end{lemma}
\begin{proof}
    Continuity in $\eta$ is straightforward as the unstable leaves produce a continuous foliation of $\til{f}(X \times \set{\xi})$ for any fixed $\xi$.
    For proving continuity in $\xi$, fix $(\xi_0,\eta)\in \Uc$ and consider a sequence $(\xi_n,\eta)$ in $\Uc$  converging to $(\xi_0,\eta)$. Apply \cref{lem:continuity-of-transverse-intersection} with $M_n= \til{f}(X\times\set{\xi_n})$ for $n\geq 0$, and $N=\wcu(\eta)$. 
\end{proof}

\begin{lemma}
\label{lem:Q_foliation}
    The map  $\Vc_0 \ni (\xi,\eta) \mapsto \mathcal F(\xi,\eta)$ is continuous with respect to the pointed Hausdorff topology.
\end{lemma}
\begin{proof}
    Continuity in $\eta$ is immediate from the fact that $\{\mathcal F(\xi,\eta):\eta\in G/Q \text{ such that } \mathcal F(\xi,\eta) \neq \varnothing\}$ is a continuous foliation of $X\times \{\xi\}$.

    For proving continuity in $\xi$, let $(\xi_n)$ be a sequence converging to $\xi_0\in G/Q$. Note that since $\til f$ covers $\id$ on $X$, 
\[\mathcal F(\xi,\eta)=p((\til f(X\times \{\xi \})\cap \wcu(\eta))\times \{\xi\}=p (\Qc_{\xi,\eta}) \times \set{\xi}.\]
Since $\xi_n \to \xi_0$, \cref{lem:continuityQ} implies that $\Qc_{\xi_n, \eta} \to \Qc_{\xi_0,\eta}$ in the pointed Hausdorff topology. Then $p(\Qc_{\xi_n, \eta}) \to p(\Qc_{\xi_0,\eta})$ in the pointed Hausdorff topology. This finishes the proof.
\end{proof}

\begin{remark}
We point out that the last two  lemmas do not follow from a general statement about the continuity of leaves of continuous foliations of the leaves of a smooth foliation.

Indeed, Bonatti and Franks \cite[Theorem 1.1]{BonattiFranks04} even provide a H\"{o}lder continuous distribution $E'$ on $\R^2$ tangent to a continuum $(\scr{F}_s)_{s\in (0,1)}$ of pairwise distinct foliations. Then consider the H\"{o}lder continuous distribution $E$ on $\R^3$ that consists of a copy of $E'$ on each $\set{t}\times \R^2$. Now we can choose a discontinuous map $\sigma:\R\to (0,1)$ and consider for each $t\in \R$ the Bonatti-Franks foliation $\scr{F}_{\sigma(t)}$ on $\set{t}\times \R^2$. Now the union of these foliations of $\R^2$ does not form even a $C^0$ foliation of $\R^3$. In particular, a sequence of leaves $\ell_{t}$ through the point $(t,0,0)$ need not converge to the leaf $\ell_0$ through $(0,0,0)$.
\end{remark}

Now we are in a position to prove the key lemma of this subsection. Before starting the proof, we remark that the argument below is a higher rank analogue of the argument in \cite[Proposition 3.8]{BowdenMann19}. However, it is considerably more complicated because the leaves have higher dimension and because the $\rho_0$-action is no longer a convergence group action.% though it is Morse-Smale.

\begin{lemma}\label{lem:well-defined}
Let $\xi, \eta_1,\eta_2\in G/Q$ be such that $\Qc_{\xi,\eta_1}$ and $\Qc_{\xi,\eta_2}$ are non-empty. For $i \in \set{1,2}$, let $\zeta_i\in G/Q$ be the unique point (obtained by \cref{prop:image-unique}) such that $\mathcal Q_{\xi,\eta_i}$ is contained in a uniform neighborhood of $\wcs(\zeta_i)\cap \wcu (\eta_i)$. Then $\zeta_1=\zeta_2$.
\end{lemma}
\begin{proof}
Assume the lemma holds in the {\em special case} when $\eta_1$ and $\eta _2$ are both opposite to $\zeta_1$.  Let us first derive the lemma for $\eta _2$ possibly not  opposite to $\zeta_1$ from the \emph{special case}.
 In that case, note that $\set{\eta: \Qc_{\xi,\eta} \neq \varnothing}$ is an open set.  Also note that $\eta _1$ is always opposite to $\zeta_1$ by definition of $\zeta _1$.  Then, by \Cref{lem:opposites}, we may choose $\eta_3$ such that $\eta_3$ is opposite to both $\zeta_1$ and $\zeta_2$ and $\Qc_{\xi,\eta_3} \neq \varnothing$. Let $\zeta_3$ be such that $\Qc_{\xi,\eta_3}$ is contained in a uniform neighborhood of $\wcs(\zeta_3) \cap \wcu(\eta_3)$. Then, applying the {\em special case}  to $\eta_i$ and  $\eta_3$, we obtain that $\zeta_i=\zeta_3$ where $i \in \set{1,2}$. Thus it suffices to prove the {\em special case}.

Now suppose that we are in the special case, i.e. $\eta_1, \eta_2$ are both opposite to $\zeta_1$. Then by \cref{lem:source_sink}, there is a sequence $(\gamma_n)$ such that $\rho_0(\gamma_n)\zeta_1$ and $\rho_0(\gamma_n)\zeta_2$ converge to two distinct points in $G/Q$, while $\rho_0(\gamma_n)\eta_1$ and $\rho_0(\gamma_n)\eta_2$ converge to the same point.

We introduce the following  notation: if $\eta\in G/Q$ such that $\Qc_{\xi,\eta}$ is non-empty, let $\zeta_{\xi,\eta}\in G/Q$ be the point such that $\mathcal Q_{\xi,\eta}$ is contained in a uniform neighborhood of  $\wcu(\eta)\cap \wcs(\zeta_{\xi,\eta})$. Note that $\zeta_{\xi,\eta}$ is well-defined  by \Cref{prop:image-unique} and that $\zeta_i=\zeta_{\xi,\eta_i}$ for $i=1,2$ where $\zeta_i$ is as in the statement of the lemma.

Since $G/Q$ is compact, up to passing to a subsequence, we assume that $\xi_0:=\lim_{n \to \infty} \rho(\gamma_n) \xi$ exists. Let $\eta_0:=\lim_{n\to \infty} \rho_0(\gamma_n) \eta_1$. We will obtain a contradiction by showing that there are two distinct leaves in $X\times\{\xi_0\}$ belonging to $\til f^{-1}(\wcu(\eta_0))$.

We let $x_0\in X$ be a point contained in the minimal singular subflat whose geodesic boundary contains $\eta_0$ and $\lim_n \rho_0(\gamma_n)\zeta_{\xi,\eta_1}$. Let $C$ be the uniform constant such that whenever $\mathcal Q_{\xi',\eta'} \neq \emptyset$, $\Qc_{\xi',\eta'}$ is contained in a $C$-neighborhood $\wcs(\zeta_{\xi',\eta'}) \cap \wcu(\eta')$. Let $D$ be the ball of radius $7C$ around $x_0$.
\begin{claim*}\label{claim:1}
There exists a sequence $(\eta_n)_{n\ge 3}$ with $\eta_n\in G/Q$ such that:
\begin{enumerate}
    \item $\til f^{-1}(\wcu(\eta_n))\cap (X\times \{\rho(\gamma_n)\xi\})$ has non-empty intersection with $D\times \{\rho(\gamma_n)\xi\}$,
    \item $\lim_{n \to \infty}\zeta_{\rho(\gamma_n)\xi,\eta_n}$ is different from $\lim_{n\to \infty}\rho_0(\gamma_n)\zeta_{\xi,\eta_1}$,
    \item the sequence $(\eta_n)$ converges to $\eta_0 =\lim_{n\to \infty}\rho_0(\gamma_n)\eta_1$.
\end{enumerate}
\end{claim*}
We will postpone the proof of this claim until after finishing the proof of  Lemma \ref{lem:well-defined}.

For the latter, we proceed by contradiction. By \cref{lem:Q_foliation}, the sequence of $C^1$ manifolds $\mc{F}(\rho(\gamma_n)\xi,\eta_n)=\til f^{-1}(\wcu(\eta_n))\cap (X\times \{\rho(\gamma_n)\xi\})$ converges to $\mc{F}(\xi_0,\eta_0)=\til f^{-1}(\wcu(\eta_0))\cap (X\times \{\xi_0\})$.

We remark that, since $\til f$ covers the identity $\id:X\to X$,
$$\mc{F}(\xi_0,\eta_0)=\til f^{-1}(\wcu (\eta_0))\cap (X\times\set{\xi_0})=p\left(\wcu (\eta_0)\cap \til f(X\times \{\xi_0\})\right)\times \set{\xi_0}.$$
By the choice of $\gamma_n$ and \cref{lem:image-unique}, $\wcu (\eta_0)\cap \til f(X\times \{\xi_0\})$ is contained in a bounded neighborhood of $\wcu(\eta_0)\cap \wcs (\lim_{n\to \infty}\rho_0(\gamma_n)\zeta_{\xi,\eta_1})$. In other words, $p\left(\wcu (\eta_0)\cap \til f(X\times \{\xi_0\})\right)$ is contained in a bounded neighborhood of $p\left(\wcu(\eta_0)\cap \wcs (\lim_{n\to \infty}\rho_0(\gamma_n)\zeta_{\xi,\eta_1})\right)$.

On the other hand, $\wcu(\eta_n)\cap \til f(X\times \{\rho(\gamma_n)\xi\})$ is in a bounded neighborhood of $\wcu(\eta_n)\cap \wcs (\zeta_{\rho(\gamma_n)\xi,\eta_n})$. Set
$\Upsilon:=\lim_{n\to \infty}\zeta_{\rho(\gamma_n)\xi,\eta_n}$. Then the $\til f$ image  of the limit manifold is contained in a bounded neighborhood of $\wcu(\eta_0)\cap \wcs (\Upsilon)$. Thus, the manifold $\mc{F}(\xi_0,\eta_0)=\lim_n\mc{F}(\rho(\gamma_n)\xi,\eta_n)$ is contained in a bounded neighborhood of $p\left(\wcu(\eta_0)\cap \wcs (\Upsilon)\right)\times \{\xi_0\}$. Therefore, $p(\mc{F}(\xi_0,\eta_0))$ is contained in a bounded neighborhood of $p\left(\wcu(\eta_0)\cap \wcs (\Upsilon)\right)$. Since $\Upsilon\neq \lim_{n\to \infty}\rho_0(\gamma_n)\zeta_{\xi,\eta_1}$, and each fiber $F_x\cong G/Q$ is compact, this contradicts \cref{prop:image-unique}.
\end{proof}
\begin{proof}[Proof of Claim.]
For every $n$, we let $U_n$ be the neighborhood of $\eta_0$ defined as follows
\[U_n=\{\eta\in G/Q: \til f^{-1}(\wcu(\eta))\cap (D\times \{\rho(\gamma_n)\xi\})\neq \varnothing\}.\]

Recall that $(\gamma_n)$ is a sequence as in the proof of the above lemma (i.e. $\lim_{n\to \infty}\rho_0(\gamma_n) \zeta_1 \neq \lim_{n\to \infty}\rho_0(\gamma_n) \zeta_2$ while 
$\lim_{n\to \infty} \rho_0(\gamma_n) \eta_1= \lim_{n\to \infty}\rho_0(\gamma_n) \eta_2$). We let $\sigma_n$ be the shortest geodesic segment (with respect to a fixed Riemannian metric on $G/Q$) connecting $\rho_0(\gamma_n)\eta_1$ and $\rho_0(\gamma_n)\eta_2$. As $\rho_0(\gamma_n)\eta_1$ and $\rho_0(\gamma_n)\eta_2$ converge to the same point $\eta_0$, the segment $\sigma_n$ converges to $\eta_0$. Thus, for $n$ sufficiently large, $\sigma_n$ is contained entirely in $U_n$. We are going to pick $\eta_n$ belonging to $\sigma_n$, thus satisfying item (3). We note that $\til f|_{X\times\{\rho(\gamma_n)\xi\}}$ is a homeomorphism to its image. Thus $\{\mc{F}(\rho(\gamma_n)\xi,\nu)=\til f^{-1}(\wcu(\nu))\cap (X\times \{\rho(\gamma_n)\xi\}):\nu\in \sigma_n\}$ is a continuous family of $C^1$ submanifolds of $X\times \{\rho(\gamma_n)\xi\}$. We pick $\eta_n\in \sigma_n$ to be the first point (in the direction from $\rho_0(\gamma_n)\eta_1$ to $\rho_0(\gamma_n)\eta_2$) such that $\til f^{-1}(\wcu(\eta_n))\cap (X\times \{\rho(\gamma_n)\xi\})$ does not intersect $B_{5C}(x_0)\times \{\rho(\gamma_n)\xi\}$. If there is no such point in $\sigma_n$, then we can pick $\eta_n=\rho_0(\gamma_n)\eta_2$.

Observe that with this choice of $\eta_n$ parts (1) and (3) of the claim follow immediately because of our choices of $D$. For the second item of the claim, in the case $\eta_n=\rho_0(\gamma_n)\eta_2$, it is immediate. Thus we only need to verify the case that $\eta_n$ is the first point in $\sigma_n$ such that $\til f^{-1}(\wcu(\eta_n))\cap (X\times \{\rho(\gamma_n)\xi\})$ does not intersect $B_{5C}(x_0)\times \{\rho(\gamma_n)\xi\}$. Recall that the manifold $\til f^{-1}(\wcu(\eta_n))\cap (X\times \{\rho(\gamma_n)\xi\})$ is in a $C$-neighborhood of $p(\wcu(\eta_n)\cap \wcs(\zeta_{\rho(\gamma_n)\xi,\eta_n}))\times \{\rho(\gamma_n)\xi\}$ and the manifold $\til f^{-1}(\wcu(\rho_0(\gamma_n)\eta_1))\cap (X\times \{\rho(\gamma_n)\xi\})$ is in a $C$-neighborhood of $p(\wcu(\rho_0(\gamma_n)\eta_1)\cap \wcs(\rho_0(\gamma_n)\zeta_{\xi,\eta_1}))\times \{\rho(\gamma_n)\xi\}$. By the triangle inequality of the Hausdorff pseudo-distance, 
\[d^{\Haus}(p(\wcu(\eta_n)\cap \wcs(\zeta_{\rho(\gamma_n)\xi,\eta_n}))\cap D, p(\wcu(\rho_0(\gamma_n)\eta_1)\cap \wcs(\rho_0(\gamma_n)\zeta_{\xi,\eta_1}))\cap D)\ge 3C,\]
for every $n$. It follows that their limits are distinct. The limit of $p(\wcu(\eta_n)\cap \wcs(\zeta_{\xi,\eta_n}))$ is the manifold $p(\wcu(\eta_0)\cap \wcs(\lim_n \zeta_{\rho(\gamma_n)\xi,\eta_n})$. The limit of $p(\wcu(\rho_0(\gamma_n)\eta_1)\cap \wcs(\rho_0(\gamma_n)\zeta_{\xi,\eta_1}))$ is the manifold $p(\wcu(\eta_0)\cap\wcs(\lim_{n}\rho_0(\gamma_n)\zeta_{\xi,\eta_1}))$. As the two limit manifolds are different, $\lim_{n \to \infty}\rho_0(\gamma_n)\zeta_{\xi,\eta_1}\neq \lim_{n\to \infty}\rho_0(\gamma_n)\zeta_{\xi,\eta_1}$. This completes the proof of the claim.
\end{proof}

 Using the previous lemma, we can now define a continuous $(\rho,\rho_0)$-equivariant map whenever $\rho$ is sufficiently close to $\rho_0$.

\begin{proposition}\label{prop:semiconj}
Suppose $\rho$ is sufficiently close to $\rho_0$. Then there is a $(\rho,\rho_0)$-equivariant map $\varphi_{\rho}: G/Q\to G/Q$.
\end{proposition}

\begin{proof}
 The map $\varphi_\rho$ is defined as follows. For every $\xi\in G/Q$, let $\eta\in G/Q$ be such that $\mathcal Q_{\xi,\eta}$ is not empty. By \cref{prop:image-unique}, there is a unique $\zeta\in G/Q$ such that $\mathcal Q_{\xi,\eta}$ is contained in a bounded neighborhood of $\wcu(\eta)\cap \wcs(\zeta)$. We then define $\varphi_{\rho}(\xi)=\zeta$. By \cref{lem:well-defined}, this map is well-defined as $\zeta$ does not depend on the choice of $\eta$. 

Since $\til f$ is $(\wh \rho,\wh \rho_0)$-equivariant,
\[\mathcal Q_{\rho(\gamma)\xi,\rho_0(\gamma)\eta}=\til f(X\times \{\rho(\gamma)\xi\})\cap \wcu(\rho_0(\gamma)\eta)=\wh\rho_0(\gamma)(\til f(X\times \{\xi\})\cap \wcu (\eta))=\wh\rho_0(\gamma)\mathcal Q_{\xi,\eta}.\]
Thus, if $\mathcal Q_{\xi,\eta}$ is contained in a uniform neighborhood of $\wcu(\eta)\cap \wcs(\zeta)$ then $\mathcal Q_{\rho(\gamma)\xi,\rho_0(\gamma)\eta}$ is contained in a uniform neighborhood of $\wh{\rho}_0(\gamma)(\wcu(\eta)\cap \wcs(\zeta))=\wcu(\rho_0(\gamma)\eta)\cap \wcs(\rho_0(\gamma)\zeta)$. It follows that $\varphi_{\rho}(\rho(\gamma)\xi)=\rho_0(\gamma)\varphi_{\rho}(\xi)$.
\end{proof}

\subsection{Local semirigidity and continuity of $\varphi _{\rho}$}\label{sec:local}

\begin{proof}[{\bf Proof of Theorem \ref{thm:rigidity_G/Q}}]
We let $\varphi_{\rho|_{\Gamma_0}}:G/Q\to G/Q$ be the map obtained from \cref{prop:semiconj}. We will show in \cref{lem:cts} and \cref{lem:stab}
 that $\varphi_{\rho|_{\Gamma_0}}$ is continuous and tends to $\id$ when $\rho\to \rho_0$.

 The surjectivity of $\varphi_{\rho|_{\Gamma_0}}$ then follows from minimality of the ${\rho_0|_{\Gamma_0}}$-action, the $({\rho|_{\Gamma_0}},{\rho_0|_{\Gamma_0}})$-equivariance and the continuity of $\varphi _{\rho|_{\Gamma_0}}$. Since ${\rho_0|_{\Gamma_0}}$ is a uniformly expanding action, we can now apply \cref{lem:semi_unique} to prove the uniqueness of $\varphi_{\rho|_{\Gamma_0}}$ semi-conjugacy provided it is close enough to $\id$ and $\rho$ is close enough to $\rho_0$.
\end{proof}

For the next two lemmas, to simplify notations, we assume that the lattice $\Gamma$ is torsion-free.
\begin{lemma} \label{lem:cts}
Suppose $\rho$ is sufficiently close to $\rho_0$.  Then the map $\varphi _{\rho}$ is continuous.
\end{lemma} 

\begin{proof} Let $\xi_0\in G/Q$. We pick $\eta\in G/Q$ such that $(\xi_0,\eta)\in \mathcal U$, or equivalently $\mathcal Q_{\xi_0,\eta}\neq \varnothing$. If $(\xi_n)$ is a sequence converging to $\xi_0$ then $(\xi_n,\eta)\in \mathcal U$ for sufficiently large $n$ as $\mathcal U$ is an open set. By \cref{lem:continuityQ}, the sequence $\mathcal Q_{\xi_n,\eta}$ converges to $\mathcal Q_{\xi_0,\eta}$ in the pointed Hausdorff topology. Let $x_0\in \mathcal Q_{\xi_0,\eta}$. Then for every $\epsilon>0$ and for every $R>0$, there exists $n_0$ such that for every $n>n_0$,
\[d^{\Haus}(\mathcal Q_{\xi_n,\eta}\cap B(x_0,R), \mathcal Q_{\xi_0,\eta}\cap B(x_0,R))<\epsilon.\]
On the other hand, by \cref{prop:image-unique}, there is $C>0$ such that $\mathcal Q_{\xi_n,\eta}$ is contained in the $C$-neighborhood of $\wcu(\eta)\cap \wcs(\varphi(\xi_n))$, and $\mathcal Q_{\xi_0,\eta}$ is contained in the $C$-neighborhood of $\wcu(\eta)\cap \wcs(\varphi(\xi_0))$. In $\mathcal Q_{\xi_n,\eta}$ and $\mathcal Q_{\xi_0,\eta}$, we pick two biLipschitz flats $BF_n$ and $BF_0$ whose intersections with $B(x_0,R)$ are $\epsilon$-close. By \cref{lem:quasiflat} and \cref{prop:geometric-orbit}, when $\rho$ is close enough to $\rho_0$ so that the constant $L$ given there is close enough to $1$, the biLipschitz flat $BF_n$ is $C$-close to a flat $F_n$ in $\wcu(\eta)\cap \wcs(\varphi(\xi_n))$, and the biLipschitz flat $BF_0$ is $C$-close to a flat $F_0$ in $\wcu(\eta)\cap \wcs(\varphi(\xi_0))$. Thus, intersections of flats $F_n$ and $F_0$ with the ball $B(x_0,R)$ are $(C+\epsilon)$-Hausdorff close. In other words, center manifolds $\wcu(\eta)\cap \wcs(\varphi(\xi_n))$ and $\wcu(\eta)\cap \wcs(\varphi(\xi_0))$ contain flats that are uniformly close on large balls for sufficiently large $n$. 
As the backward endpoints are $\eta$ in each case, it follows that the center manifolds are close in the pointed Hausdorff topology.
Therefore, $\varphi(\xi_n)$ converges to $\varphi(\xi_0)$. This shows the continuity of $\varphi$.
\end{proof}

\begin{lemma} \label{lem:stab}
The map $\varphi_\rho$ tends to $\id$ uniformly as $\rho\to \rho_0$.
\end{lemma}
\begin{proof}
We will  use the subscript $\rho$ for  $\til f$ and $\mathcal Q$ to indicate the dependence on the action $\rho$. In particular, we will use the notations $\til f_\rho$ and $\mathcal Q_{\rho,\xi,\eta}=\til f_{\rho}(X \times \xi) \cap \wcu(\eta)$. Note that $\mathcal Q_{\rho_0,\xi,\eta}=\wcs(\xi)\cap \wcu(\eta)$ and $\varphi_{\rho_0}=\id$. 
 Then we have point-wise convergence of $\varphi_{\rho_n}$ to $\varphi_{\rho_0}$ for any sequence $\rho_n \to \rho_0$. Indeed, if $\rho_n \to \rho_0$, then \cref{lem:continuity-of-transverse-intersection} implies that $\varphi_{\rho_n}(\xi) \to \varphi_{\rho_0}(\xi)=\xi$. We will prove that the convergence is uniform in $\rho$, i.e. for any $\eps>0$, there exists a neighborhood of $\Uc = \Uc(\eps)$ of $\rho_0$ such that $\sup_{\xi \in G/Q} d(\varphi_{\rho}(\xi),\xi)<\eps$ for any $\rho \in U$.

We let $D$ be the chosen fundamental domain in the construction of $\til f_\rho$ in \cref{prop:BMMforG/Q} and \cref{prop:BM}. Let $x_0$ be a point in $D$. Let $d_{G/Q}$ be any Riemannian metric on $G/Q$. Since $G/Q$ is compact, the choice of a Riemannian metric on $G/Q$ will not affect the argument.

For every $\xi\in G/Q$, pick  $\eta_{x_0,\xi}\in G/Q$ such that the  Weyl chamber $Q$-face $\mathcal C_{x_0}(\eta_{x_0,\xi})$ based at $x_0$, is opposite to the  Weyl chamber $Q$-face $\mathcal C_{x_0}(\xi)$ based at $x_0$. Observe that $\wcs(\xi)\cap \wcu(\eta_{x_0,\xi})\neq \varnothing$. 

\begin{claim*} \label{claim}
There exist constants $\epsilon', d'>0$ such that for every $\epsilon<\epsilon'$, there exists a neighborhood $\Uc = \Uc(\eps)$ of $\rho_0$  such that: for every $\rho\in \Uc$, the set $\mathcal Q_{\rho,\xi,\eta_{x_0,\xi}} \neq \varnothing$  and contains a point of distance at most $d'\epsilon$ from $p^{-1}(x_0) \cap \wcs(\xi)$. Furthermore, the constants $\epsilon'$, $d'$ and the neighborhood $\Uc$ are independent of $x_0$ and $\xi$.
\end{claim*} 
\begin{proof}[Proof of Claim]
Set $\eta:=\eta_{x_0,\xi}$. Let $\bar{x_0} \in \wcs(\xi)\cap \wcu(\eta)$ be such that $p(\bar{x_0})=x_0$.
Let $m=\dim(X)$ and $k=\dim(G/Q)$. Pick a smooth local trivialization $\Psi_{\bar{x_0}}: U_{\bar{x_0}} \to (-1,1)^m \times (-1,1)^k$ where $U_{\bar{x_0}}$ is a neighborhood of $\bar{x_0}$. We can assume that $\Psi_{\bar{x_0}}(x_0)=0^{m+k}$.

In this local trivialization, $\wcs(\xi)$ and $\wcu(\eta)$ are graphs $G_1$ and $G_2$ of two smooth functions defined on $(-1,1)^m$. Moreover, $G_1$ and $G_2$  intersect transversely along the center manifold, whose dimension is $m-k$. Let $\pi_k: (-1,1)^m \to (-1,1)^k$ be the projection to the first $k$ coordinates.
Then we can choose smooth fiber-preserving coordinate change such that 
\[G_1=\set{(v,-\pi_k(v)): v \in (-1,1)^m}\] and 
\[G_2=\set{(v,\pi_k(v)): v \in (-1,1)^m}.\]
Then $G_1 \cap G_2=\set{(0^k,w,0^k): w \in (-1,1)^{m-k}}$. 

Observe that for any $\epsilon_1$ sufficiently small, there exist $\epsilon_1'$ such that $$\Nbhd_{\epsilon_1}(G_1)\cap \Nbhd_{\epsilon_1}(G_2)\subset \Nbhd_{\epsilon_1'}(G_1\cap G_2).$$
Now suppose $G_3$ is the graph of a continuous function defined on $(-1,1)^m$ and suppose that $G_3$ is $\eps_2$ close to $G_1$ for some $\eps_2>0$. In particular, this means that $G_3=\set{(v,-\pi_k(v)+q(v)): v \in (-1,1)^m}$ where $q: (-1,1)^m \to (-1,1)^k$ is a continuous function for which $\norm{q(v)}< \eps_2$. We claim that there exists $\delta<\eps_1$ such that, whenever $\eps_2<\delta$ , then $G_3 \cap G_2 \neq \varnothing$ and $G_3 \cap G_2$ contains a point that is $\eps_1'$ close to the image of $\bar{x_0}$ in this trivialization. 
 
To prove this claim, note that  $G_3\cap G_2\neq \emptyset$ if and only if $q(v)=2\pi_{k}(v)$ for some $v\in (-1,1)^m$. Suppose this claim is false. Then there is a map $h:(-1,1)^m\to S^{k-1}(1)$ to the unit sphere $S^{k-1}(1)$ given by 
\[h(v)=\frac{-q(v)+2\pi_{k}(v)}{\norm{-q(v)+2\pi_{k}(v)}}.
\]
Let us restrict $h$ to the subset $S:=S^{k-1}(4\eps_2) \times 0^{m-k} \subset (-1,1)^m$, which is homeomorphic with $S^{k-1}$.  When $\epsilon_2$ is sufficiently small, $h|_{S}$ is close, and thus homotopic, to the $\frac{1}{4\eps_2}$-homothety $S^{k-1}(4\eps_2)\to S^{k-1}(1)$. 

Thus $h|_{S}$ has degree one. This is impossible as $h$ extends to a map of the $4\epsilon_2$-ball in $\Rb^k$ to $S^{k-1}$ and hence must have degree zero. This shows that $G_3 \cap G_2 \neq \varnothing$. 

Moreover, the above arguments further show that in fact, there exists $(w,0^{m-k}) \in (-1,1)^m$ with $\norm{w}<4\eps_2$ for which $q(w,0)=2w$. In particular, this implies that $G_3 \cap G_2$ contains a point at distance at most $4\epsilon_2$ to $0^{m+k}$.

Now we will explain how this above discussion proves the claim. Consider a sufficiently small neighborhood $\Uc'$ of $\rho$. For $\rho \in \Uc'$, set $G_3=\Psi_{\bar{x_0}}(\til f_{\rho}(X \times \xi))$. Set $\eps':=\delta$. Given any $\eps<\eps'$, set $\eps_2:=\eps$. Then choose a neighborhood $\Uc= \Uc(\eps) \subset \Uc'$ of $\rho_0$ such that $G_3$ is $\eps_2$ close to $G_1$.  Then $\Psi_{\bar{x_0}}(\Qc_{\rho, \xi,\eta_{x_0,\xi}} \cap \: U_{\bar{x_0}})=G_3 \cap G_2$ is non-empty and contains a point at a distance at most $4 \eps_2$ from $0=\Psi_{\bar{x_0}}(\bar{x_0})$. Then $\Qc_{\rho, \xi,\eta_{x_0,\xi}}$ is non-empty and contains a point within a distance $d'\eps_2$ from $\bar{x_0}$ where $\eps=\eps_2<\delta(=\eps')$ and $d'$ is a constant that depends on the chart $\Psi_{x_0}$. Thus the first part of the claim holds. But as $G$ acts transitively on $X_{G/Q}$ by isometries,  we can pre-compose this fixed chart $\Psi_{\bar{x_0}}$ with appropriate isometries and assume that the constants $\epsilon', d'$ and the neighborhood $\Uc$ are independent of $x_0$ and $\xi$.
\end{proof}

Now we use the Claim to prove the lemma. First fix $\eps',d'$ as in the Claim above. 

\cref{prop:BMMforG/Q} part (3) implies that for every $\eps> 0$, there is an open neighborhood $U = U(x_0,\xi)$ of $\rho_0$ such that for every $\rho\in U$, we have 
\[\sup_{\xi \in G/Q} d_{G/Q}(\pi_{x_0}(\til f_\rho(x_0,\xi)),\xi)<\eps. \]

Choose $\eps <\eps'$ small enough  so that $d_{G/Q}(\pi_{x_0}(\til f_\rho(x_0,\xi),\xi)<\eps$ implies 
that $\wcs(\xi)$ and $\wcs(\pi_{x_0}(\til f(x_0,\xi)))$ are sufficiently close on a ball. More precisely, for every $R>0$ and $\epsilon>0$, there exists a neighborhood $U$ of $\rho_0$ such that for every $\rho\in U$,
\[ \sup_{\xi \in G/Q} d_{\Haus} \left( \wcs(\xi)\cap p^{-1}(B_R(x_0)),\wcs \left( \pi_{x_0}(\til f_\rho(x_0,\xi)) \right) \cap p^{-1}(B_R(x_0) \right) <\eps.\]

On the other hand, \cref{prop:BMMforG/Q} part (4) implies that for any $\eps, R >0$, there is a neighborhood $U$ of $\rho_0$ such that for every $\rho \in U$,
\[ \sup_{y\in \bar{B_{R}(x_0)}} \sup_{\xi \in G/Q} d \left( \til f_\rho(y,\xi),\wcs \left( \pi_{x_0}(\til f_\rho(x_0,\xi)) \right)  \cap p^{-1}(y)\right)<\eps.\]

Thus triangle inequality implies that for any $\eps,R>0$, there exists a neighborhood $U$ of $\rho_0$ such that for any $\rho \in U$, 
\begin{align} 
\label{eqn:uniform_fiberwise_comparison_f_tilde}
\sup_{y\in \bar{B_R(x_0)}} ~\sup_{\xi \in G/Q} ~ d \left( \til f_\rho(y,\xi),\wcs(\xi)\cap p^{-1}(y) \right) <2\eps < 2 \eps'. 
\end{align}

It follows that $\til f_\rho(\bar{B_R(x_0)},\xi)\cap \wcu(\eta_{x_0,\xi})$ is contained in a $2\epsilon$-neighborhood of $\wcs(\xi)\cap \wcu(\eta_{x_0,\xi}))$.
On the other hand, by \cref{prop:image-unique},  $\til f_\rho(X\times\{\xi\})\cap \wcu(\eta_{x_0,\xi})$ is contained in a $C$-neighborhood of $\wcs(\varphi_\rho(\xi))\cap \wcu(\eta_{x_0,\xi})$, where $C$ is independent of $\rho$ in a sufficiently small neighborhood of $\rho_0$.

We fix $R>4C$ and $\eps<\eps'$. Let $U$ be the neighborhood of $\rho_0$ as in \cref{eqn:uniform_fiberwise_comparison_f_tilde}. By further shrinking $U$ if necessary, we can assume that $L$ in \cref{prop:geometric-orbit} is sufficiently close to 1. We identify $\wcu(\eta_{x_0,\xi})$ with the symmetric space $X$. All considered sets and points are subsets and points of $\wcu(\eta_{x_0,\xi})$. Thus they can be considered as subsets and points of $X$. By the Claim  above there exists $x_0'\in \mathcal Q_{\rho,\xi,\eta_{x_0,\xi}}$ of distance at most $d'\epsilon$ from $x_0$. Thus $x_0$ is in a $(C+d'\epsilon)$-neighborhood of the center manifold $\wcs(\varphi_\rho(\xi))\cap \wcu(\eta_{x_0,\xi})$. 

We denote by $\gamma_{x_0,\xi}$ the geodesic ray from $x_0$ in the center direction of the Weyl chamber $Q$-face $\xi$ based at $x_0$. We let $y_R\in \gamma_{x_0,\xi}$ be a point of distance $\frac R 2$ from $x_0$. Let $z_R\in \mathcal Q_{\rho,\xi,\eta_{x_0,\xi}}$ that is at most $d'\epsilon$ away from $y_R$. We let $BF\subset \mathcal Q_{\rho,\xi,\eta_{x_0,\xi}}$ be an $L$-biLipschitz flat  containing $z_R$. Let $F$ be the unique flat in $X$ that is at most $C$-Hausdorff distance from $BF$, see \cref{lem:quasiflat}. The flat $F$ is asymptotic to  $\eta_{x_0,\xi}$ and $\varphi_\rho(\xi)$ as opposite Weyl chamber $Q$-faces. 

Let $w_R\in F$ be a point of distance at most $C$ from $z_R$.  Recall that $x_0'\in \mathcal Q_{\rho,\xi,\eta_{x_0,\xi}}$ is at most $d'\epsilon$ far away from $x_0$. Let $x_0''\in F$ be a point of distance at most $C$ from $x_0'$. We denote by $\gamma_{x_0,\eta_{x_0,\xi}}$ and $\gamma_{x_0'',\eta_{x_0,\xi}}$ the geodesic rays in the center directions of Weyl chamber $Q$-faces $\eta_{x_0,\xi}$ based at $x_0$ and $x_0''$, respectively. We note that the rays $\gamma_{x_0,\eta_{x_0,\xi}}$ and $\gamma_{x_0'',\eta_{x_0,\xi}}$ are at  Hausdorff distance at most $C+d'\epsilon$. Let $\gamma_{w_R,\eta_{x_0,\xi}}$ be the geodesic ray initiating from $w_R$ in the center direction of Weyl chamber $Q$-face $\eta_{x_0,\xi}$. The two geodesic rays $\gamma_{y_R,\eta_{x_0,\xi}}$ and $\gamma_{w_R,\eta_{x_0,\xi}}$ are $(C+d'\epsilon)$-Hausdorff close. It follows that, in the flat $F$, the two parallel geodesic rays $\gamma_{w_R,\eta_{x_0,\xi}}$ and $\gamma_{x_0'',\eta_{x_0,\xi}}$ have distance at most $2(C+d'\epsilon)$. Hence, the geodesic segment $[x_0'', w_R]$ is contained in a $2(C+d'\epsilon)$-neighborhood of geodesic rays initiating from $x_0''$ in the center direction of $\varphi_\rho(\xi)$. Therefore, the intersections of the Weyl chamber $Q$-faces $\xi$ and $\varphi_\rho(\xi)$ based at $x_0$ with the ball radius $\frac R 3$ are $4(C+d'\epsilon)$-Hausdorff close. We note that $R$ is independent of $\xi$.
As $R>4C$ is arbitrary, we can let $R \to \infty$ when the neighborhood $U$ of $\rho_0$ shrinks further. Thus as $\rho\to \rho_0$, the semiconjugacy $\varphi_\rho$ converges uniformly to $\varphi_{\rho_0}=\text{id}$.
\end{proof}

\section{Applications of the Main Theorem: Differentiable and Lipschitz  Local Rigidity}\label{sec:Applications}

In this section, we give quick applications of our main result to local rigidity of projective actions of uniform lattices in higher rank in both the  Lipschitz and $C^1$-topology. 
This yields well-known results by Kanai \cite{Kanai} and Katok-Spatzier \cite{KatokSpatzier97} for the $C^1$-topology and a recent one by  Kapovich-Kim-Lee \cite{KapovichKimLee19} in the Lipschitz topology. The applications are deduced from  \cref{thm:rigidity_G/Q} and the fact that the $\rho_0$ action of $\Gamma$ on $G/Q$ by left multiplication is a uniformly expanding action.

We first prove \cref{cor:LipRigidityIntro} which we restate here from the introduction. This recovers the result of \cite{KapovichKimLee19} for Lipschitz neighborhoods of the standard action.

\begin{repcorollary}{cor:LipRigidityIntro}
    Let $\Gamma<G$ be a uniform lattice in a connected linear semisimple Lie group $G$ of higher rank and without compact factors. Let $Q$ be a parabolic subgroup of $G$, and let $\rho_0$ be the action of $\Gamma$ on $G/Q$ by left translation. Then there is a neighborhood $U$ of $\rho_0$ in Lipschitz topology such that every $\rho\in U$ is $C^0$-conjugate to $\rho_0$.
\end{repcorollary}
\begin{proof}
    By \cref{thm:rigidity_G/Q}, there is a neighborhood $U$ of $\rho_0$ such that every $\rho\in U$, there exists a semi-conjugacy $\varphi_\rho$. \cref{prop:removesemi} shows that $\varphi_\rho$ is a conjugacy.
\end{proof}

If we restrict to $C^1$ neigborhoods of the standard action, we obtain the following corollary, which was proved by Katok and Spatzier \cite[Theorem 17]{KatokSpatzier97} and by Kanai under additional assumptions \cite{Kanai}. Similar assertions have been conjectured for semisimple Lie groups of rank 1 but remain unresolved.

\begin{repcorollary}{cor:ks}
    Let $\Gamma<G$ be a uniform lattice in a connected linear semisimple Lie group $G$ of higher rank and without compact factors. Let $Q$ be a parabolic subgroup of $G$, and let $\rho_0$ be the action of $\Gamma$ on $G/Q$ by left translation. Then there is a neighborhood $U$ of $\rho_0$ of smooth actions that are close to $\rho_0$ in $C^1$-topology such that every $\rho\in U$ is $C^\infty$-conjugate to $\rho_0$.
\end{repcorollary}
\begin{proof}
By \cref{thm:rigidity_G/Q}, there is such a neighborhood $U$ of $\rho_0$ such that every $\rho\in U$ is semi-conjugate to $\rho_0$. We claim that the semi-conjugacy maps are actually conjugacies. Indeed, this follows immediately from \cref{prop:removesemi}.

Now apply \cite[Theorem 1.2]{GorodnikSpatzier18} to the $C^0$ conjugacies to conclude that the conjugacies are $C^\infty$.
\end{proof}

\section{Actions on the Geodesic Boundary}
\label{sec:action_on_geod_bdry}
While in rank one symmetric spaces, geodesic boundaries coincide with Furstenberg boundaries, the situation in higher rank is different. The geodesic boundary of a higher rank symmetric space $X$ is diffeomorphic to a sphere $S^{\dim(X)-1}$, and always has dimension strictly larger than the dimension of any Furstenberg boundary. Our \cref{thm:main_thm_lattice} shows that local semi-rigidity hold for standard actions on Furstenberg boundaries. However, we now prove \cref{thm:geod_bdy_not_rigid} which shows that local semi-rigidity fails for standard actions on geodesic boundaries. First let us recall the statement.

\begin{reptheorem}{thm:geod_bdy_not_rigid}
    Let $\Gamma$ be a lattice for a  higher rank symmetric space of noncompact type $X$, not necessarily  uniform. Then the standard action of $\Gamma$ on $\partial_\infty X$ is not locally semi-rigid among $C^k$ actions for any $k\in [0,\infty]$.
\end{reptheorem}
Before we start to construct perturbed actions that are not semi-conjugate to the standard ones, let us explain the structure of geodesic boundaries of higher rank symmetric spaces. 

The geodesic boundary $\partial_\infty X$ of $X$ has the structure of a spherical building. It is the union of closed Weyl chamber simplices of dimension $\op{rk}(X)-1$ where $\op{rk}(X)$ is the rank of the symmetric space. The set of  Weyl chamber simplices can be identified with the Furstenberg boundary $G/P$, where $G$ is a semisimple Lie group associated to $X$ and $P$ is a minimal parabolic subgroup of $G$. On $\partial_\infty X$ there is a visual topology and a Tits metric. All Weyl chamber simplices are isometric with respect to the Tits metric. We fix a model Weyl chamber $W$. Then there is a continuous map with respect to visual topology $\phi:\partial_\infty X\to W$ such that $\phi$ is an isometry with respect to the Tits metric on each Weyl chamber simplex. If we identify $W$ with a Weyl chamber in $\partial_\infty X$ and identify $\partial_\infty X= \cup_{g\in G}gW$, then the map $\phi$ coincides with translating points in $\partial_\infty X$ into $W$ by elements of $G$.
Indeed, $G$ acts transitively on the Weyl chambers with stabilizer $P$, and $P$ itself acts by the identity on $W$.

On $G/P$, we let $\mu$ be a  Haar measure.  The action of $G$ on $G/P$ preserves the measure class of $\mu$. We thus have the following smooth Radon-Nikodym derivative cocycle $c:G\times G/P\to \mathbb R$, defined by
\[c(g,x)=\frac{dg_*\mu}{d\mu}(x).\]
Since $G$ does not preserve a probability measure on $G/P$, this cocycle is not trivial.
Note also that for every $\alpha\in [0,+\infty)$, $c^\alpha$ is also a smooth cocycle.
\vspace{.2em}

\subsection*{Construction of perturbed actions.} We construct  new $G$-actions with the following properties:
\begin{enumerate}
    \item each $g \in G$ maps each closed Weyl chamber simplex to a closed Weyl chamber simplex by a homeomorphism ($C^k$-diffeomorphism resp.),
    \item the induced action on $G/P$ is the standard action $\rho_0$ of $G$ on $G/P$.
\end{enumerate}
\vspace{.2em}

We thus only need to define how $G$ acts on the model Weyl chamber simplex $W$. For this, first we choose a center $w\in \op{int}(W)$, and pick a continuous homeomorphism $\tau: W\to \mathbb R^{\op{rk}(X)-1}\cup \partial_\infty \mathbb R^{\op{rk}(X)-1}$ with $\tau(w)=0$. For every $\delta\neq 0$, we let $D_\delta:\mathbb R^{\op{rk}(X)-1} \to \mathbb R^{\op{rk}(X)-1}$ be the linear dilation by factor $\delta$. It is clear that $D_\delta$ extends to a homeomorphism on $\mathbb R^{\op{rk}(X)-1}\cup \partial_\infty \mathbb R^{\op{rk}(X)-1}$ that fixes $\partial_\infty \mathbb R^{\op{rk}(X)-1}$ pointwise. By pulling back by $\tau$, i.e. the map $\tau^*D_{\delta}=\tau^{-1}\circ D_{\delta}\circ \tau$, we obtain a homeomorphism on $W$ that fixes $w$ and $\partial W$ pointwise. Since points on $\partial W$ are fixed, $\tau^*D_{\delta}$ extends to a homeomorphism, still denoted by $\tau^*D_{\delta}$, from $W \cup \partial W$ into itself by fixing every point in $\partial W$ and $w$.

Now let $g\in G$ and $\mathcal C\subset \partial_\infty X$ be a Weyl chamber. We define the homeomorphism $\kappa_g^\alpha: \mathcal C\to g\mathcal C$ such that $$\phi\circ \kappa_g^\alpha=(\tau^*D_{c(g,\mathcal C)^\alpha})\circ \phi,$$ where $\mathcal C$ here is regarded as an element of $G/P$. In other words, $\kappa_g^\alpha$ from $\mathcal C$ to $g\mathcal C$ is just the pulling back of $\tau^*D_{c(g,\mathcal C)^\alpha}$ by the model map $\phi$. 

Explicitly, for $g\in G$, the map $\kappa_g^\alpha:\partial_\infty X\to \partial_\infty X$ defined as follows: if $x\in hW$ for an $h\in G$, then
\[\kappa_g^\alpha(x)=gh\cdot (\tau^*D_{c(g,hW)^\alpha}(h^{-1}\cdot x))=gh\cdot (\tau^*D_{c(g,hW)^\alpha}(\phi (x))).\]
To extend $\kappa_g^\alpha$ to all of $\partial_\infty X$,  note that $\kappa_g^\alpha$ agrees with the standard action on Weyl chamber faces.   Indeed, the map  $\tau^*D_{c(g,hW)^\alpha}$ is the identity on the Weyl chamber faces.  This matters since the  cocycle $c(g,hW)$ is not constant on the star of a face.
From this it follows easily  that $\kappa_g^\alpha$ is a homeomorphism.

For any fixed $k\in [0,\infty]$, we observe that, for a good choice of $\tau$, the map $\tau^*D_{\delta}$ is even $C^k$ and agrees with the standard action on $\partial W$ up to order $k$. Hence the action $g\mapsto \kappa_g^\alpha$ is also $C^k$.

\begin{remark} 
\label{rem:semi}
We can use  flows other than the $D_{\delta}$ on the model Weyl chamber to define actions of $G$ on $G/P$ as long as the flows are the identity on the Weyl chamber faces.  We suspect that such flows could lead to more semi-conjugacy classes. 
\end{remark}

\begin{proof}[{\bf Proof of Theorem \ref{thm:geod_bdy_not_rigid}}]
    Fix $k\in[0,\infty]$.
    The map $\rho_\alpha: g\mapsto \kappa_g^\alpha\in \op{Diff}^k(\partial_\infty X)$ constructed above defines an action of $G$ on $\partial_\infty X$ by $C^k$-diffeomorphisms. Furthermore, when $\alpha$ varies from 0 to $+\infty$, we obtain a continuous deformation of the standard action. The actions of $\Gamma$ are just the restrictions of $G$-actions to $\Gamma$. 

    We prove that for $0\neq \alpha$, then the action $\rho_{\alpha}$ is not semi-conjugate to $\rho_0$. Suppose, for the sake of contradiction, that there exists a continuous and surjective map $\varphi:\partial_\infty X\to \partial_\infty X$ such that $\rho_{0}\circ \varphi =\varphi \circ \rho_{\alpha}$. We first need the following lemma.

\begin{lemma}\label{lem:phi_const_C}
Let $\mathcal C$ be a Weyl chamber such that there exists $\gamma\in \Gamma$ stabilizing $\mathcal C$ with $c(\gamma, C)\neq 1$. Then $\phi(\varphi(\mc{C}))=\{\phi(\varphi(c))\}$  where $c$ is the center of $\mathcal C$.
\end{lemma}

\begin{proof}
Let $c\in \mc{C}$ be the center, namely the point such that $\phi(c)=w\in W$. For all $\alpha\in[0,\infty)$, by construction $\rho_\alpha(\gamma)$ stabilizes $\mc{C}$ since it has the same quotient action on $G/P$ as $\rho_0$. Moreover $c$ and $\partial\mc{C}$ are pointwise fixed by both $\rho_\alpha(\gamma)$ and $\rho_0(\gamma)$. Hence the same conclusions hold for $\rho_\alpha(\gamma^n)$ for $n\in \Z$. 

For any other point $x\in \op{int}(\mc{C})$, after possibly swapping $\gamma$ with $\gamma^{-1}$, we have $\lim_{n\to\infty}\rho_\alpha(\gamma^n)(x)=c$. Hence by continuity 
$\varphi(c)=\lim_{n\to\infty}\varphi(\rho_\alpha(\gamma^n)x) =\lim_{n\to\infty}\rho_0(\gamma^n)\varphi(x)$. It follows that,
\[\phi \circ \varphi(c)=\lim_{n\to\infty}\phi \circ \rho_0(\gamma^n)\varphi(x)=\phi \circ \varphi(x),\]
for every $x\in \op{int}(\mathcal C)$. By continuity of $\phi\of\varphi$, $\phi(\varphi(\mc{C}))=\{\phi(\varphi(c))\}$.
\end{proof}

We continue the proof of Theorem \ref{thm:geod_bdy_not_rigid}. Let $\mathcal C$ be a Weyl chamber as in \cref{lem:phi_const_C}. Such $\mathcal C$ and associated $\gamma$ always exist by work of Prasad and Raghunathan \cite[Theorem 2.8]{Prasad-Raghunathan1972}.
Then for every $\delta\in \Gamma$, $\delta\cdot \mathcal C$ also satisfies the assumption of \cref{lem:phi_const_C}. By the minimality of $\Gamma$-action on $G/P$, \cref{lem:phi_const_C} holds for every Weyl chamber.

Let $\mathcal C_1$ and $\mathcal C_2$ be two adjacent Weyl chambers and let $c_1$ and $c_2$ be their centers. We have that $\phi(\varphi(\mc{C}_1))=\{\phi(\varphi(c_1))\}$ and $\phi(\varphi(\mc{C}_2))=\{\phi(\varphi(c_2))\}$. Since $\varphi(\mc{C}_1)$ and $\varphi(\mc{C}_2)$ have non-empty intersection, the point $\phi(\varphi(c_1))$ coincides with $\phi(\varphi(c_2))$.

Note that for any pair of Weyl chambers, there exists a chain (gallery) of Weyl chambers from one to the other such that any two consecutive chambers in the chain are adjacent. It follows that there exists $z\in W$ such that $\phi(\varphi(\mathcal C))=\{z\}$ for every Weyl chamber $\mathcal C$. Since $\varphi$ is surjective, we obtain $\phi(\partial_\infty X)=\phi(\cup_{\mathcal C}\varphi(\mathcal C))=\{z\}$, a contradiction. Therefore, the action $\rho_\alpha$ is not semiconjugate to $\rho_0$.
\end{proof}

\appendix
\section{Barycenters and Estimates}\label{sec:derivative_barycenter}
For the rest of this section, let us fix a compact, connected $C^k$ manifold $F$ where $k \geq 3$. We will study how the barycenter maps on $F$ vary as we vary measures and Riemannian metrics on $F$ along a parametrized family.

\subsection{Two lemmas from differential geometry}
The goal of this section is to establish two basic lemmas that will be used in the next section for estimating derivatives of the barycenter map. In this section, assume that $F$ is equipped with a Riemannian metric $g$ with associated norm $\norm{\cdot}$. We will denote by $\nabla$ the Riemannian connection on $F$ and by $D$ the covariant derivative operator, namely if $z \in F$ and $X, Y$ are smooth vector fields on $F$, then $D_zY(X)=(\nabla_X Y)(z)$. Moreover, if $f:F \to \Rb$ is a smooth function, then $\nabla f$ is the vector field dual to the 1-form $Df$, i.e. for any $v \in T_zF$, $g(\nabla_zf,v)=D_zf(v)$.

\begin{lemma}
\label{lem:geom_lemma_1}
Suppose the sectional curvatures of $F$ are bounded by $-a^2\leq \kappa_F\leq b^2$ for $a,b\geq 0$. For any point $z\in B(x_0,\frac{\pi}{2b})$, we have
\[
D_{x=x_0}d(x,z)\tensor \nabla_{x=x_0}d(x,z)+d(x_0,z)D_{x=x_0}\nabla_x d(x,z)=I+T,
\]
where $I$ is the identity and $T$ is a $(1,1)$-tensor with $\norm{T}\leq \max\set{\frac{a^2}{3}, \frac{b^2}{2}} d(x_0,z)^2$.
\end{lemma}

\begin{proof}
The $(1,1)$-form $D_{x=x_0}d(x,z)\tensor \nabla_{x=x_0}d(x,z)$ has norm at most $1$ since one term  is the gradient of the distance function, and the other is its dual.  It is equal to $1$ precisely in the radial direction at $x_0$ pointing away from $z$, and vanishes on the orthogonal subspace.  

Since the points  $z$ belong to a convex ball centered at $x_0$. Then  the second fundamental form $D_{x=x_0}\nabla_x d(x,z)$ of the sphere $S(z,d(x_0,z))$ at the point $x_0$ is positive semidefinite, and vanishes only in the radial direction $\nabla_{x=x_0}d(x,z)$.

Indeed, the $(1,1)$-tensor $d(x_0,z)D_{x=x_0}\nabla_{x} d(x,z)$
is zero in the radial direction and in the tangent directions it is $d(x_0,z)\cdot II(x_0)$ where $II(x_0):T_{x_0}S(z,d(x_0,z))\to T_{x_0}S(z,d(x_0,z))$ is the second fundamental form of the sphere $S(z,d(x_0,z))$ at the point $x_0$.

If the sectional curvatures are bounded above by $b^2$ for $b\geq 0$, then by comparison geometry (\cite[Theorem 27]{Petersen16}), we may estimate the second fundamental form $II(y(t))$ of $S(z,t)$ at the point $y(t)$ by,
\[
a \coth(a t) I\geq II(y(t))\geq b \cot(b t) I.
\]
where the inequality is as positive definite forms and $II(y(t))$ is positive definite for $t<\frac{\pi}{2 b}$. We have the expansions 
\[
b t \cot(b t) =1- b^2 t^2/3 + O((b t)^4)\quad \text{and}\quad a t\coth(a t)=1+a^2 t^2/3 +O((at)^4).
\]
Moreover, for all $t>0$, $a t\coth(a t)-1<\frac{a^2}{3} t^2$ and for all $0<t<\frac{\pi}{2 b}$, we have $b t \cot(b t)-1 > -\frac{b^2}{2} t^2$.
Choosing $t=d(x_0,z)$ we thus have the expansion,
\[
d(x_0,z)\cdot II(x_0)= I  + T'
\]
where $\norm{T'} \leq \max\set{\frac{b^2}{2}, \frac{a^2}{3}} d(x_0,z)^2$. Combining these two terms establishes the lemma.
\end{proof}

 If $x, z \in F$ are connected by a unique minimal geodesic,  we denote  $\lVert^{x}_{z}: T_{z}F \to T_{x}F$ the parallel translation along this geodesic. This applies when $z\in B(x,\frac{\pi}{2b})$ where $b$ is an upper bound for sectional curvature of $F$.
 
 \begin{lemma}
\label{lem:geom_lemma_2}
Suppose the sectional curvatures of $F$ are bounded by $-a^2\leq \kappa_F\leq b^2$ for $a,b\geq 0$. For any point $x_0\in B(z_0,\frac{\pi}{2b})$,
\begin{align*}
%\label{eqn:eqn_in_geom_lemma_2}
D_{z=z_0}d(x_0,z)\tensor \lVert^{z_0}_{x_0}\nabla_{x=x_0}d(x,z_0)
+d(x_0,z_0)\lVert^{z_0}_{x_0} D_{z=z_0}\nabla_{x=x_0} d(x,z)=-(I+S),
\end{align*}
where $S$ is a $(1,1)$ tensor with $||S|| \leq 2\max\set{a^2,b^2} d(x_0,z_0)^2$.
%for some constant $C$ depending only on $a$ and $b$.
\end{lemma}
\begin{proof}

We observe that if $d(x,z) < \pi/2b$, then
\begin{align*}
%\label{eqn:opposite_normals}
    \nabla_{x}d(x,z)= -\lVert^{x}_{z}\nabla_{z}d(x,z).
\end{align*}
This implies that
\begin{align*}
%\label{eqn:first_term}
 D_{z=z_0}d(x_0,z)\tensor  \lVert^{z_0}_{x_0}  \nabla_{x=x_0}d(x,z_0) =-
     D_{z=z_0}d(x_0,z)\tensor \nabla_{z=z_0}d(x_0,z).
\end{align*}

Since $x_0 \in B(z_0,\pi/2b)$, \cref{lem:geom_lemma_1} and the previous equation implies that: 
\begin{align*}
     D_{z=z_0}d(x_0,z)\tensor \lVert^{z_0}_{x_0} \nabla_{x=x_0}d(x,z_0)=-
     \left( I + T - d(x_0,z_0) D_{z=z_0} \nabla_{z} d(x_0,z)\right)
\end{align*}
where $\norm{T} \leq \max\set{\frac{a^2}{3},\frac{b^2}{2}} d(z_0,x_0)^2$. Then
\begin{align*}
    D_{z=z_0}d(x_0,z)\tensor \lVert^{z_0}_{x_0} \nabla_{x=x_0}d(x,z_0) + d(x_0,z_0)\lVert^{z_0}_{x_0} D_{z=z_0}\nabla_{x=x_0} d(x,z)=-(I+T)+T_1
\end{align*}
where $$T_1:=d(x_0,z_0)(   D_{z=z_0}\nabla_z d(x_0,z) +  \lVert^{z_0}_{x_0} D_{z=z_0}\nabla_{x=x_0} d(x,z)).$$ In order to finish the proof of this lemma, it suffices to show that:

\begin{claim*}
$\norm{T_1} \leq \max\set{a^2,b^2} d(x_0,z_0)^2$. 
\end{claim*}

We will now prove this claim. Note that it suffices to show that 
\begin{align*}
   \norm{ D_{z=z_0}\nabla_z d(x_0,z) +  \lVert^{z_0}_{x_0} D_{z=z_0}\nabla_{x=x_0} d(x,z)} \leq \max\set{a^2,b^2} d(x_0,z_0).
\end{align*}

Let $t_0:=d(x_0,z_0)$. For $0 \leq t \leq t_0$, let $S(x_0,t):=\{z : d(x_0,z)=t\}$. Then, if $z\in S(x_0,t)$, then $\nabla_z d(x_0,z)$ is the outward pointing unit normal to $S(x_0,t)$ at $z$ while $\nabla_{x=x_0}d(x,z)$ is the inward pointing unit normal to $S(x_0,t)$ at $z$, parallel translated to $x_0$ along the geodesic joining $x_0$ and $z$.

Let $v \in T_{z_0}F$ and let $\gamma_v:(-\varepsilon,\varepsilon) \to F$ be a curve on $S(x_0,t_0)$ such that $\gamma_v(0)=z_0$ and $\dot{\gamma_v}(0)=v$. If $v \in (T_{z_0}S(x_0,t_0))^{\perp}$, then $T_1(v)=\lVert^{x_0}_{z_0}D_{\dot{\gamma}(t_0)}(-\dot{\gamma})(t_0)+D_{\dot{\gamma}(0)}\dot{\gamma}(0)=0$. So we now assume that $v \in T_{z_0}S(x_0,t_0)$ and consider the  geodesic variation given by $p_v(s,t):=\exp_{x_0}(t \cdot \alpha_v(s))$
where $\alpha_v(s):=\frac{1}{t_0}\exp_{x_0}^{-1}(\gamma_v(s))$. 
Then 
\begin{align*}
    &D_{z=z_0}\nabla_{z}d(x_0,z)(v)=\frac{\partial}{\partial s}\at_{s=0} \frac{\partial}{\partial t}\at_{t=t_0} p_v(s,t) \\
    &\text{ and } D_{z=z_0}\nabla_{x=x_0}d(x,z)(v)= - \frac{\partial}{\partial s}\at_{s=0}\frac{\partial}{\partial t}\at_{t=0} p_v(s,t).
\end{align*}
Since the map $p_v$ is smooth, we can interchange the order of the derivatives above to get that 
\begin{align*}
     D_{z=z_0}\nabla_{z}d(x_0,z)(v)=J_v'(t_0) \quad \text{ and } \quad D_{z=z_0}\nabla_{x=x_0}d(x,z)(v)= - J_v'(0)
\end{align*}
where $J_v(t):=\frac{\partial}{\partial s}\big|_{s=0} p_v(s,t)$ is a Jacobi field along the geodesic $c(t):=p_v(0,t)$ associated to the geodesic variation $p_v(s,.)$ \cite[section 2, Proposition 2.4]{doCarmo92}. Also note that $J_v(t_0)=v$. Thus
\begin{align*}
   \norm{  D_{z=z_0}\nabla_z d(x_0,z) +  \lVert^{z_0}_{x_0} D_{z=z_0}\nabla_{x=x_0} d(x,z)} = \sup_{||v||=1}   \norm{J_v'(t_0)-\lVert^{c(t_0)}_{c(0)}J_v'(0)}.
\end{align*}
However
\begin{align*}
    & \norm{J_v'(t_0)-\lVert^{c(t_0)}_{c(0)}J_v'(0)}= \norm{\int_{0}^{t_0}  \lVert^{c(t_0)}_{c(t)}J_v''(t) dt} \leq   \int_{0}^{t_0} \norm{ J_v''(t)} dt \\
    &= \int_0^{t_0} \norm{R_{c(t)}\left(J_v(t),\dot{c}(t)\right)\dot{c}(t) }   dt \leq \int_0^{t_0}\norm{R_{c(t)}}||J_v(t)||dt\\
    &\leq t_0 M = Md(x_0,z_0),
\end{align*}
where $M:=\sup_{x \in F} ||R_x|| \cdot \sup_{t \in [0,t_0]}||J_v(t)||$. Note that $\sup_{x\in F} ||R_x||\leq \max\set{a^2,b^2}$.  %and $\sup_{t \in [0,t_0]}||J_v(t)||\leq \norm{J_v(t_0)=v}=1$. 
Also, by generalized Rauch Comparison (Corollary 2.19 of \cite{Sakai84}) we have for $t\leq \frac{\pi}{2 b}$, $\norm{J_v(t)}\leq \norm{J^b(t)}\leq 1$ where $J^b(t)$ is any Jacobi field in constant curvature $b^2\geq 0$ with conditions $J^b(0)=0$ and $\norm{J^b(t_0)}=1=\norm{J_v(t_0)}$. Thus $M\leq \max\set{a^2,b^2}$.
%This also shows that $M$ depends only on $a$ and $b$. 

This establishes the Claim and hence the lemma.
\end{proof}

\subsection{Derivatives of the barycenter map}
The goal of this section is to establish \cref{prop:dbary_limit}. Recall that for this entire Appendix, we have fixed a compact, connected $C^1$ (hence smooth) manifold $F$. %\todo{Double check if $F$ is smooth or $C^1$}\\

In the first part of this section (Lemma \ref{lem:Q_id} to Lemma \ref{lem:dbary_measure}), we will assume the following standing assumptions.
\vspace{.5em}

\label{link:setup-a}
\textbf{\bf Assumptions A:} 
\emph{Let $a,b \geq 0$ and $n \in \Nb$ be fixed constants. Suppose that:
\begin{enumerate}
    \item $F$ is equipped with a $C^2$ Riemannian metric $g$ whose sectional curvatures are bounded between $-a^2\leq \kappa_F\leq b^2$.
    \item $\mu=\sum_{i=1}^n w_i \delta_{z_i}$ is a probability measure on $F$ where $0 \leq w_i \leq 1$ are the weights and $\set{z_1, \dots, z_n}\subset F$ is the support of $\mu$. 
    \item We have $\op{diam}\set{z_1, \dots, z_n}< \min\set{\frac{\pi}{4b},\frac12\op{InjRad}_g(F)}$, where $\op{InjRad}_g(F)$ denotes the injectivity radius of $(F,g)$.
\end{enumerate}
We remark that any open ball of radius $\min\set{\frac{\pi}{2b},\frac12\op{InjRad}_g(F)}$ is strongly geodesically convex (cf. \cite[Theorem IX.6.1]{Chavel06}). 
}
\vspace{.5em}

\noindent \textbf{Notation for Assumptions A:} We standardize the following notation for Assumptions A. We will denote by $\nabla$ the Riemannian connection on $F$ and by $\norm{\cdot}$ the norm associated to the metric $g$. We will denote by $D$ the covariant derivative operator, namely if $z \in F$, $X$ is a smooth vector field, and  $Y$ is a vector field on $F$, then $D_zY(X)=(\nabla_X Y)(z)$. Moreover, if $f:F \to \Rb$ is a smooth function, then $\nabla f$ is the vector field dual to the 1-form $Df$.

\medskip
Under Assumptions A, the function $G$ defined in \cref{eq:G}) becomes 
\begin{equation}
\label{eqn:G_in_appendix}
    G(x,w,z)=\sum_{i=1}^nw_id(x,z_i)\nabla_xd(x,z_i).
\end{equation}
% If $z_1, \dots, z_n$ are close enough and lie in a geodesically convex neighbourhood, then 
We denote by $\op{bar}(\mu)$ denotes the barycenter of $\mu$ (which is defined as the unique point in $F$ such that $G(\op{bar}(\mu), w,z)=0$; see \cref{sec:barycenter} for details). This exists and is well defined (see Section \ref{sec:barycenter} or \cite[Proposition IX.7.1]{Chavel06}).

In what follows, we define a $(1,1)$-tensor $Q$ by
\begin{equation}\label{eq:Q}
Q:=\sum_{j=1}^n w_j \left(D_{x=\op{bar}(\mu)}d(x,z_j)\tensor \nabla_{x=\op{bar}(\mu)}d(x,z_j)+d(\op{bar}(\mu),z_j)D_{x=\op{bar}(\mu)}\nabla_x d(x,z_j)\right).
\end{equation}

\begin{lemma}\label{lem:Q_id}
Under \hyperref[link:setup-a]{\bf Assumptions A},   $\max_{1\leq i \leq n} d(\op{bar}(\mu),z_i) \leq 2\op{diam}\set{z_1, \dots, z_n}$, and 
\begin{align*}
  \norm{Q-I} & \leq \max\set{\frac{a^2}{3}, \frac{b^2}{2}}\max_{1\leq i \leq n} d(\op{bar}(\mu),z_i)^2 \\
  &\leq 2\max\set{a^2,b^2}\op{diam}\set{z_1, \dots, z_n}^2.
\end{align*}

\end{lemma}
\begin{proof}
By Remark \ref{rem:barybound}, $\max_i d(\op{bar}(\mu),z_i) \leq 2\op{diam}\set{z_1, \dots, z_n}$. The second statement of the lemma follows from \cref{lem:geom_lemma_1} by letting $x_0=\op{bar}(\mu)$, $z=z_j$ and  since $\sum_j w_j=1$ (which implies that each summand is close to identity).
\end{proof}

Noting that $Q$ is a symmetric tensor, we immediately obtain the following from Lemma \ref{lem:Q_id}.
\begin{corollary}
\label{cor:Q_pos_def}
Under \hyperref[link:setup-a]{\bf Assumptions A}, %let $C$ be the constant as in  Lemma \ref{lem:Q_id} above.
% \todo{need to clarify what happens if $a=b=0$}
if  $\diam \set{z_1,\dots,z_n}<\frac{1}{3\max\set{a,b}}$, then $\norm{Q-I} \leq \frac14$ and $Q$ is a strictly positive definite tensor and thus  invertible. Moreover, $\norm{Q^{-1}}\leq \frac43 < 2$. %(Note that the constant $C$ depends only on $a$ and $b$). 
\end{corollary}

\begin{remark}
We allow $\max\set{a,b}=0$ in the above corollary. In this case $Q=I$.
\end{remark}

\begin{lemma}[Derivative of the barycenter with respect to the measure]
\label{lem:dbary_measure}
%Suppose the sectional curvatures of $F$ are bounded between $-a^2\leq \kappa_F\leq b^2$ for $a,b\geq 0$. If $\mu=\sum_{i=1}^n w_i \delta_{z_i}$ where 
Suppose \hyperref[link:setup-a]{\bf Assumptions A}. If $\diam \set{z_1,\dots, z_n}< \frac{1}{3\max\set{a,b}}$, then the partial derivative of the map $(w_1,\dots,w_n, z_1,\dots,z_n) \mapsto \op{bar}(\mu)$ in the weight $w_i$ is
\begin{align*}
D_{w_i}\op{bar}(\mu)=&-Q^{-1}\left(d(\op{bar}(\mu),z_i)\nabla_{x=\op{bar}(\mu)}d(x,z_i)\right)
\end{align*}
and the partial derivative in $z_i$ is
\begin{align*}
D_{z_i}\op{bar}(\mu)=&-Q^{-1} \left(w_i D_{z_i}d(\op{bar}(\mu),z_i)\tensor \nabla_{x=\op{bar}(\mu)}d(x,z_i)+w_i d(\op{bar}(\mu),z_i) D_{z_i}\nabla_x d(x,z_i)\right).
\end{align*}
\end{lemma}

\begin{proof}
We observe that $0= G(\op{bar}(\mu),w,z)$ and differentiating this we obtain,
\begin{align}
\label{eqn:solve_D_w_G}
0=D_{w_i} G(\op{bar}(\mu),w,z)=D_{x=\op{bar}(\mu)} G(x,w,z)\of D_{w_i}\op{bar}(\mu)+D_{w_i}G(x,w,z)\rest{x=\op{bar}(\mu)} .
\end{align}
We have already seen from \eqref{eq:DG} that
\[D_x G(x,w,z)=\sum_{j=1}^n w_j \left(D_xd(x,z_j)\tensor \nabla_xd(x,z_j)+d(x,z_j)D_{x}\nabla_x d(x,z_j)\right).\]
Then $D_{x=\op{bar}(\mu)}G(x,w,z)=Q$ (cf. \cref{eq:Q}). Since $\diam \set{z_1,\dots,z_n}<\frac{1}{3\max\set{a,b}}$, Corollary \ref{cor:Q_pos_def} implies that $Q$ is strictly positive definite and hence invertible. Moreover, from \cref{eqn:G_in_appendix}, 
\[
 D_{w_i}G(x,w,z)\rest{x=\op{bar}(\mu)}=d(\op{bar}(\mu),z_i)\nabla_{x=\op{bar}(\mu)}d(x,z_i).
\]
Then we can solve \cref{eqn:solve_D_w_G} to obtain: $$D_{w_i}\op{bar}(\mu)=-Q^{-1} \left( d(\op{bar}(\mu),z_i)\nabla_{x=\op{bar}(\mu)}d(x,z_i) \right).$$

The case of $D_{z_i}\op{bar}(\mu)$ is similar. By differentiating $0= G(\op{bar}(\mu),w,z)$ we obtain
\[
0=D_{z_i} G(\op{bar}(\mu),w,z)=D_x G(x,w,z)\of D_{z_i}\op{bar}(\mu)+D_{z_i}G(x,w,z)\rest{x=\op{bar}(\mu)} .
\]
For the second term we observe that 
\begin{align}
\label{eqn:solve_d_zi}
    D_{z_i}G(x,w,z)\rest{x=\op{bar}(\mu)}=D_{z_i}d(\op{bar}(\mu),z_i)\tensor \nabla_{x=\op{bar}(\mu)}d(x,z_i)+w_i d(\op{bar}(\mu),z_i) D_{z_i}\nabla_{x=\op{bar}(\mu)}d(x,z_i). 
\end{align}
As $\diam\set{z_1, \dots, z_n}<\frac{\pi}{4b}$, 
\cref{lem:geom_lemma_2}  implies: 
%if $\diam\set{z_1, \dots, z_n}<\frac{\pi}{4b}$, then
$D_{z_i}G(x,w,z)|_{x=\op{bar}(\mu)}=-\sum_{i=1}^n w_i \lVert^{\op{bar}(\mu)}_{z_i}(I+S)$ (where $S$ is as in \cref{lem:geom_lemma_2}).  Thus, by similar reasoning as in the proof of \cref{cor:Q_pos_def}, we see that: there exists $r'$ sufficiently small (depending only on the constants $a$ and $b$) such that $\diam \set{z_1,\dots, z_n}< r'$ implies  $D_{z_i}G(x,w,z)$ is invertible. Then solving for $D_{z_i}\op{bar}(\mu)$ from \cref{eqn:solve_d_zi} completes the result. 
\end{proof}

In the next two lemmas, we will consider families of metrics $g_s$ on $F$ jointly satisfying \hyperref[link:setup-a]{\bf Assumptions A}.  Recall the definition of $C^{1,2}$ family of metrics from \cref{sec:C_12_functions}.

We precisely formulate the:
\vspace{.5em}

\textbf{\bf Assumptions B: }
\label{link:setup-b}
\emph{Let $\eps_0>0$, $a,b \geq 0$ and $m, n \in \Nb$ be fixed constants. Suppose that:
\begin{enumerate}
    \item $B(0,\eps_0) \ni s \mapsto g_s$ is a family of $C^{1,2}$ Riemannian metrics on $F$ where each $g_s$ is a $C^2$ Riemannian metric with sectional curvature $\kappa_s \in [-a^2,b^2]$.  
    \item for each $s$, there is a probability measure $\mu(s):=\sum_{i=1}^n w_i(s) \delta_{z_i(s)}$ where the weights $w_i(s)$ and the points $z_i(s)$ depend on the parameter $s \in B(0,\eps_0)$ in a $C^1$ way. 
    \item We have $\op{diam}_s\set{z_1(s), \dots, z_n(s)}< \min\set{\frac{\pi}{4b},\frac12\op{InjRad}_{g_s}(F),\frac{1}{3\max\set{a,b}}}$, where $\op{InjRad}_{g_s}(F)$ denotes the injectivity radius of $(F,g_s)$. (We interpret $\frac{\pi}{4b}=\infty$ when $b=0$ and $\frac{1}{3\max\set{a,b}}=\infty$ when $a=b=0$.)
\end{enumerate}
}
\noindent \textbf{Notations under Assumptions B:} We also standardize the following notation for use henceforth. For each Riemannian metric $g_s$, we denote the corresponding Riemannian connection by $\nabla^s$, distance function by $d_s$, diameter measured in $d_s$ by $\text{diam}_s$, barycenter map by $\op{bar}_s$, and the covariant derivative operator by $D^s_x$ (i.e. if $X, Y$ are smooth vector fields on $F$ and $x \in F$, then $D^s_xY(X)=\nabla^s_XY(x)$). Moreover, let $\nabla^s_xf$ be the gradient of a smooth function $f:F  \to \Rb$ for the metric $g_s$ (i.e. $\nabla^s_xf(v)=D_vf(x)$ for any $v \in T_xF$). Finally, $D_s$ is the usual covariant derivative operator on $B(0,\eps_0) \subset\Rb^m$. 

Finally, for each $s$, let $Q_s$ denote the $(1,1)$ tensor $Q$ (cf. \cref{eq:Q}) that corresponds to the distance function $d_s$, i.e. 
\begin{align}
\label{eqn:Qs}
Q_{s}:=\sum_{i=1}^n w_i(s) \left(D_{x=x_{s}}^{s}d_{s}(x,z_i(s))\tensor \nabla^{s}_{x=x_s}d_{s}(x,z_i(s))+d_{s}(x,z_i(s))D_{x=x_s}^{s}\nabla^{s}_x d_{s}(x,z_i(s))\right).
\end{align}
where $x_s:=\op{bar}_{s}(\mu(s))$.

In the next three lemmas, we will compute the derivatives of the map $s \mapsto \op{bar}_s(\mu(s))$ under Assumptions B as above.

\begin{lemma}
\label{lem:reg_of_dist}
Suppose \hyperref[link:setup-b]{\bf Assumptions B part (1)} holds. Then the map $E_x:T_xF\times B(s_0,\eps)\to F$ defined by $E_x(v,s)=\exp^{g_s}_{x}(v)$ satisfies:
\begin{enumerate}[label=(\alph*)]
    \item it is $C^1$ in $x$ and $v$  and the first order derivatives (in $x$ and $v$) are $C^1$ in $s$,
    \item $D_sE_x(v,s)$ is differentiable with respect to $v$ and the mixed partial derivatives with respect to $s$ and $v$ are equal (i.e. $D_vD_sE_x(v,s)=D_sD_vE_x(v,s)$). 
\end{enumerate}
 Moreover, for all $z$ and $x$ with $d_s(x,z)$ sufficiently small, $d_s(x,z)^2$ is $C^2$ in $x$ and $z$, $C^1$ in $s$ and $\nabla_{x}^s (d_s(x,z))^2$ is $C^1$ in $s$. 
\end{lemma}
\begin{proof}
By assumption, the Christoffel symbols for $g_s$ are $C^1$ in $x$ and $s$ and have derivatives up to order one which are also $C^1$ in $s$. Regularity of solutions of ODEs (the geodesic equations) implies that $E_x(v,s)$ is $C^1$ in $x$ and $v$ and its derivatives up to order $1$ are $C^1$ functions of $s$. This proves part (a). Part (b) follows from \cite[Theorem 9.41, page 236]{Rudin-book}, because $E_x(v,s)$ is $C^1$ in $(v,s)$ and $D_vE_x(v,s)$ is $C^1$ in $s$. 

By the Implicit Function Theorem, there is a map $\alpha_x:B(s_0,\eps)\to T_xF$ with the same regularity as $E_x$ (i.e. $\alpha_x(s)$ is $C^1$ in both $x$ and $s$) such that $E_x(\alpha_x(s),s)=z$. However, observe that 
\[
\alpha_x(s)=- d_s(x,z)\nabla^s_{x}d_s(x,z)=-\frac12\nabla^s_{x} (d_s(x,z))^2.
\]
Hence the stated regularity of $d_s(x,z)^2$ follows.
%\todo[inline]{Then $\alpha_x(s)$ is $C^1$ in both $x$ and $s$. But $\alpha_x$ equals the gradient of $d_s(x,z)^2$, so $d_s(x,z)^2$ is in fact $C^2$ in x. }
\end{proof}

\begin{corollary}\label{cor:reg_of_dist}
Suppose \hyperref[link:setup-b]{\bf Assumptions B part (1)} holds and $s \mapsto z(s)$ is $C^1$ as $s$ varies in $B(0,\eps_0)$. Then $\nabla_x^sd_s(x,z(s))^2$ varies $C^1$ in both $s$ and $x$. 
%{\color{red} In particular, $Q_s$ is continuous in $s$ and $x$.}
\end{corollary}
\begin{proof}
We will prove this on $B(s_0,\eps)$ and let $g_0:=g_{s_0}$. Let $E_x(v,s)$ be as above and define $F_x(v,s):=(\exp^{g_0}_x)^{-1} \circ E_x(v,s) - (\exp^{g_0}_x)^{-1} (z(s))$. Then $F_x(v,s)$ is $C^1$ in both $x$ and $s$ and further, 
\[
D_v F_x(v,s)=D_{w=E_x(v,s)}(\exp^{g_0}_x)^{-1}(w) \circ D_vE_x(v,s).
\]
This implies that $D_vF_x(v,s)$ is invertible whenever $\norm{v}_{g_0}$ sufficiently small. Then, by Implicit Function theorem, there exists a function $\beta_x: B(s_0,\eps) \to F$ which is $C^1$ in both $x$ and $s$ such that $F_x(\beta_x(s),s)=0$. Note that $F_x(v,s)=0$ if and only if $E_x(v,s)=z(s)$. Thus $E_x(\beta_x(s),s)=z(s).$
Then 
\[\beta_x(s)=-d_s(x,z(s)) \nabla^s_{x} d_s(x,z(s)) = -\frac12\nabla^s_{x} (d_s(x,z(s)))^2,
\]
which finishes the proof as $\beta_x$ is $C^1$ in $s$ and $x$. 
\end{proof}

\begin{lemma}[Derivative of the barycenter with respect to the metric]
\label{lem:dbary_metric}
Suppose  \hyperref[link:setup-b]{\bf Assumptions B} hold. Assume that $w_i(s) = w_i$ and $z_i(s)= z_i$ are constant functions for each $i \in \set{1,\dots,n}$ and $\mu(s)\equiv \mu = \sum_{i=1}^n w_i \delta_{z_i}$. %Further assume that $\diam \set{z_1, \dots, z_n}< \min\set{\frac{\pi}{4b}, \frac{1}{3\max\set{a,b}}}$.
Then the derivative of the map $s \mapsto \op{bar}_s(\mu)$ with respect to the parameter $s$ at $s=s_0$ is:
\begin{align*}
D_{s=s_0}\op{bar}_{s}(\mu)=-Q_{s_0}^{-1} \left(\sum_{i=1}^n w_i \left(D_{s=s_0} d_{s}(x_0,z_i)\tensor\nabla^{s_0}_{x=x_0}d_{s_0}(x,z_i)+ d_{s_0}(x_0,z_i)D_{s=s_0}\nabla^s_{x=x_0}d_{s}(x,z_i)\right) \right)
\end{align*}
where $x_0:=\op{bar}_{s_0}(\mu)$. 
\end{lemma}

\begin{proof}
For every $s$, we define a function $G_s(x,w,z) = G(x,w,z,s)$ for the distance function $d_s$ as in equation \ref{eqn:G_in_appendix}. 
As in Lemma \ref{lem:dbary_measure} we begin by differentiating the equation $0= G(\op{bar}_{s}(\mu),w,z,s)$ 
to obtain
\[
0=D_{s=s_0} G(\op{bar}_{s}(\mu),w,z,s)= D_{x=\op{bar}_{s_0}(\mu)}^{s_0} G(x,w,z,s_0)\of D_{s=s_0}\op{bar}_{s}(\mu)+D_{s=s_0}G(x,w,z,s)\rest{x=\op{bar}_{s_0}(\mu)}. 
\]

From \eqref{eq:DG} we have 
\[
D_x^{s_0} G(x,w,z,s_0)=\sum_{i=1}^n w_i \left(D_x^{s_0}d_{s_0}(x,z_i)\tensor \nabla^{s_0}_xd_{s_0}(x,z_i)+d_{s_0}(x,z_i)D_{x}^{s_0}\nabla^{s_0}_x d_{s_0}(x,z_i)\right).
\]

Note that $D_{x=\op{bar}_{s_0}(\mu)}^{s_0} G(x,w,z,s_0)=Q_{s_0}$ and since $\diam\set{z_1,\dots,z_n}<\frac{1}{3\max\set{a,b}}$, \cref{cor:Q_pos_def} implies that $Q_{s_0}$ is a strictly positive definite (hence invertible) tensor. 

Moreover, we compute,
\begin{align*}
D_{s=s_0}G(x,w,z,s)\rest{x=\op{bar}_{s_0}(\mu)}&=\sum_{i=1}^n w_i D_{s=s_0} d_{s}(\op{bar}_{s_0}(\mu),z_i)\tensor\nabla^{s_0}_{x=\op{bar}_{s_0}(\mu)}d_{s}(x,z_i(s))\\
&\qquad\qquad\qquad +w_i d_{s_0}(\op{bar}_{s_0}(\mu),z_i)D_{s=s_0}\nabla^s_{x=\op{bar}_{s_0}(\mu)}d_{s}(x,z_i).
\end{align*}
We may now solve for $D_{s=s_0}\op{bar}_s(\mu)$ to obtain the result stated.
\end{proof}

We now prove \cref{prop:bary_C1}.

\begin{proposition}
\label{prop:dbary_limit}
Suppose  \hyperref[link:setup-b]{\bf Assumptions B} hold and let $r(s):=\op{diam}_{(F,d_s)}\set{z_1(s),\dots,z_n(s) }$. Then $s\mapsto \op{bar}_s(\mu(s))$ is $C^1$ and for any $s_0 \in B(0,\eps_0)$, 
\begin{align*}
\norm{D_{s=s_0} \op{bar}_s(\mu(s)) -  Q_{s_0}^{-1}  \left(\sum_{i=1}^n w_i(s_0)\lVert^{\op{bar}_{s_0}(\mu(s_0))}_{z_i(s_0)} D_{s} z_i(s) \right)} \leq 32 \max \set{a^2,b^2} L(s_0) r(s_0)
\end{align*}
where 
\begin{align*}
%\label{eqn:L}
L(s):= \max_{1\leq i \leq n} \set{\norm{D_s z_i(s)}, \sum_{i=1}^n\norm{D_sw_i(s)}}+ \norm{g_s}_{C^{1,2}},
%  L:= \max_{1\leq i \leq n}\left\{\sup_{s} \norm{D_sw_i(s)}, \sup_{s} \norm{D_s z_i(s)}, \sup_{s} \sup_{v \neq 0} \abs{\grad_{s=s_0} g_{s}(v,v)}\right\}.
\end{align*}
where $\norm{g_s}_{C^{1,2}}$ represents the bounds of derivatives up to order one in $s$ at $s\in B(0,\eps_0)$ of derivatives up to order two of components of $g_s$ with respect to $g_0$-exponential coordinates.
\end{proposition}

\begin{remark}
In the above lemma (and its proof below), whenever we have a vector (or a tensor) decorated by $s$, then the corresponding norm $\norm{\cdot}$ should be understood as being taken with respect to the Riemannian metric $g_s$. For example, $\norm{T_s}$, $\norm{D_sw_i(s)}, \norm{D_s z_i(s)}$, etc. are all norms taken with respect to the metric $g_s$. We avoid the more precise but cumbersome notation $\norm{\cdot}_s$.  
\end{remark}
%\color{black}

\begin{proof}

We first note that the function $L(s)$ is well-defined. Indeed both $w_i(s)$ and $z_i(s)$ vary $C^1$ in $s$ and $(s,v) \to g_s(v,v)$ is at least $C^2$.
Moreover  \cref{lem:reg_of_dist} and \cref{cor:reg_of_dist} implies that whenever $d_s(x,z_i(s))$ is sufficiently small, $d_s(x,z_i(s))^2$ is $C^2$ in $x$ and $C^1$ in $s$ and its gradient $\nabla^s_x d_s(x,z_i(s))^2$ is $C^1$ in $s$. By the Implicit Function Theorem, we obtain that $s \mapsto \op{bar}_s(\mu(s))$ is $C^1$ in $s$.

Fix any $s_0 \in B(0,\eps_0)$ and let $\mu:=\mu(s_0)$, $w_i:=w_i(s_0)$, $z_i:=z_i(s_0)$ and $x_0:=\op{bar}_{s_0}(\mu)$. Considering the map $(t,s)\mapsto \op{bar}_{t}(\mu(s))$, and  differentiating along the diagonal we obtain,
\begin{equation}
\label{eqn:ds_bary_mu_s}
   D_{s=s_0}\left( \op{bar}_s(\mu(s)) \right) =D_{s=s_0}\op{bar}_s(\mu)+D_{s=s_0}\op{bar}_{s_0}(\mu(s)). 
\end{equation}

We will evaluate the  two terms on the right hand side of this equation using Lemmas \ref{lem:dbary_metric} and \ref{lem:dbary_measure} respectively.

\setcounter{claim}{0}

\begin{claim} \label{claim1}
% $D_{s} \op{bar}_s(\mu)=O(r(s))$.
$\norm{D_{s=s_0} \op{bar}_{s}(\mu)} \leq 32 L(s_0) r(s_0).$
\end{claim}

{\em Proof of Claim 1:} To prove this, we use the equation for $D_{s} \op{bar}_s(\mu)$ from Lemma \ref{lem:dbary_metric}. 
Since $r(s_0)<\frac{1}{3\max\set{a,b}}$, \cref{cor:Q_pos_def} implies that $\norm{Q_{s_0}^{-1}} < 2$. Since $\norm{Q_{s_0}^{-1}}<2$ and $\sum {w_i}= 1$, it suffices to show that 

\begin{align}
    \label[ineq]{eqn:bound_ds_square}
\norm{D_{s=s_0} \nabla^s_{x=x_0} (d_s(x,z_i))^2} \leq 16 L(s_0) r(s_0), 
\end{align}
as 
\begin{align*}
    \norm{D_{s=s_0}\op{bar}_s(\mu)} &\leq 2 \max_{1 \leq i \leq n}\norm{\left(D_{s=s_0} d_{s}(x_0,z_i)\tensor\nabla^{s_0}_{x=x_0}d_{s_0}(x,z_i) + d_{s_0}(x_0,z_i)D_{s=s_0}\nabla^s_{x=x_0}d_{s}(x,z_i)\right)} \\
    &=2 \max_{1 \leq i \leq n} \norm{D_{s=s_0} \nabla^s_{x=x_0} (d_s(x,z_i))^2}. \\
\end{align*}

In order to prove \cref{eqn:bound_ds_square}, it suffices to show that 
\[\norm{D_{s=s_0}\alpha(s)} \leq 8 L(s)r(s)\]
where \[\alpha(s):=- d_s(x_0,z_i)\nabla^s_{x=x_0}d_s(x,z_i)=-\frac{1}{2}\nabla^s_{x=x_0} (d_s(x,z_i))^2.
\]
By \cref{lem:reg_of_dist}, $\alpha: B(s_0,\eps) \to F$ is a $C^1$ map such that $E(\alpha(s),s)=z_i$ for all $s \in B(s_0,\eps)$. Here $E(v,s)\equiv E_{x_0}(v,s)=\exp^{g_s}_{x_0}(v)$ is the exponential map in the $g_s$ metric, as defined in \cref{lem:reg_of_dist}.  Then
\[
D_s\alpha(s)=-(D_{v=\alpha(s)}E(v,s))^{-1} \of D_sE(v,s)|_{v=\alpha(s)},
\]
as $0=D_s E(\alpha(s),s)=D_{v=\alpha(s)}E(v,s)\of D_s\alpha(s)+D_sE(v,s)\rest{v=\alpha(s)}$.

For all $(v,s)$ in a sufficiently small neighbourhood of $(0,s_0) \in T_{x_0}F \times B(s_0, \eps)$ (whose diameter only depends on the curvature bounds $a$ and $b$ of $F$), $D_{v=0}E(v,s)=I:T_{x_0}F\to T_{x_0}F$ and $(D_vE(v,s))$ is uniformly bounded away from $0$.

Since $g_s$ satisfies \hyperref[link:setup-b]{\bf Assumptions B} part (1),  \cref{lem:reg_of_dist} implies that 
$D_{s=s_0}E(v,s)$ is differentiable at $v=0$, that is, 
\begin{align*}
    \lim_{\norm{v} \to 0}\frac{\norm{D_{s=s_0}E(v,s)-D_{s=s_0}E(0,s)-D_{v=0}D_{s=s_0}E(v,s)(v)}}{\norm{v}} =0.
\end{align*}
Since $D_{s=s_0}E(0,s)=0$, 
\begin{align*}
    \norm{D_{s=s_0}E(v,s)} \leq (\norm{D_{v=0}D_{s=s_0}E(v,s)}+h(v)) \norm{v}
\end{align*}
for some function with $\lim_{v\to 0}h(v)=0$. Since $\norm{D_{v=0}D_{s=s_0}E(v,s)} \leq \norm{g_{s_0}}_{1,2} \leq L(s_0)$, 
\[
\norm{D_{s=s_0}E(v,s)} \leq 2L(s_0) \norm{v}
\]
whenever $\norm{v}$ is sufficiently close to $0$. Since at $v=\alpha(s_0)$, we have $\norm{\alpha(s_0)}=d_{s_0}(x,z_i)< 2r(s_0)$, this implies that 
\[
\norm{D_{s=s_0}\alpha(s)} \leq 4 L(s_0) r(s_0) \norm{(D_{v=\alpha(s_0)}E(v,s_0))^{-1}}.
\]
By Theorem II.7.1 of \cite{Chavel06}, $D_{v=\alpha(s_0)}E(v,s_0)(w)=\frac{1}{t_0}J(t_0)$ where $t_0=\norm{\alpha(s_0)}$ and $J$ is the Jacobi field along $t \to E(t\frac{\alpha(s_0)}{\norm{\alpha(s_0)}},s_0)$ such that $J(0)=0$ and $J'(0)=w$. By Rauch comparison theorem, $\norm{J(t)} \geq \frac{\norm{w}}{a}\sin(at)$. Since $\frac{\sin(t)}{t}$ is monotonically decreasing on $[0,\pi]$ and $t_0<\frac{1}{3a}<\frac{\pi}{2a}$, we have $\norm{D_vE(v,s_0)(w)} \geq \frac{2}{\pi}\norm{w}$.   Then, $\norm{(D_{v=\alpha(s_0)}E(v,s_0))^{-1}} \leq \pi/2 < 2. $

Thus 
\[
\norm{D_{s=s_0}\alpha(s)} \leq 8 L(s_0) r(s_0)
\]
which finishes the proof of Claim \ref{claim1}.   \qed

We will now finish the proof of \cref{prop:dbary_limit} by establishing:
\vspace{.5em}

\begin{claim} \label{claim2}
%\begin{align*}
$ \norm{D_{s=s_0}\op{bar}_{s_0}(\mu(s))- Q_{s_0}^{-1} \of \sum_{i=1}^n w_i\lVert_{z_i}^{\op{bar}_{s_0}(\mu)}D_{s=s_0}z_i(s)} \leq 16 \max\set{1, a^2, b^2}L(s_0) r(s_0).$
%\end{align*}
\end{claim}

{\em Proof of Claim 2:}
We first expand the term $D_{s=s_0}\op{bar}_{s_0}(\mu(s))$ as,
\begin{align}\label{eqn:chain_rule_Ds_w_and_z}
D_{s=s_0}\op{bar}_{s_0}(\mu(s))=\sum_{i=1}^{n}  D_{w_i}\op{bar}_{s_0}(\mu)\of D_{s=s_0} w_i(s) + \sum_{i=1}^n D_{z_i}\op{bar}_{s_0}(\mu)\of D_{s=s_0} z_i(s).
\end{align}

We will now prove this claim by first showing that
\begin{align}
\label[ineq]{eqn:claim_for_wi}
   \norm{\sum_{i=1}^n D_{w_i}\op{bar}_{s_0}(\mu)\of  D_{s=s_0} w_i(s)}  \leq 4 L(s_0) r(s_0)
\end{align}
and then showing that
\begin{align}
\label[ineq]{eqn:claim_for_zi}
    \norm{ \sum_{i=1}^n \left( D_{z_i}\op{bar}_{s_0}(\mu)-Q_{s_0}^{-1}  \of w_i \lVert^{\op{bar}_{s_0}(\mu)}_{z_i} \right) \of D_{s=s_0} z_i(s)  } \leq 16 \max \set{a^2,b^2} L(s_0) r(s_0).
\end{align}

We first prove the claimed \cref{eqn:claim_for_wi}. The first part of Lemma \ref{lem:dbary_measure} implies that 
\[
\norm{D_{w_i}\op{bar}_{s_0}(\mu)} \leq \norm{Q_{s_0}^{-1}} \abs{d_{s_0}(\op{bar}_{s_0}(\mu),z_i)}\norm{\nabla_{x=\op{bar}_{s_0}(\mu)}d_{s_0}(x,z_i)}.
\]
Note that $\nabla_{x=\op{bar}_{s_0}(\mu)}d_{s_0}(x,z_i)$ has norm at most 1. Since $r(s_0)<\frac{1}{3\max\set{a,b}}$, \cref{cor:Q_pos_def} implies that $\norm{Q_{s_0}^{-1}}\leq 2$. Then
\[
\norm{D_{w_i}\op{bar}_{s_0}(\mu)} \leq 2 d_{s_0}(\op{bar}_{s_0}(\mu),z_i) \leq 4 r(s_0).
\]
The claimed \cref{eqn:claim_for_wi} then follows as 
\begin{align*}
   \norm{\sum_{i=1}^n D_{w_i}\op{bar}_{s_0}(\mu)\of  D_{s=s_0} w_i(s)} \leq 4 r(s_0) \sum_{i=1}^n \norm{D_{s=s_0}w_i(s)} \leq 4 L(s_0) r(s_0). 
\end{align*}

In order to prove claimed \cref{eqn:claim_for_zi}, first fix $i \in \set{1,\dots,n}$. Recalling that $x_0=\op{bar}_{s_0}(\mu)$, we get 
% We now consider the terms $D_{z_i}\op{bar}_{s_0}(\mu)\of D_{s=s_0} z_i(s)$ for $1 \leq i \leq n$. Note that by \cref{eqn:L}, 
\begin{align*}
%\label{eqn:LC_bound}
    \norm{\left(D_{z_i}\op{bar}_{s_0}(\mu)-Q_{s_0}^{-1} \of w_i \lVert^{x_0}_{z_i}\right) \of D_{s=s_0} z_i(s)} \leq L(s_0) \norm{D_{z_i}\op{bar}_{s_0}(\mu)-Q_{s_0}^{-1}\of w_i \lVert^{x_0}_{z_i}}.
\end{align*}
We will now show that 
\begin{align*}
    \norm{D_{z_i}\op{bar}_{s_0}(\mu)-Q_{s_0}^{-1} \of w_i \lVert^{x_0}_{z_i}} \leq 16 w_i\max \set{a^2,b^2} r(s_0)^2.
\end{align*}
Indeed, the second part of Lemma \ref{lem:dbary_measure} gives the formula for $D_{z_i}\op{bar}_{s_0}(\mu)$ and we have:
\begin{align*}
\norm{\underbrace{-Q_{s_0}^{-1} \of w_i  \left(D_{z_i}d_{s_0}(x_0,z_i)\tensor\nabla_{x=x_0}d_{s_0}(x,z_i)
+d_{s_0}(x_0,z_i) D_{z_i}\nabla_{x=x_0} d_{s_0}(x,z_i)\right)}_\text{from \cref{lem:dbary_measure}} -Q_{s_0}^{-1} \of w_i (\lVert^{x_0}_{z_i})}.
\end{align*}
 Since $r(s_0) <\frac{1}{3\max\set{a,b}}$, \cref{cor:Q_pos_def} implies that $\norm{Q_{s_0}^{-1}} \leq  2$. Then the above expression is at most
\begin{align*}
    2 w_i \norm{\underbrace{D_{z_i}d_{s_0}(x_0,z_i)\tensor\nabla_{x=x_0}d_{s_0}(x,z_i)
+d_{s_0}(x_0,z_i) D_{z_i}\nabla_{x=x_0} d_{s_0}(x,z_i)}+(\lVert^{x_0}_{z_i})}.
\end{align*}
As $d_{s_0}(x_0,z_i)<2r(s_0)<\pi/2b$, \cref{lem:geom_lemma_2} implies that the term over the underbrace is equal to $-\lVert^{x_0}_{z_i}(I+S_i)$ where $S_i$ is a $(1,1)$-tensor such that 
\[\norm{S_i} \leq 2 \max\set{a^2,b^2} d_{s_0}(x_0,z_i)^2 \leq 8 \max \set{a^2,b^2} r(s_0)^2.
\]
Hence
\begin{align*}
    \norm{D_{z_i}\op{bar}_{s_0}(\mu)-Q_{s_0}^{-1} \of w_i \lVert^{x_0}_{z_i}}&\leq 2w_i \norm{\left(-\lVert^{x_0}_{z_i}(I+S_i)\right)+\lVert^{x_0}_{z_i}} \leq 16w_i \max\set{a^2,b^2}r(s_0)^2. 
\end{align*}
As $\sum_{i=1}^n w_i=1$ and $r(s_0) \leq 1$, 
\begin{align*}
    \norm{\sum_{i=1}^n\left(D_{z_i}\op{bar}_{s_0}(\mu)-Q_{s_0}^{-1} \of w_i \lVert^{x_0}_{z_i}\right) \of D_{s=s_0} z_i(s)} \leq  16 \max \set{a^2,b^2} r(s_0) L(s_0). 
\end{align*}
This establishes \cref{eqn:claim_for_zi} and finishes the proof of Claim \ref{claim2}.   \qed

Now Claims 1 and 2, taken together with \cref{eqn:ds_bary_mu_s}, finish the proof of  \cref{prop:dbary_limit}.
\end{proof}

Let $d^{\op{Gr}}$ be the Grassmannian distance. The following is straightforward from the definition of $d^{\op{Gr}}$.
\begin{lemma}
\label{lem:dist_gr_1}
Suppose $A,C: \Rb^m \to \Rb^p$ and $B: \Rb^p \to \Rb^p$ are linear maps and $W=\im(\id_{\Rb^m},A)$ and $W'=\im(\id_{\Rb^m},C+ B \circ A)$. Then $d^{\op{Gr}}(W,W') \leq \norm{C}+ \norm{B-I} \norm{A}$.
\end{lemma}

We now prove \cref{prop:parallel_close}.

\begin{corollary}
\label{prop:parallel_close_in_appendix}
Assume that \hyperref[link:setup-b]{\bf Assumptions B} hold. 
Suppose that there is a continuous foliation $\Hc$ on  $B(0,\epsilon_0)\times F$ such that:
\begin{enumerate}[label=(\alph*)]
    \item $\Hc$ has $C^1$ leaves which transversely vary continuously in the leafwise $C^1$-topology,
    \item for each $\xi \in F$, there is a $C^1$ function $\varphi_\xi:B(0,\epsilon_0)\to F$ satisfying $\varphi_\xi(0)=\xi$ and the foliation $\Hc$ is given by
    \[
    \Hc=\set{(s,\varphi_{\xi}(s)) : s \in B(0,\eps_0)}_{\xi \in F},
    \]
    \item the tangent distribution of $\Hc$ is denoted by $\Dc$.
\end{enumerate}

Fix $\xi_1,\dots,\xi_n \in F$, smooth weight functions $w_i: B(0,\eps_0) \to [0,1]$ and let $G$ denote the graph of the map 
\[ 
B(0,\epsilon_0) \ni s\mapsto \op{bar}_s(\mu(s))=\op{bar}_s \left( \sum_{i=1}^nw_i(s)\delta_{\varphi_{\xi_i}(s)} \right).
\]
Then for every $\epsilon>0$ and every $c\in [0,1)$, there exists $\delta = \delta(\eps,c)>0$ such that for every $s_0\in B(0,c\epsilon_0)$, if $r(s_o)<\delta$ then 
\[d^{Gr}(T_{(s_0,\op{bar}_{s_0}\mu(s_0))}G,\mc{D}(s_0,\op{bar}_{s_0}\mu(s_0)))<\epsilon.\]
Moreover, the derivatives  $D_{s=s_0}\op{bar}_s(\mu(s))$ and distributions $T_{(s_0,\op{bar}_{s_0}\mu(s_0))}G$ vary continuously in $s_0, \xi_1,\dots,\xi_n$.
\end{corollary}

\begin{proof}

Let $z_i(s):=\varphi_{\xi_i}(s)$ and $x_0:=\op{bar}_{s_0}(\mu(s_0))$. Now $$T_{(s_0,x_0)}G=\im (\id, D_{s=s_0}\op{bar}_{s}\mu(s)).$$

 We will first explain why the last moreover statement in the corollary follows directly from our previous results. Note that by the assumption on the foliation $\Hc$, $\varphi_{\xi_i}$ is $C^1$ in $s$ and $D_sz_i(s)=D_s \varphi_{\xi_i}(s)$ is continuous in both $\xi_i$ and $s$ for each $i=1,\dots, n$. 

By the formula for $D_{s=s_0} \op{bar}_{s}(\mu(s))$ given in \cref{eqn:ds_bary_mu_s} and \cref{eqn:chain_rule_Ds_w_and_z} above,
\[
D_{s=s_0} \op{bar}_s(\mu(s))  =D_{s=s_0}\op{bar}_s(\mu(s_0))+ \sum_{i=1}^{n}  D_{w_i}\op{bar}_{s_0}(\mu(s_0))\of D_{s=s_0} w_i(s) + \sum_{i=1}^n D_{z_i}\op{bar}_{s_0}(\mu(s_0))\of D_{s=s_0} z_i(s).
\]
The first term is continuous in $s_0$ and $\xi_i$ by \Cref{lem:dbary_metric} and \Cref{lem:reg_of_dist} (for the continuity of $Q_s$, see \Cref{eqn:Qs} and \Cref{lem:reg_of_dist}). The continuity of the second and the third term in $s_0$ and $\xi_i$ follows from \Cref{lem:dbary_measure} and the above mentioned regularity of $D_s \phi_{\xi_i}$. 

Thus $T_{(s_0,x_0)} G$ varies continuously in  $s_0, \xi_1,\dots,\xi_n$.

We now prove the remaining statements. For that, we can fix $\xi_1,\dots, \xi_n \in F$ and let $s$ vary in $B(0,\eps_0)$. 
By \cref{prop:dbary_limit}, 
\[
D_{s=s_0} \op{bar}_s(\mu(s))=R_{s_0}+ Q_{s_0}^{-1}  \left(\sum_{i=1}^n w_i(s_0)\lVert^{\op{bar}_{s_0}(\mu(s_0))}_{z_i(s_0)} D_{s=s_0} z_i(s) \right),
\] where $\norm{R_{s_0}}_{s_0} \leq 32 \max \set{a^2,b^2} L(s_0) r(s_0)$ for any $s_0 \in B(0,\eps_0)$.

Let 
\[
W(s_0):=\im(\id, \sum_{i=1}^n w_i(s_0)\lVert^{x_0}_{z_i(s_0)} D_{s=s_0} \varphi_{\xi_i}(s)).
\]
Note that $W(s)$ is continuous in $s$. Then $M':=\sup_{s_0 \in B(0,\eps_0)} \max_{1 \leq i \leq n} \norm{D_{s=s_0}\varphi_{\xi_i}(s)}$ exists. Then the above \cref{lem:dist_gr_1} implies that 
\begin{align}
\label{ineq:G_close_to_W}
    d^{Gr}(T_{(s_0,x_0)}G,W(s_0)) &\leq \norm{R_{s_0}}+ \norm{Q^{-1}_{s_0}-I} \norm{\sum_{i=1}^n w_i(s_0) D_{s=s_0}\varphi_{\xi_i}(s)}  \nonumber \\
    & \leq \norm{R_{s_0}}+ \norm{Q^{-1}_{s_0}-I} \left( \max_{1 \leq i \leq n} \norm{D_{s=s_0}\varphi_{\xi_i}(s)} \right) \nonumber \\
    & \leq 32 \max \set{a^2,b^2} L(s_0) r(s_0) + 2 \max \set{a^2,b^2} M' r(s_0)^2 \nonumber  \\
    & < M r(s_0),
\end{align}
for a constant $M$ that is independent of $s_0$, when we assume that $r(s_0)<1$.

We will now show that $d^{Gr}(W(s_0),\Dc(s_0,x_0))< \eps$ whenever $r(s_0)$ is sufficiently small, depending on $\eps$. Define
\begin{align*}
    \Dc_i(s,x)=\im(\id,\lVert^{x}_{z_i(s)}D_s\phi_{\xi_i}(s))
\end{align*}
where $\lVert^{x}_{z_i(s)}$ is the parallel transports along $\{\{s\}\times F:s\in B(0,\epsilon_0)\}$, with respect to the metrics $g_s$.

First note that $W(s_0)=\sum w_i(s_0) \Dc_i(s_0,x_0)$. Fix any $\eps>0$ and $c\in[0,1)$. We now claim that there exists $\delta'=\delta'(\eps,c)>0$ such that if $s_0 \in B(0,c\eps_0)$ and $r(s_0)<\delta'$, then 
\[
d^{Gr}(\Dc_i(s_0,x_0), \Dc(s_0,x_0)) <\eps/2. 
\]
Indeed, this follows from the uniform continuity of the distribution $\Dc$ and uniform continuity of the parallel translation maps $\lVert^{(s,x)}_{(s,y)}$ on $\overline{B(0,c\eps_0)} \times F$. More precisely, uniform continuity of $\Dc$ and the parallel translation implies the existence of $\delta'=\delta'(\eps,c)$ such that: if $s \in B(0,c\eps_0)$, $d_s(x,y)<\delta'$ and $v\in T_yF$ then 
\[
d^{Gr}(\Dc(s,x),\Dc(s,y))<\eps/2 \text{ and  } d^{TF}_s(v,\lVert^{(s,x)}_{(s,y)}v)< \eps/2
\] (here $d^{TF}$ is the distance in the tangent bundle of $F$ induced by $g_s$). Now observe that $(u,u') \in \Dc_i(s_0,x_0)$ if and only if $u'=\lVert^{x_0}_{z_i(s_0)} v'$ where $v'\in \im(D_{s=s_0}\varphi_{\xi_i}(s))$.  It follows that $(u,v') \in \Dc(s_0,z_i(s_0))$. Indeed we have 
\[\mc{D}(s_0,z_i(s_0))=\op{Span}\{(v,D_{s=s_0}\varphi_{\xi}(s)(v)):v\in \mathbb R^m\}=\im (\id, D_{s=s_0}\phi_{\xi_i}(s)).\]
This finishes the proof of the claim that if $r(s_0)<\delta'$, then $d(\Dc_i(s_0,x_0),\Dc(s_0,x_0))<\eps$.  

Now suppose $\delta' =\delta'(\eps,c)$ as above. We will now show that: if $r(s_0)<\delta'$, then 
\begin{align}
\label{ineq:distros_close}
 d^{Gr}(\sum_{i=1}^n w_i(s_0) \Dc_i(s_0,x_0), \Dc(s_0,x_0))<\eps.   
\end{align}

Indeed, if $v \in \Dc(s_0,x_0)$, then there exists $v_i \in \Dc_i$ such that $\norm{v-v_i} < \eps.$ Then $\sum w_i(s_0)v_i \in W(s_0)$ and $\norm{v- \sum w_i(s_0) v_i} \leq \sum w_i(s_0) \norm{v-v_i} < \eps$. Conversely, if $v_i' \in \Dc_i$, then there exist $y_i' \in \Dc$ such that $\norm{y_i-v_i'}<\eps$. Then $\sum w_iy_i \in \Dc$ and $\norm{\sum v_i' -\sum y_i} < \eps$. 

By \cref{ineq:G_close_to_W} and \cref{ineq:distros_close}, for any $\eps>0$ and $c\in[0,1)$, there exists $\delta=\min\set{\frac{\eps}{M},1, \delta'(\eps,c)}$ such that: if $s_0\in B(0,c\eps_0)$ and $r(s_0) <\delta$, then 
\[
d^{Gr}(T_{(s_0,x_0)}G, \Dc(s_0,x_0))< 2\eps.
\] \end{proof}

\section{Some Lemmas about Foliations}
\label{sec:foliation_lemmas}

A manifold $V$ is said to be {\em uniformly transverse}
to a foliation $\mathcal H$ of a Riemannian manifold if there exists $0<\theta\le \frac \pi 2$ such that at every  point  $p \in V$, the angle of $V$ with the leaf of $\mathcal H$ through $p$ is at least $\theta$.

\begin{lemma}
\label{lem:foliation_connected}
For $0\leq k\leq n$, consider the product foliation $\mathcal H=\{\mathbb R^k\times \{y\}:y\in \mathbb R^{n-k}\}$ of $\mathbb R^n$. Let $V$ be a complete connected $C^1$-submanifold of $\mathbb R^n$ such that $V$ is uniformly transverse to the foliation $\mathcal H$. Then the intersection of $V$ and any leaf of $\mathcal H$ is either empty or path connected.
\end{lemma}

We note that the  uniform transversality  assumption is required  as simple counterexamples show (e.g. the surface of revolution for the function  $e^x$). However, if $V$ has dimension $n-k$, the lemma holds under  just the transversality assumption. 

\begin{proof}
We first set up some preliminaries before starting the proof. Let $H$ be a leaf of the foliation $\Hc$ and define $f:\mathbb R^n\to [0,\infty)$ as $f(x):=d(x, H)^2$. Note that $\grad_xf$ is well-defined whenever $f(x)>0$. Then, if $x \in V$ and $f(x)>0$

Since $V$ is uniformly transverse to $\mathcal H$, there is $\theta>0$ such that for every $x\in V$ there is $v\in T_xV$ making  angle at least $\theta$ with distribution associated to $\mathcal H$. Thus, whenever $f>0$, the vector field $-\nabla f$ makes angles at most $\frac \pi 2 - \theta$ with tangent spaces of $V$. Let $\mathcal F$ be the vector field along $V$ obtained by projecting $-\nabla f$ to the tangent spaces of $V$ (whenever $f>0$). Then for every $x\in V$, $\norm{\mathcal F(x)}\ge \cos (\frac{\pi}{2}-\theta)\norm{-\nabla f(x)}=2\cos (\frac{\pi}{2}-\theta)d(x,H)$. Since $V$ is a $C^1$ manifold and $\nabla f$ is smooth (whenever $\grad f$ is well-defined, i.e. $f>0$), the vector field $\mathcal F$ is $C^1$ along $V$.

We let $F_t$ be the flow defined by $\mathcal F$. The flow $F_t$ is defined for all non-negative time and for every point $x\in V$ since $\norm{\mc{F}(x)}<2 d(x,H)$ and $d(x,H)$ decreases under the flow. For any $x\in V$, we let $\phi(t,x)=d(F_t(x),H)$. Then $\phi(0,x)=d(x,H)$. Whenever $\phi(t,x)>0$, $\phi(t,x)$ is differentiable  and 
\[
\frac{\partial}{\partial t}\phi(t,x)\le -2\cos^2(\frac \pi 2 -\theta) \phi(t,x).
\] 
Indeed, let $\theta_{y}$ be the angle between $\mc{F}(y)\in T_yV$ and $-\grad_y d(y,H)$. Then $\inner{\nabla_yd(y,H),\Fc(y)}= - \cos(\theta_y)\norm{\Fc(y)} \norm{\grad_yd(y,H)}=-2d(y,H) \cos^2 \theta_y $. Thus
\begin{align*}
\frac{\partial}{\partial t}\phi(t,x) & =D_{y=F_t(x)}d(y,H) \left( \frac{\partial}{\partial t}F_t(x)\right) = \langle \grad_{y}d(y,H), \Fc(y) \rangle |_{y=F_t(x)}= -2\cos^2(\theta_{F_t(x)}) d(F_t(x),H) \\
& \leq -2 \cos^2(\frac{\pi}{2}-\theta) \phi(t,x).
\end{align*}
It follows that $\phi(t,x)\le e^{-\epsilon_0t}d(x,H)$, where $\epsilon_0=2\cos^2(\frac \pi 2-\theta)$.

Now we begin the proof. Suppose, for the sake of contradiction, that $V \cap H$ is not connected. Let $x_1$ and $x_2$ be two points on two different connected components of $V\cap H$. We let $$r:=\inf_{\{\sigma :[0,1]\to V, \sigma (0)=x_1, \sigma (1)=x_2\}}\max_{t\in [0,1]}d(\sigma (t),H).$$

Now we consider two cases. 

\emph{ Case 1}: Suppose that $r>0$. Let $\sigma$ be a path such that the furthest point to $H$ of $\sigma$ is at distance at most $re^{\epsilon_0}$. Then consider the path $F_2(\sigma)$, obtained by flowing the path $\sigma$ for time 2 by the flow $F_t$. This path in $V$ is continuous and connects $x_1$ and $x_2$. The furthest distance of points on this path to $H$ is at most $re^{\epsilon_0}e^{-2\epsilon_0}<r$. This contradicts with the definition of $r$.

\emph{Case 2}:  Suppose $r=0$. We first show that for every $x\in V$, the limit $\lim_{t\to +\infty}F_t(x)$ exists. Indeed, for every $x\in V$, we let $\mathcal G(x)$ be the projection of $\mathcal F(x)$ to  the distribution associated to $\mathcal H$. Then $\mathcal G$ is a $C^1$ vector field, and $\norm{\mathcal G(x)}\le \frac 1 2 \norm{-\nabla f(x)}=d(x,H)$. As $F_t(x)$ gets closer to $H$ in forward time, the set $\{F_t(x):t\ge 0\}$ is contained in the closed $d(x,H)$-neighborhood of $H$. Furthermore, let $proj:\Rb^n\to H$ be the orthogonal projection. Then we get the estimate
\begin{align*}d(proj(F_{t_1}(x)), proj(F_{t_2}(x)))&\le \int_{t_1}^{t_2}\norm{\mathcal G(F_s(x))}ds\le \int_{t_1}^{t_2} \phi(s,x)ds\\ &\le \int_{t_1}^{t_2} d(x,H)e^{-\epsilon_0s}ds=\epsilon_0^{-1}d(x,H)(e^{-\epsilon_0t_1}-e^{-\epsilon_0t_2}).\end{align*}
Hence the set $\{F_t(x):t\ge 0\}$ is contained in the ball radius $2\epsilon_0^{-1}d(x,H)$ around $x$. It follows that  flowing the path $\sigma$ in forward time is trapped in a compact set. Furthermore, the inequality shows that the forward trajectory $\{F_t(x):t\ge 0\}$ is a Cauchy net. Therefore, $\lim_{t\to +\infty}F_t(x)$ exists.

 We also have the following estimate: for every $x\in V$ and $t \geq 0$, $$d(x,F_t(x))\le (1 +\epsilon_0^{-1})d(x,H).$$ Indeed $d(x,F_t(x)) \leq (\phi(0,x)-\phi(t,x)) + d(proj(x),proj(F_t(x))) \leq d(x,H)+\eps_0^{-1} d(x,H).$

Let $\sigma$ be a continuous path connecting $x_1$ and $x_2$. We claim that $\bar{\sigma}:[0,1]\to V$, defined by $\bar{\sigma}(s)=\lim_{t\to +\infty}F_t(\sigma(s))$ for every $s\in[0,1]$, is continuous. The map is well-defined because of the Cauchy property remarked above and since  $V$ is complete. Indeed, first we note that $\bar{\sigma}(s)\in H$ for every $s$. Let $s_0\in [0,1]$ and let $\epsilon>0$. Let $T_0 = T_0(s_0,\eps)>0$ be a time such that %$\epsilon_0^{-1}d(\sigma(s_0),H)e^{-\epsilon_0T_0}<\alpha \epsilon$, where $\alpha=\frac{1}{4+2\epsilon_0^{-1}}$. It follows that
$$d(F_{T_0}(\sigma(s_0)),\bar{\sigma}(s_0))\le  \alpha \epsilon,$$
where $\alpha=\frac{1}{4+2\epsilon_0^{-1}}$. Since the flow $F_t$ is continuous, the path $F_{T_0}(\sigma)$ is continuous. Thus there exists $\delta>0$ such that for $s\in [s_0-\delta,s_0+\delta]$, 
\[d(F_{T_0}(\sigma(s_0)),F_{T_0}(\sigma(s)))\le \alpha \epsilon.\]

It follows that $d(F_{T_0}(\sigma(s)), H)\le d(F_{T_0}(\sigma(s_0)),F_{T_0}(\sigma(s)))+d(F_{T_0}(\sigma(s_0)),\bar{\sigma}(s_0))\le  2\alpha \epsilon$. 
Hence, by the triangle inequality,

\begin{align*}d(\bar{\sigma}(s_0),\bar{\sigma}(s))&\le d(\bar{\sigma}(s_0), F_{T_0}(\sigma(s_0)))+ d(F_{T_0}(\sigma(s_0)),F_{T_0}(\sigma(s)))+ d(F_{T_0}(\sigma(s)),\bar{\sigma}(s))\\ &\le \alpha \epsilon + \alpha\epsilon+2\alpha\epsilon(1+\epsilon_0^{-1})\le\epsilon.\end{align*}
Hence, $\bar{\sigma}$ is continuous at $s_0$. This proves $\bar{\sigma}$ is a continuous path.

On the other hand, $\bar{\sigma}(s)\in H$ since $d(\bar{\sigma}(s),H)=0$, for every $s\in [0,1]$. Hence $\bar{\sigma}$ is a path in $V\cap H$, connecting $x_1$ and $x_2$. This is a contradiction.

Therefore, $V\cap H$ is connected.
\end{proof}

For two arbitrary vector subspaces $V,W$ of a common vector space $Z$ we define
\[
\angle(V,W)=\min\set{\sup_{w\in W}\inf_{v\in V}\angle(v,w), \sup_{v\in V}\inf_{w\in W}\angle(v,w)}
\]
In particular, the angle is $0$ precisely when either $V\subset W$ or $W\subset V$. When $\dim V=\dim W$ then $\angle(V,W)$ is bilipschitz with $d^{Gr}(V,W)$.  Indeed if $g\in O(\dim Z)$ is an  element of $O(\dim Z)$ closest to the identity that moves $V$ to $W$, then $d(e,g)$ is at least $\angle(V,W)$ and bounded by $\dim Z$ times this angle. 
\begin{lemma}\label{lem:pi_surj}
Let $M$ be an $m$-dimensional complete connected $C^1$ sub-manifold without boundary of $\mathbb R^n=\R^{n-k}\times \R^k$. We denote by $\mathcal V$ the foliation of $\mathbb R^n$ by vertical $\R^k$ spaces. Suppose that for every $p\in M$, we have $\angle (T_p M, T_p(\mathcal V(p))\le \epsilon$ for a fixed $\eps<\frac{\pi}{2}$. Let $\pi:\R^n\to \R^k$ be the projection of $\R^n=\R^{n-k}\times \R^k$ to the $\R^k$ factor. If $m\geq k$, then $\pi|_{M}$ is surjective onto $\Rb^k$, and in particular is a diffeomorphism when $m=k$.
\end{lemma}

\begin{proof}
Transversality implies that $\pi|_{M}$ is a submersion, and hence is an open map.

Next, we show that if $m\geq k$ then $\pi\rest{M}$ is surjective by showing that the image of $\pi\rest{M}$ is closed. Suppose, for the sake of contradiction, that the image is not closed. We let $z\in \pi(M)$ and let $r:[0,1)\to \pi(M)$ be a finite line segment contained in $\pi(M)$ from $z$ to some $w\in \partial \pi(M)$. Since $\pi\rest{M}$ is a submersion onto its image, $\pi\rest{M}^{-1}(r)$ is a $C^1$ submanifold of $M\subset\R^n$. For each $t\in [0,1)$, and each point $q\in \pi\rest{M}^{-1}(r(t))$, 
note that $\op{ker}d_q(\pi\rest{M})=T_q \pi\rest{M}^{-1}(r(t))$, and moreover $d_q(\pi\rest{M})$ is an isomorphism on the orthogonal complement of this subspace in $T_qM$  which we denote by $T_q^\perp \pi\rest{M}^{-1}(r(t))\subset T_qM$ . Hence there is a unique vector $v_q\in T_q^\perp \pi\rest{M}^{-1}(r(t))$  such that $d\pi(v_q)=r'(t)$. This gives a vector field on $\pi\rest{M}^{-1}(r)$.

We observe that for any $v\in T_qM$ we may write it as $v=v_\perp+v_{\op{ker}}$ for $v_\perp\in T_{q}^\perp \pi\rest{M}^{-1}(r(t))$ and $v_{\op{ker}}\in \op{ker}d_q\pi=(T_q\mc{V}(q))^{\perp}$. Let $w\in T_q\mc{V}(q)$. If $v_{\perp}=0$, then $\inner{v,w} = \inner{v_{\perp},w}=0$. So suppose $v_{\perp} \neq 0$. Since $\inner{v,w}=\inner{v_{\perp},w}$ and $\norm{v} \geq \norm{v_{\perp}}$, we have  $\inner{\frac{v}{\norm{v}},w}\leq \inner{\frac{v_\perp}{\norm{v_\perp}},w}$. Thus $\angle(v,w)\geq \angle(w,v_\perp)$ for any $w\in T_q\mc{V}(q)$ and $v \in T_qM$. Hence we have 
\[
\angle (T_{q}^\perp \pi\rest{M}^{-1}(r(t)),T_q\mc{V}(q))\leq\angle(T_qM,T_q\mc{V}(q))\leq \eps.
\]
Therefore $\norm{r'(t)}\leq \norm{v_q}\leq \frac{1}{\cos(\eps)}\norm{r'(t)}$ where $\eps<\frac{\pi}{2}$ was fixed.
Let us fix any point $q_0 \in \pi\rest{M}^{-1}(z)$ and let $\sigma:[0,1)\to M$ be the integral curve of the vector field $q\mapsto v_q$ starting at $q_0$. Then $\sigma$ projects to $r$. Moreover, $\sigma$ has finite length since the norm of $d\pi$ is bounded below on $v_q$ for each $q\in \pi\rest{M}^{-1}(r)$. By completeness of $M$, there is a limit point $\sigma(1)\in M$ which must map to the endpoint $w$ of $r$. This  contradicts the assumption that $w$ did not belong to $\pi(M)$.

Finally, we show that when $m=k$, $\pi|_{M}$ is a diffeomorphism onto its image. By \cref{lem:foliation_connected}, the intersection of $M$ and each horizontal $\mathbb R^{n-k}\times \{v\}$ is path connected for every $v\in\mathbb R^k$. But by transversality, each intersection is a zero dimensional manifold. Thus the intersection of $M$ and each horizontal $\mathbb R^{n-k}\times \{v\}$ consists of at most one point, and hence exactly one point, for every $v\in\mathbb R^k$. Therefore $\pi\rest{M}$ is a diffeomorphism onto its image.
\end{proof}
Finally, we need the following lemma about Hausdorff convergence of a sequence of transverse intersections of submanifolds.

\begin{lemma}\label{lem:continuity-of-transverse-intersection}
    Suppose $\set{M_n: n \geq 0}$ is a sequence of connected embedded $C^1$ submanifolds of a manifold $X_0$ such that $M_n \to M_0$ in the pointed Hausdorff topology and $TM_n \to TM_0$ pointwise. Suppose $N$ is another connected embedded $C^1$ submanifold of $X_0$ such that $M_n \cap N \neq \varnothing$ and the intersection is transverse for all $n \geq 0$. Then $M_n \cap N \to M_0\cap N$ in the pointed Hausdorff topology.
\end{lemma}
\begin{proof}
   Consider any point of $x\in M_0\cap N$.  Note that it is enough to show that for a pre-compact open subset $U_x \subset X_0$ containing $x$, $M_n \cap N \cap U_x \to M_0 \cap N \cap U_x$ in the Hausdorff topology. Let $m=\dim(M_0)$, $m+k=\dim(X_0)$, $r=\dim(N)$. Then $\dim(M_0 \cap N)=r-k \geq 0$. 
   
   Choose a coordinate chart $\phi:U_x\to (-1,1)^{m+k}\subset \R^{m+k}$ where $U_x\subset X_0$ is a pre-compact open set,  $\phi(M_0\cap U_x)=(-1,1)^m\times\set{0}^k$ and $\phi(N\cap U_x) = \set{0}^{m+k-r}\times (-1,1)^r$. We have to show that $\phi(M_n \cap U_x)\cap (\set{0}^{m+k-r}\times (-1,1)^r)$ converges to $\phi(M_0 \cap U_x)\cap (\set{0}^{m+k-r}\times (-1,1)^r)=\set{0}^{m+k-r}\times (-1,1)^{r-k}\times \set{0}^k$ in the Hausdorff topology. Note that since $\phi(M_n \cap U_x)$ converges to $\phi(M_0\cap U_x)$ in the pointed Hausdorff topology and $U_x$ is pre-compact. So, for any $0<\epsilon<1$ there is an $n$ large enough such that $\phi(M_n \cap U_x)$ is contained in $(-1,1)^{m}\times (-\epsilon,\epsilon)^k$.

Since $\phi$ is $C^1$ and the distributions $TM_n$ converge to $TM_0$ pointwise, $M_n \cap U_x$ is transverse to $\set{0}^{m+k-r}\times\R^r$ for $n$ large enough. Indeed, if $\set{y_n}$ is a sequence such that $y_n \in \phi(M_n \cap U_x)$ and $y_n \to y_0$, then $T_{y_n}\phi(M_n\cap U_x)$ converges to $T_y\phi(M_0\cap U_x)=\R^m\times\set{0}^k$. Thus for $n$ sufficiently large, $M_n \cap U_x$ is transverse to $\set{0}^{m+k-r}\times\R^r$. 

Now suppose that $n$ is large enough. As the image of $\phi(M_n\cap U_x)$ has tangent space nowhere tangent to ${0}\times \R^r$, each connected component of $\phi(M_n\cap U_x)$ is a graph of a function whose domain is a subset of $(-1,1)^m\times\set{0}^k$. We now claim that for sufficiently large $n$, each connected component of $\phi(M_n\cap U_x)$ is a graph of a function whose domain is $(-1,1)^m\times\set{0}^k$. Since $M_n$ is embedded in $X_F$, $\phi(M_n\cap U_x)$ is an embedded submanifold of $\phi(U_x)$. If the claim is false, then there is a  connected component $C_n$ of $\phi(M_n\cap U_x)$ and $y_n \in (-1,1)^m $ such that $(y_n,z_n) \in \partial C_n$. As $C_n$ is properly embedded in $\phi(U_x)$ and $\eps<1$, $(y_n,z_n)$ is an interior point of $C_n$ and $T_{(y_n,z_n)}C_n$ is transverse to $\set{0} \times \Rb^r$. This contradicts that $\phi(M_n \cap U_x) \subset (-1,1)^m \times \set{0}^k$.  Consequently, we have $\partial(\phi(M_n\cap U_x))\subset\partial((-1,1)^{m})\times [-\eps,\eps]^k$ and $\partial \pi_m(\phi(M_n) \cap U_x) \subset \partial ((-1,1)^m)$ (where $\pi_m$ is projection onto the first $m$ coordinates in $\Rb^{m+k}$). Therefore, each connected component of $\phi(M_n\cap U_x)$ is a graph of a function whose domain is entire $(-1,1)^m\times\set{0}^k$.

 We let $d_n=d^{\Haus}(\phi(M_n \cap N \cap U_x), \phi(M_0 \cap N \cap U_x))$. We then have $d_n\to 0$ when $n\to \infty$. For each connected component of $\phi(M_n\cap U_x)$, there is a function $\theta_n:(-1,1)^m\to (-1,1)^k$ such that the connected component is precisely the graph $\{(x,\theta_n(x)):x\in (-1,1)^m\}$. It follows that the intersection of this connected component with $N$ is the set $S=\{(x,\theta_n(x)):x\in \set{0}^{m+k-r}\times (-1,1)^{r-k}\}$. It is clear that 
\[d^{\Haus}(S,\set{0}^{m+k-r}\times (-1,1)^{r-k}\times \set{0}^k)\le d_n,\]
which tends to $0$ when $n\to \infty$. This shows that $\phi(M_n\cap N\cap U_x)$ converges to $\phi(M_0\cap N\cap U_x)$ in the Hausdorff topology.
\end{proof}

\providecommand{\bysame}{\leavevmode\hbox to3em{\hrulefill}\thinspace}
\providecommand{\MR}{\relax\ifhmode\unskip\space\fi MR }
\providecommand{\MRhref}[2]{%
  \href{http://www.ams.org/mathscinet-getitem?mr=#1}{#2}
}
\providecommand{\href}[2]{#2}

\bibliography{hyprankloc.bib}{}
\bibliographystyle{alpha}

\end{document}